\documentclass{article}
\usepackage{natbib}

\usepackage[english]{babel}
\usepackage{authblk}

\title{Computable Bounds for Strong Approximations\\ with Applications}
\author{Haoyu Ye}
\author{Morgane Austern}
\affil{Department of Statistics, Harvard University}
\date{}
\usepackage[letterpaper,top=2cm,bottom=2cm,left=3cm,right=3cm,marginparwidth=1.75cm]{geometry}

\usepackage[ruled,vlined]{algorithm2e}
\usepackage{float}
\usepackage{amsmath}   

\allowdisplaybreaks  
\usepackage{amssymb}
\usepackage{mathtools}
\usepackage{graphicx}

\usepackage[colorlinks=true, allcolors=blue]{hyperref}
\usepackage{amsthm}
\usepackage{cleveref}

\newtheorem{assumption}{Assumption}
\newtheorem{theorem}{Theorem}[section]
\newtheorem{corollary}{Corollary}[theorem]
\newtheorem{lemma}[theorem]{Lemma}

\newtheorem{proposition}[theorem]{Proposition}
\newtheorem{remark}[theorem]{Remark} 

\crefname{lemma}{Lemma}{Lemmas}

\usepackage{algpseudocode}

\usepackage[dvipsnames]{xcolor}

\usepackage{algpseudocode}

\usepackage{forest}

\usepackage{listings}
\usepackage{xcolor}

\definecolor{codegreen}{rgb}{0,0.6,0}
\definecolor{codegray}{rgb}{0.5,0.5,0.5}
\definecolor{codepurple}{rgb}{0.58,0,0.82}
\definecolor{backcolour}{rgb}{0.95,0.95,0.92}

\lstdefinestyle{mystyle}{
    backgroundcolor=\color{backcolour},   
    commentstyle=\color{codegreen},
    keywordstyle=\color{magenta},
    numberstyle=\tiny\color{codegray},
    stringstyle=\color{codepurple},
    basicstyle=\ttfamily\footnotesize,
    breakatwhitespace=false,         
    breaklines=true,                 
    captionpos=b,                    
    keepspaces=true,                 
    numbers=left,                    
    numbersep=5pt,                  
    showspaces=false,                
    showstringspaces=false,
    showtabs=false,                  
    tabsize=2
}

\newcommand{\sig}{\sigma}

\usepackage{mathtools, stmaryrd}
\usepackage{xparse}

\usepackage{mathtools, stmaryrd}
\usepackage{xparse} \DeclarePairedDelimiterX{\Iintv}[1]{\llbracket}{\rrbracket}{\iintvargs{#1}}
\NewDocumentCommand{\iintvargs}{>{\SplitArgument{1}{,}}m}
{\iintvargsaux#1} %
\NewDocumentCommand{\iintvargsaux}{mm}
  {#1\IfValueT{#2}{, #2}}

\def\balign#1\ealign{\begin{align}#1\end{align}}
\def\baligns#1\ealigns{\begin{align*}#1\end{align*}}
\def\balignat#1\ealign{\begin{alignat}#1\end{alignat}}
\def\balignats#1\ealigns{\begin{alignat*}#1\end{alignat*}}
\def\bitemize#1\eitemize{\begin{itemize}#1\end{itemize}}
\def\benumerate#1\eenumerate{\begin{enumerate}#1\end{enumerate}}

\newenvironment{talign*}
 {\csname align*\endcsname}
 {\endalign}
\newenvironment{talign}
 {\csname align\endcsname}
 {\endalign}

\def\balignst#1\ealignst{\begin{talign*}#1\end{talign*}}
\def\balignt#1\ealignt{\begin{talign}#1\end{talign}}

\newcommand{\qtext}[1]{\quad\text{#1}\quad} 
 
\newcommand{\stext}[1]{\ \text{#1}\ }

\let\originalleft\left
\let\originalright\right
\renewcommand{\left}{\mathopen{}\mathclose\bgroup\originalleft}
\renewcommand{\right}{\aftergroup\egroup\originalright}

\def\tinycitep*#1{{\tiny\citep*{#1}}}
\def\tinycitealt*#1{{\tiny\citealt*{#1}}}
\def\tinycite*#1{{\tiny\cite*{#1}}}
\def\smallcitep*#1{{\scriptsize\citep*{#1}}}
\def\smallcitealt*#1{{\scriptsize\citealt*{#1}}}
\def\smallcite*#1{{\scriptsize\cite*{#1}}}

\def\<{\left\langle} 
\def\>{\right\rangle}

\def\defeq{\triangleq} 

\newcommand{\Ber}{\textnormal{Bernoulli}}

\newcommand{\Bin}{\textnormal{Binomial}}

\newcommand{\Unif}{\textnormal{Uniform}}

\newcommand{\eqdist}{\stackrel{d}{=}}

\newcommand{\iid}{\textrm{i.i.d.}\@\xspace}

\providecommand{\arccos}{\mathop\mathrm{arccos}}

\ifdefined\nonewproofenvironments\else
\ifdefined\ispres\else
 \newtheorem{definition}{Definition}

\fi

\newcommand{\wass}[1][p]{\mathcal{W}_{#1}}

\newcommand{\mnkappa}{M_{n,\kappa}}

\newcommand{\Rsig}[1][\sigma]{\hyperref[Rsig]{\color{black}{R_{#1}}}} 

\newcommand{\wassbd}[1][\sigma]{\hyperref[cabris]{\color{black}{\omega^R_p(#1)}}} 
\newcommand{\KRsig}[1][\sigma]{\hyperref[KRsig]{\color{black}{K_{R,\sig}}}} 
\newcommand{\Rsigsqd}[1][\sigma]{\hyperref[Rsig]{\color{black}{R_{#1}^2}}} 
\newcommand{\Rsigpowfour}[1][\sigma]{\hyperref[Rsig]{\color{black}{R_{#1}^4}}}

\usepackage{nicefrac}

\usepackage{mdframed} 
\newmdenv[topline=false,rightline=false,bottomline=false,nobreak=false]{proofaside}

\usepackage{pifont}

\newtheorem*{assumption*}{Assumption}

\newcommand{\assumptionRS}{Assumption~($R,\sigma$)}

\newcounter{noteHctr} 

\newcounter{noteMctr} 

\DeclareUnicodeCharacter{2032}{\ensuremath{^{\prime}}}
\begin{document}
\maketitle
\begin{abstract}
The Koml{\'o}s--Major--Tusn\'{a}dy (KMT) inequality for partial sums is one of the most celebrated results in probability theory. Yet its practical application has been hindered by a lack of practical constants. This paper addresses this limitation for bounded i.i.d. random variables.  
At the cost of an additional logarithmic factor, we propose a computable version of the KMT inequality that depends only on the variables' range and standard deviation. We also derive an empirical version of the inequality that achieves nominal coverage even when the standard deviation is unknown. We then demonstrate the practicality of our bounds through applications to online change point detection and first hitting time probabilities. As a byproduct of our analysis, we obtain a Cramér-type moderate deviation bound for normalized centered partial sums.
\end{abstract}


\section{Introduction}\label{sec:intro}
Strong approximations---couplings between the partial sums of independent and identically distributed (i.i.d.) random variables and a Brownian motion---are a cornerstone of modern probability and statistics. The celebrated Koml{\'o}s--Major--Tusn\'{a}dy (KMT) inequality \cite{komlos1975approximation,komlos1976approximation} provides the canonical result in this area, demonstrating that the coupling can be achieved with a maximal deviation that grows only logarithmically with the sample size.  Specifically, the KMT inequality states that if $(Y_i)$ is a sequence of  i.i.d. random variables satisfying $\inf_{a>0}\mathbb{E}[e^{a|Y_1|}]<\infty$, then there exist positive constants $K_1, K_2<
\infty$ and a sequence of i.i.d. Gaussian random variables $(G_i)\overset{i.i.d}{\sim}\mathcal{N}\big(\mathbb{E}[Y_1],\rm{Var}(Y_1)\big)$ such that for all $n\in \mathbb{N}$ and $\alpha\in (0,1)$:
$$\mathbb{P}\Big(\max_{i\le n}\big|\sum_{j\le i}Y_j-\sum_{j\le i}G_j\big|\ge K_1\log n -K_2\log \alpha \Big)\le \alpha.$$ 
This result has been substantially extended to handle more complex settings, including multivariate processes, dependent random variables, and weighted approximations (e.g., \cite{sakhanenko1989accuracy, zaitsev2003estimates, berkes2014komlos, mies2023sequential}). We refer the reader to \cite{mason2012quantile} for a comprehensive survey. It has become a powerful tool for establishing asymptotical theory in diverse fields such as equivalence of statistical experiments \cite{nussbaum1996asymptotic}, change-point detection \cite{frick2014multiscale,chen2022inference}, queuing theory \cite{nussbaum1996asymptotic},  first-passage times for Markov Chains (e.g., \cite{portnoy1978probability,lai1979first}), finance \cite{kifer2013dynkin,lamberton2000optimal,walsh2003rate} and sequential inference (e.g., \cite{waudby2023distribution}). 

However, a significant practical limitation of the KMT inequality is the lack of computable constants: the constants $K_1$ and $K_2$ are not explicitly known and depend on the distribution of $(Y_i)$ in a complex way. Consequently, while the KMT inequality is a powerful tool for studying the asymptotics of statistical methods, it does not provide finite-sample statistical guarantees. 
Obtaining such guarantees requires a \emph{computable} bound on the maximum deviation between $\sum_{j\le i}Y_j$ and $ \sum_{j\le i}G_j$. In other words, the constants $K_1$ and $K_2$ need to be known. Very recently, \cite{waudby2025nonasymptotic} established that these constants can be uniformly upper-bounded for all distributions sharing a uniform lower bound on their Sakhanenko parameter. Note, however, they remain uncomputable, as the bounds provided in \cite{waudby2025nonasymptotic} rely on some unknown universal constants.
In some specific cases, the constants can be inferred.  A notable result in this direction is \cite{chatterjee2012new} which proposed a Stein-based method for strong embeddings. Their result, however, relies on the existence of a bounded Stein kernel, which is a strong assumption. By adapting this method, Bhattacharjee et al.\ \cite{Bhattacharjee16} extended the approach to random variables that take values in a finite alphabet and have a null third moment. However, the bounds proposed in \cite{Bhattacharjee16} grow quadratically with the size of the alphabet, and the assumption that the support of $(Y_i)$ is finite is itself restrictive. 
Building on this line of work, we develop a new proof technique based on Stein's method that generalizes the discrete framework of \cite{Bhattacharjee16} to arbitrary bounded random variables. 

In this paper, we derive practical \emph{computable} thresholds $(\mathcal{D}_k(\alpha))_{k\le n}$  for the KMT inequality under the  assumption that the i.i.d. random variables $(Y_i)$ are bounded, i.e., $R:=\|Y_1\|_\infty < \infty$. For any confidence level $\alpha \in (0,1)$, these thresholds (see \Cref{ts122} and \Cref{thm:main2}) satisfy 
$$\mathbb{P}\Big(\exists i\le n~\textrm{s.t.}~\big|\sum_{j\le i}Y_j-\sum_{j\le i}G_j\big|>\mathcal{D}_i(\alpha)\Big)\le \alpha.$$
Crucially, these bounds can be computed using only the range $R$ and the standard deviation $\sigma:=\sqrt{\rm{Var}(Y_1)}$.
The price of this computability, however, is a slightly suboptimal rate of convergence: the thresholds grow as $\max_{k\le n}\mathcal{D}_k(\alpha)=O\big(\log n (\log n -\log \alpha )\big)$, which includes an additional logarithmic factor compared to the original KMT inequality. Similarly to us, \cite{castelle1998strong} derived a computable KMT inequality growing at the same suboptimal rate in $n$ as ours. However, their bound is unreasonably large for any moderate sample sizes (see \Cref{section:compare_clb_bg}) which significantly hinders its applicability. The computable bounds obtained in \cite{Bhattacharjee16} not only require the random variables to take value in a finite alphabet but are also significantly overconservative (see \Cref{section:compare_clb_bg}). Hence our bound constitutes the first practical computable KMT inequality for bounded random variables. 
Finally, since the standard deviation $\sigma$ is often unknown in practice, we also provide in \Cref{empi} an empirical version of our KMT inequality that does not require prior knowledge of $\sigma$. Our proof relies on two key components: a novel inductive construction and a computable conditional Wasserstein-$p$ bound. In addition, as a byproduct of the conditional Wasserstein-$p$ bound, we obtain a Cramér-type moderate deviation result.
See \Cref{wass_sec} for more details.

We demonstrate the utility of our bounds in two distinct settings. First, we use them to derive thresholds for the CUSUM statistic in online change point detection. The proposed thresholds guarantee control of the false alarm probability uniformly over time. Our method achieves a high detection rate at relatively small sample sizes compared to existing approaches. Second, we provide non-asymptotic bounds on first-hitting-time probabilities for random walks with small drift. We apply our result to the setting considered in \cite{busani2020bounds}, which developed a novel proof technique because the incomputable constants of the KMT inequality hindered a direct analysis (see page 2 of \cite{busani2020bounds}). Our method yields a similar convergence rate but with significantly smaller constants. 
In both applications, our bounds are non-asymptotic and fully computable.
\subsection{Outline of the paper }
The remainder of this paper is organized as follows. \Cref{main} presents our main results together with the algorithm and theoretical guarantees. \Cref{empi} provides our empirical bounds. \Cref{applications} demonstrates two applications. In \Cref{wass_sec}, we show our computable upper bound for the conditional Wasserstein-$p$ distance, followed by a resulting Cramér-type moderate deviation result. Finally, \Cref{discussion} discusses the proof technique and the limitations of our results. 


\section{Main results}\label{main}
\subsection{Assumptions and notations}We first state the assumptions under which we derive our results.
\begin{assumption*}[Assumption $(R,\sigma)$]\label{DA}
The random variables $(Y_i)_{i\ge1}$ are independent and identically distributed (i.i.d.), with $\mathrm{Var}(Y_1) = \sigma^2$ and $Y_1 \in [0,R]$ almost surely. Define $X_i:=Y_i-\mathbb{E}[Y_i]$ as the centered version of $Y_i$.
\end{assumption*}
Throughout the paper, we denote $\mathcal{U}(X_{1:n})$ the unordered multiset $\{X_1, X_2, \dots, X_n\}$ and write $$S_n:=\sum_{i\le n}X_i\qquad \textrm{and}\qquad W_k:=S_k-\frac{k}{n}S_n=S_k-\mathbb{E}[S_k|S_n].$$
We adopt the standard convention that $S_0 = 0$ and $W_0 = 0$. Let $\Sigma_n\in \mathcal{M}_{n}(\mathbb{R})$ denote the covariance matrix of $(S_k)_{k\le n}$, i.e., $$\Sigma_{n,i,j}=\sigma^2 \min\{i,j\},\qquad \forall i,j\le n.$$

Additionally, $(U_i)_{i\ge1}$ denotes a sequence of i.i.d. $\Unif[0,1]$ random variables independent of $(Y_i)_{i\ge 1}$. These are used to construct couplings between potentially discrete random variables and continuous ones.

We will use the Wasserstein-$p$ distance to measure the discrepancy between two probability distributions.

\begin{definition}[Wasserstein Distance]
For two probability distributions $\mu$ and $\nu$ on the real line, the Wasserstein-$p$ distance is defined as
\begin{equation*}
    \mathcal{W}_p(\mu, \nu) := \left( \inf_{\gamma \in \Gamma(\mu,\nu)} \int |x - y|^p \, d\gamma(x, y) \right)^{1/p},
\end{equation*}
where $\Gamma(\mu,\nu)$ is the set of all couplings with marginals $\mu$ and $\nu$.
\end{definition} By abuse of notation, we sometimes write $\mathcal{W}_p(Z,Z')$ to denote the distance between the distributions of $Z$ and $Z'$, or $\mathcal{W}_p(Z,\nu)$ to denote the distance between the distribution of $Z$ and $\nu$.

We will also use the following conditional Wasserstein-$p$ distance.
\begin{definition}[Conditional Wasserstein Distance]Let $Z,Z'$ be two random variables on the same probability space as $(X_i)$.
We define the conditional Wasserstein-$p$ distance as
\begin{equation*}
    \mathcal{W}_p(Z,Z'|\mathcal{U}(X_{1:n})) := \left( \inf_{\gamma \in \Gamma(Z, Z' \mid \mathcal{U}(X_{1:n}))} \int |x - y|^p \, d\gamma(x, y) \right)^{1/p},
\end{equation*}
where $\Gamma(Z, Z' \mid \mathcal{U}(X_{1:n}))$ is the set of all couplings whose marginals are the conditional distributions of $Z,Z'$ given $\mathcal{U}(X_{1:n})$.
\end{definition}
Similarly, by abuse of notation, we will sometimes write
$\mathcal{W}_p(Z,\nu|\mathcal{U}(X_{1:n}))$ to denote the distance between the conditional distribution of $Z$ knowing $\mathcal{U}(X_{1:n})$ and $\nu$.



\subsection{Inductive step and proof overview}
In this subsection, we outline an inductive approach to constructing a Gaussian process coupling for $(S_k)_{1\le k \le n}$. More details can be found in \Cref{additional,section:KMTproof,prelim}.

Our goal is to approximate the sum $(S_k)_{k\le n}$ with a carefully constructed process $(Z_k)_{k\le n}$, where $ Z_k:= \sum_{i \leq k} G_i$, with $(G_i)_{i\ge1} \overset{\text{i.i.d.}}{\sim} \mathcal{N}(0, \sigma^2)$. Specifically, given $\alpha > 0$, we aim to construct $(Z_k)_{k \le n} \sim \mathcal{N}(0, \Sigma_n)$ such that there exists a sequence $(\delta_k)_{k\le n}$ satisfying $\max_{k\le n}\delta_k=O((\log n)^2 )$ and  
\[
\mathbb{P}\big(\exists k \leq n \text{ s.t. } |S_k - Z_k| \geq \delta_k\big) \leq \alpha.
\]

To achieve this, we construct $(Z_i)_{1\le i \le n}$ inductively, starting from $Z_n$. This is done by exploiting \Cref{lemma:bound_S_and_Z}, which allows one to translate a  Wasserstein distance bound into a probability bound via Markov's inequality and optimization over $p$. Specifically, according to \Cref{nelson3}, we have $\mathcal{W}_p(S_n, \mathcal{N}(0, n \sigma^2)) \le s^R_p(n, \sigma)$, where $s^R_p(n, \sigma)$ is a computable quantity defined in \Cref{nelson3}. 
By \Cref{lemma:bound_S_and_Z}, it follows that for any $\alpha_1 \leq \alpha$, there exists $Z_n \sim \mathcal{N}(0, n \sigma^2)$ such that, defining  
\[
\delta^\star := \inf_{\substack{p \in \mathbb{N}, \\ p \geq 2}} \frac{s^R_p(n, \sigma)}{\alpha_1^{1/p}},
\]  
we have  
\[
\mathbb{P}\big(|Z_n - S_n| \geq \delta^\star\big)\leq \alpha_0.
\]  
 Interestingly, $Z_n$ can be chosen to be measurable with respect to $\sigma(S_n,U_n)$, where $U_n$ is an independent $\Unif[0,1]$ random variable (see \Cref{lemma:PITdiscrete}).
 
 Once we have chosen $Z_n$, it remains to build $(Z_k)_{k\le n}$ in a way that ensures it has the correct joint distribution and that the difference $|Z_k-S_k|$ is small. As $Z_n\in \sigma(U_n,S_n)$, if $(\tilde Z_k)$ are normally distributed conditionally on $S_n$ and $U_n$, meaning that $(\tilde Z_k)\big|S_n,U_n\sim \mathcal{N}(0,\tilde\Sigma_n)$ where $\tilde \Sigma_{n,i,j}=\sigma^2(\min\{i,j\}-\frac{ij}{n})$, then the process $\big(\tilde Z_j+\frac{j}{n}Z_n\big)_{j\le n}$ is a centered Gaussian vector with covariance $\Sigma_n$ and is hence a candidate for $(Z_k)_{k\le n}$. The next step therefore consists of carefully building $(\tilde Z_k)_{k\le n}$ to satisfy:

\begin{assumption}[Conditionally coupled on $\mathcal{I} \subset\mathbb{N}$]\label{coupled} Let $a, b \in \mathbb{N}$ with $a \le b$, and define $\mathcal{I} := \Iintv*{a, b}$.
 Let $(\tilde Z_k)_{k\in \mathcal{I}}$ be a sequence of random variables. We say that $(\tilde Z_k)_{k\in \mathcal{I}}$ is conditionally coupled with $(X_i)_{i\in \mathcal{I}}$ on $\mathcal{I}$ if it is measurable with respect to $\sigma((X_i)_{i\in \mathcal{I}},(U_i)_{i\in \mathcal{I}})$ and its conditional marginal distribution satisfies $$(\tilde Z_k)_{k\in \mathcal{I}}\Big|\sum_{i\in\mathcal{I}}X_i,U_{{b}}\sim \mathcal{N}(0,\mathrm{Var}(W_{\mathcal{I}})),$$ where $W_{\mathcal{I}}$ designates the vector with $k$-th coordinate $W_{\mathcal{I},k}=\sum_{\ell\le k,\ell\in \mathcal{I}}X_\ell-\frac{|\Iintv{k}\cap \mathcal{I}|}{|\mathcal{I}|}\sum_{\ell\in \mathcal{I}}X_\ell$, for $k\in \mathcal{I}$.
\end{assumption}

 In the next lemma, we show how $(\tilde Z_k)_{k\le n}$ and $(\delta_k)_{k\le n}$ can be built recursively.


\begin{lemma}\label{ausecours}
    Let $(Y_i)_{i\ge 1}$ be generated according to \assumptionRS. Let $(\delta_k)_{k\le n}$ be a sequence of positive reals satisfying $\delta_k=\delta_{n-k}$ for all $k>n/2$. Suppose that $n$ is even and that one can construct $(\tilde Z_k^1)$ satisfying \Cref{coupled} with $\mathcal{I}=\Iintv*{1,\frac{n}{2}}$. For all $k\le n/2$, denote $W_k^1:=S_k-\frac{2k}{n}S_{n/2}$. Then there exists $(\tilde Z_k)$ satisfying \Cref{coupled} with  $\mathcal{I}=\Iintv*{1,n}$ such that the following holds
    \begin{align*}&
   \mathbb{P}\big(\exists k\le n~\textrm{s.t}~|W_k-\tilde Z_k|\ge \delta_k\big)\\\le& 2   \mathbb{P}\big(\exists k\le n/2~\textrm{s.t}~|W^1_k-\tilde Z^1_k|\ge \delta_k-\frac{2k}{n}\delta_{n/2}\big) -\mathbb{P}\big(\exists k\le n/2~\textrm{s.t}~|W^1_k-\tilde Z^1_k|\ge \delta_k-\frac{2k}{n}\delta_{n/2}\big)^2
   \\&+\inf_{p\ge 2}\Big(\frac{ \omega_p^R(n, \sigma)}{\delta_{n/2}}\Big)^p,
     \end{align*} where $ \omega_p^R(n, \sigma)$ is defined in \Cref{thm:main1}.
\end{lemma}
\begin{proof}
Firstly, we denote $$\tilde{W}^1_k = S_{k}-\frac{2k}{n}S_{n/2}\quad\text{for }k=1,\dots, n/2,$$ and $$\tilde{W}^2_k = S_{k}-S_{n/2}-\frac{2(k-n/2)}{n}(S_{n}-S_{n/2})\quad\text{for }k=n/2,\dots, n.$$
Let $(\tilde Z^1_i)_{i\le n/2}$ be a Gaussian vector satisfying \Cref{coupled} with $\mathcal{I}=\Iintv*{1,\frac{n}{2}}$. Note that as $(X_i)_{i>\frac{n}{2}}$ have the same distribution as $(X_i)_{i\le \frac{n}{2}}$, we have that there also exists $(\tilde Z^2_i)_{i>\frac{n}{2}}$ that satisfy \Cref{coupled} with $\mathcal{I}=\Iintv*{n/2+1,n}$ and is such that 
\begin{align*}
  \Big((\tilde W_k^1)_{k\le n/2},(\tilde Z_k^1)_{k\le n/2} \Big)\overset{d}{=} -\Big((\tilde W_{n-k}^2)_{k\le n/2},(\tilde Z_{n-k}^2)_{k< n/2} \Big).
\end{align*}  
We remark that conditionally on $\mathcal{U}(X_{1:n})$, $W_{n/2}$ is a discrete random variable. Let $F_{n/2}(\cdot|\mathcal{U}(X_{1:n}))$ designates its conditional CDF and $p_{n/2}(\cdot|\mathcal{U}(X_{1:n}))$ its conditional p.m.f.  We define  \begin{equation}\label{uc}
\tilde F^{U_{n/2}}_{n/2}(x) =F_{n/2}(x^-|\mathcal{U}(X_{1:n})) + U_{n/2} \cdot p_{n/2}(x|\mathcal{U}(X_{1:n})).
\end{equation} and set $$\tilde Z_{n/2}:=\frac{\sqrt{n}\sigma}{2}\Phi^{-1}(\tilde F^{U_{n/2}}_{n/2}(W_{n/2})).$$Using \Cref{lemma:PITdiscrete}, we know that $$\tilde Z_{n/2}|U_n,S_n\sim \mathcal{N}(0,\frac{n\sigma^2}{4}).$$ We will build $(\tilde Z_k)$ by linear interpolation between $(\tilde Z_k^1),(\tilde Z_k^2)$, and $\tilde Z_{n/2}$. More precisely, we define
\begin{equation}\label{eqn:def_Z_k_ell}
        \tilde Z_k: = 
        \begin{cases}
            \tilde Z^{1}_k + \dfrac{2k}{n}\tilde Z_{n/2}\quad & \text{if } k< n/2, \\
            \tilde Z_{n/2} \quad & \text{if } k= n/2, \\
            \tilde{Z}^2_{k}+ \dfrac{2n-2k}{n}\tilde Z_{n/2} \quad & \text{if } k> n/2.
        \end{cases}
\end{equation} 
We first remark that, since $(\tilde Z_k^1)$ and $(\tilde Z_k^2)$ are chosen to satisfy \Cref{coupled} on $\mathcal{I}=\Iintv*{1,\frac{n}{2}}$ and $\Iintv*{n/2+1,n}$ respectively, we obtain that, conditionally on $U_n,U_{n/2},S_{n/2},\text{ and }S_n-S_{n/2}$, the random vector $(\tilde Z_k)$ is Gaussian with mean $\Big(\frac{2}{n}\tilde Z_{n/2},\dots,\tilde Z_{n/2},\dots, \frac{2}{n}\tilde Z_{n/2}\Big)^\top$ and covariance \[
\begin{bmatrix}
\mathrm{Var}(W_{1:\frac{n}{2}}) &\mathbf{0}\\
\mathbf{0} & \mathrm{Var}(W_{1:n/2})
\end{bmatrix}.
\]  Hence, using the fact that Gaussian distributions are conjugate priors, we obtain that $(\tilde Z_k)$ satisfies \Cref{coupled} for $\mathcal{I}=\Iintv*{1,n}$. Hence, to prove \Cref{ausecours}, it remains to bound $|\tilde Z_k-W_k|$ carefully. In this goal, we note, by a union bound argument, that
\begin{align*}
    &\mathbb{P}(\exists k\le n \textrm{ s.t. } |W_k-\tilde Z_k|\ge \delta_k)\\
    \le& \mathbb{P}\Big(|W_{n/2}-\tilde Z_{n/2}|\ge \delta_{n/2} \Big)+\mathbb{P}\Big( |W_{n/2}-\tilde Z_{n/2}|< \delta_{n/2} \textrm{ and }\exists k\le n~|W_k-\tilde Z_k|\ge \delta_k \Big)\\
    =: & \mathbb{P}\Big(|W_{n/2}-\tilde Z_{n/2}|\ge \delta_{n/2} \Big)+(D).
\end{align*} 
We bound each term successively. Firstly, using Markov's inequality, we obtain that for all $p\in \mathbb{N}$, we have 
\begin{align*}
    \mathbb{P}\Big(|W_{n/2}-\tilde Z_{n/2}|\ge \delta_{n/2}\Big)&\le \frac{\|W_{n/2}-\tilde Z_{n/2}\|_p^p}{\delta_{n/2}^p}.
\end{align*}Moreover, according to \Cref{boy_breaks}, we have that $$\|W_{n/2}-\tilde Z_{n/2}\|_p=\Big\|\mathcal{W}_p\big(W_{n/2},\mathcal{N}(0,\frac{\sigma^2n}{4})\big|\mathcal{U}(X_{1:n})\big)\big\|_p.$$ This directly implies that \begin{align*}
    \mathbb{P}\Big(|W_{n/2}-\tilde Z_{n/2}|\ge \delta_{n/2}\Big)&\le \inf_{p\ge 2}\frac{\Big\|\mathcal{W}_p\big(W_{n/2},\mathcal{N}(0,\frac{\sigma^2n}{4})\big|\mathcal{U}(X_{1:n})\big)\Big\|_p^p}{\delta_{n/2}^p}.
\end{align*} According to \Cref{thm:main1} we know that for all $p\ge 2$, 
\begin{align}\label{eqn:computable_upper_bound_intro}
   \Big\|\mathcal{W}_p\big(W_{n/2},\mathcal{N}(0,\frac{\sigma^2n}{4})\big|\mathcal{U}(X_{1:n})\big)\Big\|_p\le \omega_p^R(n,\sigma).
\end{align}
The definition of $\omega_p^R(n,\sigma)$ can be found in \Cref{thm:main1}. Therefore, we have \begin{align*}
    \mathbb{P}\Big(|W_{n/2}-\tilde Z_{n/2}|\ge \delta_{n/2}\Big)&\le \inf_{p\ge 2}\Big\{\frac{\omega_p^R(n,\sigma)}{\delta_{n/2}}\Big\}^p,
\end{align*}
Then to bound $(D)$, suppose that $|W_{n/2}-\tilde Z_{n/2}|<\delta_{n/2}$ but that there exists $ k\le n$ such that $|W_k-\tilde Z_k|\ge\delta_k.$ If $k<\frac{n}{2}$, then we have\begin{align*}
\delta_k\ge& \big|W_k-\tilde Z_k\big|\\
=&\Big| W_k^1-\tilde Z^1_k+\frac{k}{n/2}(W_{n/2}-\tilde Z_{n/2})\Big|
\\\le& \Big| W_k^1-\tilde Z^1_k\Big|+\frac{k}{n/2}\Big|W_{n/2}-\tilde Z_{n/2}\Big|
\\\le&  \Big| W_k^1-\tilde Z^1_k\Big|+\frac{k}{n/2}\delta_{n/2},
\end{align*}which implies that $$\big|W_k^1-\tilde Z^1_k\big|\ge \delta_k-\frac{k}{n/2}\delta_{n/2}.$$
Moreover, similarly if $k>n/2$, then since 
$$W_k-\tilde Z_k-\frac{n-k}{n/2}(W_{n/2}-\tilde Z_{n/2})=  W_{k-n/2}^2-\tilde  Z^{2}_{k-n/2},$$ we have $$\big|W_k^2-\tilde Z^2_k\big|\ge \delta_k-\frac{n-k}{n/2}\delta_{n/2}.$$ Hence this implies that 
\begin{align*}
(D)\le& \mathbb{P}\Big(\exists k< n/2,| W_k^1 -\tilde Z_k^1|\ge \delta_k -\frac{k}{n/2}\delta_{n/2}\textrm{ or }~ \exists k> n/2, |W_k^2 -\tilde Z_k^{2}|\ge \delta_k-\frac{n-k}{n/2}\delta_{n/2}\Big).
\end{align*}
We remark that $(\tilde Z_k^1)$ is, by assumption, measurable with respect to $X_{1:n/2}$ and $U_{1:n/2}$. Similarly, by assumption, we have $(\tilde Z_k^2)$ is measurable with respect to $X_{n/2+1:n}$ and $U_{n/2+1:n}$.
 Hence,  we obtain that \begin{align*}&
     \mathbb{P}\Big(\exists k<n/2~\textrm{s.t}~\big|W_k^1-\tilde Z^1_k\big|\ge \delta_k-\frac{k}{n/2}\delta_{n/2}~\textrm{or}~\exists k>n/2 \textrm{~s.t~}|W_k^2 -\tilde Z_k^{2}|\ge \delta_k-\frac{n-k}{n/2}\delta_{n/2}\Big)
     \\\overset{(a)}{=}&\mathbb{E}\Big[\mathbb{I}\Big(\exists k<n/2~\textrm{s.t,}~\big|W_k^1-\tilde Z^1_k\big|\ge \delta_k-\frac{k}{n/2}\delta_{n/2}\Big)\\&\qquad\mathbb{P}\Big(\exists k>n/2 \textrm{~s.t.~}|W_k^2 -\tilde Z_k^{2}|\ge \delta_k-\frac{n-k}{n/2}\delta_{n/2}\Big|X_{1:n/2},U_{1:n/2}\Big)\Big]
     \\\overset{(b)}{=}&\mathbb{E}\Big[\mathbb{I}\Big(\exists k<n/2~\textrm{s.t,}~\big|W_k^1-\tilde Z^1_k\big|\ge \delta_k-\frac{k}{n/2}\delta_{n/2}\Big)\\&\times\mathbb{P}\Big(\exists k>n/2 \textrm{~s.t.~}|W_k^2 -\tilde Z_k^{2}|\ge \delta_k-\frac{n-k}{n/2}\delta_{n/2}\Big)\Big]
     \\\overset{}{=}&\mathbb{P}\Big(\exists k<n/2~\textrm{s.t.}~\big|W_k^1-\tilde Z^1_k\big|\ge \delta_k-\frac{k}{n/2}\delta_{n/2}\Big)\\&\times\mathbb{P}\Big(\exists k>n/2 \textrm{~s.t.~}|W_k^2 -\tilde Z_k^{2}|\ge \delta_k-\frac{n-k}{n/2}\delta_{n/2}\Big),
\end{align*} where (a) comes from the tower rule, (b) from the fact that $(W_k^2,\tilde Z_k^2)$ are independent of $X_{1:n/2}$ and $U_{1:n/2}$. 
Hence, as for any events $A,B$, we have $\mathbb{P}(A\cup B)=\mathbb{P}(A)+\mathbb{P}(B)-\mathbb{P}(A\cap B)$, we obtain that 
\begin{align*}
    (D)&\le \mathbb{P}\Big(\exists k<n/2~\textrm{s.t}~\big|W_k^1-\tilde Z^1_k\big|\ge \delta_k-\frac{k}{n/2}\delta_{n/2}\Big)\\&+\mathbb{P}\Big(\exists k>n/2 \textrm{~s.t~}|W_k^2 -\tilde Z_k^{2}|\ge \delta_k-\frac{n-k}{n/2}\delta_{n/2}\Big)
    \\&-\mathbb{P}\Big(\exists k<n/2~\textrm{s.t}~\big|W_k^1-\tilde Z^1_k\big|\ge \delta_k-\frac{k}{n/2}\delta_{n/2}\Big)\\&\qquad\times\mathbb{P}\Big(\exists k>n/2 \textrm{~s.t~}|W_k^2 -\tilde Z_k^{2}|\ge \delta_k-\frac{n-k}{n/2}\delta_{n/2}\Big).
 \end{align*} 
Finally, we remark that
\begin{align*}
  \Big((\tilde W_k^1)_{k\le n/2},(\tilde Z_k^1)_{k\le n/2} \Big)\overset{d}{=} -\Big((\tilde W_{n-k}^2)_{k\le n/2},(\tilde Z_{n-k}^2)_{k< n/2} \Big).
\end{align*}  Hence, as $\delta_{n-k}=\delta_k$, we obtain that\begin{align*}
    &\mathbb{P}\Big(\exists k>n/2 \textrm{~s.t~}|W_k^2 -\tilde Z_k^{2}|\ge \delta_k-\frac{n-k}{n/2}\delta_{n/2}\Big)\\&=\mathbb{P}\Big(\exists k<n/2~\textrm{s.t}~\big|W_k^1-\tilde Z^1_k\big|\ge \delta_{k}-\frac{k}{n/2}\delta_{n/2}\Big).
\end{align*}
This directly implies that \begin{align*}
    (D)&\le2 \mathbb{P}\Big(\exists k<n/2~\textrm{s.t}~\big|W_k^1-\tilde Z^1_k\big|\ge \delta_k-\frac{k}{n/2}\delta_{n/2}\Big)
    \\&-\mathbb{P}\Big(\exists k<n/2~\textrm{s.t}~\big|W_k^1-\tilde Z^1_k\big|\ge \delta_k-\frac{k}{n/2}\delta_{n/2}\Big)^2.
\end{align*} 
\end{proof}

Lemma \ref{ausecours} tells us that if we choose  
\[
\delta_{n/2} \geq \inf_{p \geq 2} \alpha_0^{1/p} \omega_p^R(n,\sigma),
\]  
then, as long as we can construct $(\tilde{Z}_k^1)$ to satisfy \Cref{coupled} with $\mathcal{I} = \Iintv*{1, n/2}$, there exists $(\tilde{Z}_k)$ that satisfies \Cref{coupled} with $\mathcal{I} = \Iintv*{1, n}$ such that  
\begin{align*}
  \mathbb{P}\big(\exists k \leq n \textrm{ s.t. } |W_k - \tilde{Z}_k| \geq \delta_k\big) 
  &\leq 2 \mathbb{P}\big(\exists k \leq n/2 \textrm{ s.t. } |W^1_k - \tilde{Z}^1_k| \geq \delta_k - \frac{k}{n/2} \delta_{n/2}\big) \\
  &\quad - \mathbb{P}\big(\exists k \leq n/2 \textrm{ s.t. } |W^1_k - \tilde{Z}^1_k| \geq \delta_k - \frac{k}{n/2} \delta_{n/2}\big)^2 \\
  &\quad + \alpha_0.
\end{align*}

Hence, it remains to choose $\delta_k$ for all $k < n/2$. However, we note that this step involves only half of the data, namely $X_1, \dots, X_{n/2}$, suggesting that constructing $(\delta_k)$ can be done inductively.

\subsection{Algorithm and theoretical guarantees}\label{theoretical_guarantees}
In this subsection, we exploit \Cref{ausecours} to introduce our main algorithm, which allows us to couple both $(W_k)$ with Gaussian vectors, and hence to obtain a coupling for $(S_k)$. To this end, for all $\ell \in \mathbb{N}$, we denote  
\[
W_k^\ell = S_k - \frac{k}{2^{\ell}} S_{2^\ell}\qquad \forall k\le 2^\ell.
\]  
Moreover, we set $\delta_1^0:=0$. For all $\ell \geq 1$ and $\alpha>0$, define $\delta_{2^{\ell-1}}^\ell$ as:  
\[
\delta_{2^{\ell-1}}^\ell(\alpha) := \inf_{p \geq 2} \alpha^{1/p} \omega_p^R(n, \sigma),
\]  
where $\omega_p^R(n, \sigma)$ is the computable upper bound introduced in \Cref{eqn:computable_upper_bound_intro} and defined in \Cref{thm:main1}. For all other $k \leq 2^\ell$, define $\delta_k^\ell(\alpha)$ recursively as follows:  
\begin{align}\label{autre}
\delta_k^\ell(\alpha) :=
\begin{cases}
    \delta_k^{\ell-1}(\alpha) + \frac{k}{2^{\ell-1}} \delta_{2^{\ell-1}}^{\ell}(\alpha), & \text{if } k \leq 2^{\ell-1}, \\ 
    \delta_{2^\ell-k}^{\ell-1} (\alpha)+ \frac{2^\ell - k}{2^{\ell-1}} \delta_{2^{\ell-1}}^{\ell}(\alpha), & \text{if } k > 2^{\ell-1}.
\end{cases}
\end{align}  

With this definition, we can establish the following result.
\begin{theorem}\label{ts11}
    Suppose that $(Y_i)_{i\ge 1}$ satisfies \assumptionRS. Let $\alpha_0>0$ be a constant and $n\in \mathbb{N}^\star$. Let $(\delta_k)$ be defined as $\delta_k:=\delta_k^L(\alpha_0)$ where $L=\lceil\log_2 n \rceil$. Then there exists $(\tilde Z_k)$ satisfying \Cref{coupled} with $\mathcal{I}=\Iintv{1,2^L}$ such that the following holds $$\mathbb{P}\Big(\exists k\le n~\textrm{s.t.}~|W_k-\tilde Z_k|> \delta_k\Big)\le \beta_L,$$ where  $\beta_L$ is defined recursively by $\beta_0 := 0$ and $\beta_\ell := 2\beta_{\ell - 1} - \beta_{\ell - 1}^2 + \alpha_0$ for all $\ell \ge 1$.
\end{theorem}The proof can be found in \Cref{triste}. 
 This can be further used to couple $(S_k)$ with a Gaussian vector and obtain the following result:
 
\begin{corollary}\label{all_too_well}
    Suppose that \Cref{ts11} holds with some $\alpha_0>0$ and  $n=2^L$ for $L \in \mathbb{N}$. Let $\alpha_1>0$ be a constant, and define $\delta^\star(\alpha_1):=\inf_{p\ge 2}\alpha_1^{1/p}s_p^R(n,\sigma)$, where $s_p^R(n,\sigma)$ is defined in \Cref{nelson3}. Then there exists a centered Gaussian vector $(Z_k)_{1\le k\le n}$ with the same covariance structure as $(S_k)_{1\le k\le n}$ such that
    \begin{align}
        \mathbb{P}\Big(\exists k\le n~\textrm{s.t.}~|S_k- Z_k|\ge \delta_k+\frac{k}{n}\delta^*(\alpha_1)\Big)\le \beta_L+\alpha_1,
    \end{align} where $\delta_k := \delta_k^L(\alpha_0)$ and $\beta_L$ is defined recursively by $\beta_0 := 0$ and $\beta_\ell := 2\beta_{\ell-1} - \beta_{\ell-1}^2 + \alpha_0$ for all $\ell \ge 1$.
\end{corollary} The proof can be found in \Cref{triste}. We note that, for a fixed $\alpha_0$, $\beta_L$ increases toward $1$ as $L=\log_2(n)$ increases. Therefore, to guarantee a given coverage level, $\alpha_0$ must depend on the sample size $n$. For a given $n$, we choose $\alpha_0$ algorithmically to obtain the desired coverage level. This is implemented in \Cref{ts122} for a specified coverage $\alpha>0$.

\begin{algorithm}[h]
\caption{ Computing $(\Delta_k(\alpha))_{k\le n}$}\label{alg:main_delta_computation}\label{ts122}
\DontPrintSemicolon
\KwIn{An integer \(n >0\), $R>0$, $\sigma>0$, total probability budget $\alpha > 0$  }
\KwOut{Thresholds $(\Delta_k(\alpha))_{k\le n}$}

\textbf{Define:} \(L:=\lceil \log_2 n\rceil\)\;

\tcp{Step 1: Determine $\nu_0$ based on the chosen $\alpha$}
\textbf{Initialize for $\nu_0$ search:} Choose an initial guess for \( \nu_0 \) and set $\beta_0(\nu_0)=0$\;

\textbf{Compute $\beta_k$ recursively:}

\For{$k = 1$ \KwTo $L$}{
    $\beta_k(\nu_0) \gets 2\beta_{k-1}(\nu_0) - \beta_{k-1}^2(\nu_0) + \nu_0$\;
}
\textbf{Finalize $\nu_0$:} Find the largest \( \nu_0^* \) such that \( \beta_L(\nu_0^*) \le \alpha\)\;

\tcp{Step 2: Calculate intermediate $\delta^M_k$ values using $\nu_0^*$}
$\delta_1^0 \gets 0$\;
\For{$M = 1$ \KwTo $L$}{
    $\delta_{2^{M-1}}^M\gets  \min_{2 \le p }\frac{\omega_p^R(2^M,\sigma)}{(\nu_0^*)^{1/p}}$\;
    \For{$k = 1$ \KwTo $2^{M-1}$}{
        \(\delta_k^M \gets \delta_k^{M-1} + \frac{k}{2^{M-1}} \delta_{2^{M-1}}^M\)\;
    }
    \For{$k = 2^{M-1}+1$ \KwTo $2^M-1$}{
        \(\delta_k^M \gets \delta_{2^{M}-k}^{M-1} + \frac{2^{M}-k}{2^{M-1}} \delta_{2^{M-1}}^M\)\;
    }
}

Let $\delta^L_k$ denote the final values $\delta_k^L$ from this step.


\For{$k = 1$ \KwTo $n$}{
    $\Delta_k(\alpha) \gets \delta^L_k$\;
}
\Return{$\Delta_k(\alpha)$}
\end{algorithm}


\begin{algorithm}
\caption{ Computing $(\mathcal{D}_k(\alpha))_{k\le n}$}\label{ts12}
\DontPrintSemicolon
\KwIn{An integer \(n>0\), $R>0$, $\sigma>0$, total probability budget $\alpha > 0$}
\KwOut{Thresholds $(\mathcal{D}_k(\alpha))$ minimized by optimal choice of $\alpha_0$}

\textbf{Define:} \(L:=\lceil \log_2 n\rceil\)\;

\tcp{Step 1: For any candidate split $\alpha_0 \in (0, \alpha)$, compute thresholds}
Given $\alpha_0 \in (0, \alpha)$ and $\alpha_1 := \alpha - \alpha_0$: 

\Indp
    \tcp{Step 1a: Compute $\Delta_k$ using $\alpha_0$}
    \For{$k = 1$ \KwTo $n$}{
        $\Delta_k(\alpha_0) \gets \text{output of \cref{ts122} with input } n \text{ and } \alpha_0$\;
    }

    \tcp{Step 1b: Compute $\delta^\star$ using $\alpha_1$}
    $\delta^\star(\alpha_1) \gets \min_{2 \le p}\frac{s_p^R(n,\sigma)}{(\alpha_1)^{1/p}}$\;

    \tcp{Step 1c: Compute $\mathcal{D}_k(\alpha)$, values to be minimized}
    \For{$k = 1$ \KwTo $n$}{
        $\mathcal{D}_k(\alpha) \gets \Delta_k(\alpha_0) + \frac{k}{n}\delta^\star(\alpha_1)$\;
    }
\Indm
\tcp{Step 2: Optimize the split}
Find $\alpha_0^* \in (0, \alpha)$ that minimizes $\max_{k \le n} \mathcal{D}_k(\alpha)$. Let $\alpha_1^* := \alpha- \alpha_0^*$

\For{$k = 1$ \KwTo $n$}{
$\mathcal{D}_k(\alpha) \gets \Delta_k(\alpha^*_0) + \frac{k}{n}\delta^\star(\alpha_1^*)$\;
}
\tcp{Step 3: Return optimized thresholds}
\Return{$(\mathcal{D}_k(\alpha))_{k \le n}$}
\end{algorithm}






\begin{theorem}\label{thm:main3}Suppose that $(Y_i)_{i\ge 1}$ satisfy \assumptionRS. 
Let $\alpha>0$ and $n\in\mathbb{N}^\star$. Define $
(\Delta_k(\alpha))$ as the output of \Cref{ts122}.
Then there exists a centered Gaussian vector $\tilde Z:=(\tilde Z_k)_{0\le k\le n}$ with the same covariance structure as $(W_k)_{0\le k\le n}$ such that the following inequality holds:
\begin{equation*}
    \mathbb{P}(\exists k\le n\textrm{ s.t. }\big|W_k-\tilde {Z}_k\big|\ge \Delta_k(\alpha))\le\alpha.
\end{equation*}
\end{theorem}The proof can be found in \Cref{triste}. This leads to the following coupling lemma between $(S_k)$ and a Gaussian random vector. 
\begin{theorem}\label{thm:main2}Suppose that $(Y_i)_{i\ge 1}$ satisfy \assumptionRS. 
Let $\alpha>0$ and $L\in\mathbb{N}$. Set $n=2^L$. Define $
(\mathcal{D}_k(\alpha))$ as the output of \Cref{ts12}.
Then there exists a centered Gaussian vector $(Z_k)_{k\le n}$ with covariance $\Sigma_n$ such that the following inequality holds:
\begin{equation*}
    \mathbb{P}(\exists k\le n\textrm{ s.t. }\big|S_k-{Z}_k\big|\ge \mathcal{D}_k(\alpha))\le\alpha.
\end{equation*}
\end{theorem}
The proof can be found in \Cref{triste}. Note that the covariance matrix $\Sigma_n$ is the same as the covariance structure of $(Z_k)_{k \le n}$.
\begin{remark}\label{turing}
   In practice, we use a slightly modified version of \Cref{ts12} and \Cref{ts122} that, for very small $\ell$, allows the use of  alternative bounds other than the conditional Wasserstein-$p$ bound to control $\delta_{2^{\ell-1}}^\ell(\alpha)$. Moreover, we allow the confidence level to vary with $\ell$ instead of using a fixed $\alpha_0$. These changes do not affect the theoretical guarantees, but can offer slight improvements in practice. The corresponding algorithms details are presented in \Cref{turing_details}. In addition, for implementation, optimizations over $p$ are carried out over $p \in \mathbb{N}$ with $2\le p \le p_{\mathrm{max}}$ for some sufficiently large $p_{\mathrm{max}}$, which does not affect the theoretical guarantees.

\end{remark}

\subsection{Analysis of the rate}

The size of the thresholds $(\mathcal{D}_k(\alpha))_{k\le n}$ proposed in \Cref{thm:main2} depends on the sample size $n\in \mathbb{N}^\star$.  In this subsection, we show how fast they grow with $n$ and compare it to the rate given by the classical KMT inequality.

\begin{lemma}\label{lemma:Delta_order}
Let $\alpha>0$ and $n\in\mathbb{N}^\star$. Let $L:=\lceil \log_2 n\rceil$, and let$\big(\mathcal{D}_k(\alpha))_{k\le n}$ be defined as the output of \Cref{ts12}. Then there exists a constant $ \kappa_R^\sigma$, depending only on $R$ and $\sigma$, such that the following inequality holds:
$$\max_{\ell \le n}\mathcal{D}_\ell(\alpha)\le  L\kappa_R^\sigma(L-\log \alpha ).
$$
Moreover, for all $0<\sigma_1<\sigma_2\le \frac{R}{2}$, the following holds: $$\sup_{\sigma\in [\sigma_1,\sigma_2]}\kappa_R^\sigma<\infty.$$
\end{lemma} The proof can be found in \Cref{lemma:Wp_order}.
The KMT inequality \cite{komlos1975approximation,komlos1976approximation} tells us that there exist constants $K_1,K_2<\infty$, depending on the distribution of $(Y_i)$, such that for all $n\in\mathbb{N}^\star$ and $\alpha>0$, there exists a Gaussian vector $(Z_k)$ satisfying $$\mathbb{P}\Big(\max_{k\le n}|S_k-Z_k|\ge K_1\log n-K_2\log\alpha\Big)\le \alpha.$$ We note that both $\max_{k\le n} \mathcal{D}_k(\alpha)$, where $(\mathcal{D}_k(\alpha))_{k\le n}$ is the output from \Cref{ts12}, and $K_1\log n-K_2\log \alpha$ grow linearly in $-\log \alpha $.  
However, $\max_{\ell\le n}\mathcal{D}_\ell(\alpha)$ grows quadratically, rather than linearly, in $\log n $, which is a suboptimal rate. On the other hand, our thresholds are fully computable, in contrast to almost all previous bounds, and depend only on $R$ and $\sigma$. Among the few existing computable bounds in the literature, one depending only on $R$ can be derived from results on uniform empirical processes in \cite{castelle1998strong}, but this bound also has quadratic growth in $\log n$. Moreover, while \cite{Bhattacharjee16} provides a computable bound, it requires that the random variables $(Y_i)$ take values from a finite alphabet and satisfy $\mathbb{E}[(Y_i-\mathbb{E}(Y_1))^3]=0$. In addition, their bound grows quadratically with the alphabet size. In \Cref{section:compare_clb_bg}, we compare our bounds to these existing computable bounds and show that, for various $R$ and even for small alphabet sizes, our bounds are significantly tighter for any reasonable sample size. See \Cref{discussion} for comments on how this rate could be improved if the random variables $(Y_i)$ admit a Stein kernel.

\subsection{Comparison to existing results}\label{section:compare_clb_bg}
We now compare the practical performance of our bound with the alternative results in \cite{Bhattacharjee16} and \cite{ castelle1998strong}, both discussed in the introduction and the previous subsection. For clarity, we state the bound we derive from \cite{Bhattacharjee16} and \cite{ castelle1998strong} below. 

\begin{theorem}[Corollary derived from Bhattacharjee and Goldstein (BG)\cite{Bhattacharjee16}]\label{thm:comparison1}Suppose that $(Y_i)_{i\ge 1}$ satisfy \assumptionRS ~and assumptions of \Cref{lemma:finte_alphabet_result}. 
Let $\alpha>0$ and $n\in\mathbb{N}$. 
Then there exists a centered Gaussian vector $(Z_k)_{k\le n}$ with covariance $\Sigma_n$ such that the following inequality holds:
\begin{equation}\label{eqn:comparison1}
\mathbb{P} \left( \max_{0 \le k \le n} |S_k - Z_k| \ge \lambda^{-1} \left( A \log n - \log \alpha + \log A \right) \right) \le \alpha,
\end{equation}
where $\lambda$ and $A$ are constants specified in \Cref{lemma:finte_alphabet_result}. 
\end{theorem}
The proof can be found in \Cref{section:compare}. Note that although this bound has a better asymptotic rate of $O(\log n)$ compared to our $O((\log n)^2)$ rate, it is substantially larger for any moderate values of $n$, as shown in \Cref{fig:comparison_clb}. 

\begin{theorem}[Corollary derived from Castelle and Laurent-Bonvalot (C\&LB)\cite{castelle1998strong}]\label{thm:comparison2}Suppose that $(Y_i)_{i\ge 1}$ satisfy \assumptionRS. 
Let $\alpha>0$ and $n\in\mathbb{N}$. 
Then there exists a centered Gaussian vector $(Z_k)_{k\le n}$ with covariance $\Sigma_n$ such that the following inequality holds:
\begin{equation}\label{eqn:comparison2}
\mathbb{P} \left( \max_{0 \le k \le n} |S_k - Z_k| \ge R\Big(60\log n+30\log(0.67/\alpha)\Big)\log n\right) \le \alpha.
\end{equation}
\end{theorem}

The proofs can be found in \Cref{section:compare}.

We compare our bound $
(\mathcal{D}_k(\alpha))$ in \Cref{thm:main2} with the two bounds above. 
The comparison in Figures~\ref{fig:comparison_clb} shows that our bound provides a considerably tighter constant than those in \cite{castelle1998strong} and \cite{Bhattacharjee16}, which do not aim at optimizing constants. Note that the $x$-axis corresponds to $\log_2 n$ where $n$ is the sample size, and the $y$-axis corresponds to $\log$ of the bounds.

\begin{figure}[h]
\centering
\includegraphics[width=\textwidth]{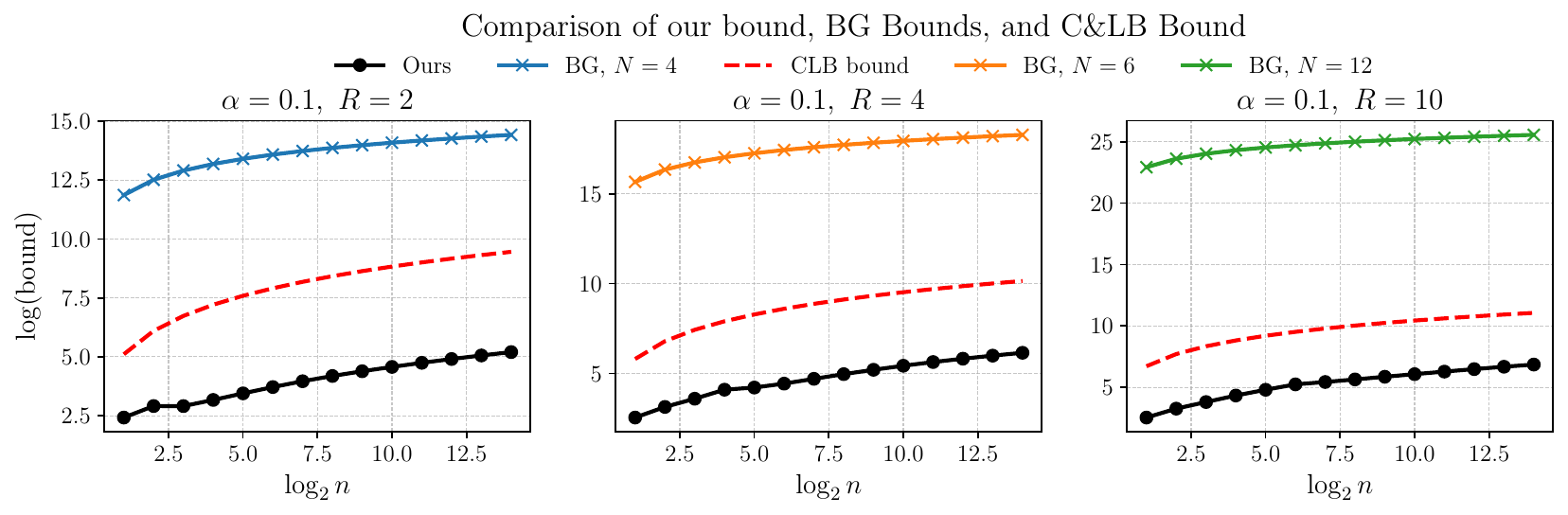}
\caption{Comparison between our bound and bounds in \cref{eqn:comparison1}, \cref{eqn:comparison2}.
Each panel shows $\alpha=0.1$ for sample sizes $n = 2^L$ with $L \in \{1, \dots, 14\}$ with different ranges $R=2,4,10$. 
The bound \cref{eqn:comparison1} (x-marked lines), derived from~\cite{Bhattacharjee16}, are for  random variables $X_i$'s sampled i.i.d. from the symmetric finite set $\mathcal{A}_N := \{-R/2, -R/2 + R/(N-1), \dots, R/2\}$ for some $N > R$.  We choose the distribution so that $X_i$ has mean zero, unit variance, and zero third moment. 
The dashed red line corresponds to the bound in \cref{eqn:comparison2}, derived from \cite{castelle1998strong}.
The solid black line shows our bound.
}
\label{fig:comparison_clb}
\end{figure}

\section{Empirical bound:  extension to unknown variance}\label{empi}
In practice, the variance $\sigma$ is often unknown and must be estimated from the data. This can in particular be done by using the empirical variance $$\hat\sigma^2_k:= \frac{1}{k}\sum_{i\le k}(X_i-\frac{S_k}{k}) ^2$$ as an estimator of $\sigma$. As our KMT inequality relies on knowledge of $\sigma$, it is tempting to use $\hat\sigma_k$ as a proxy for $\sigma$ and, for example, to use \Cref{ts12} with variance input $\hat\sigma_k$. However, $\hat\sigma_k$ is a random estimator of the variance. Hence, to obtain a correct inequality, the uncertainty in our estimate of $\sigma$ must be taken into account. This can, for example, be done using anytime-valid confidence sequences for $(\hat\sigma_k)_{k\ge 1}$. 

For clarity throughout this section, we make the dependence on $R$ and $\sigma$ explicit by writing $$(\Delta_k(\alpha,R,\sigma))_{k\le n}\quad \text{and} \quad(D_k(\alpha,R,\sigma))_{k\le n}$$ for the outputs of \Cref{ts12} and \Cref{ts122}, respectively.


We consider estimators $\hat\sigma_k$ that satisfy the following conditions.
\begin{assumption}[Empirical variance with confidence level $\delta$]\label{avril}
     Let $(\hat\sigma_k)_{k\ge1}$ be a sequence of estimators of the variance $\sigma$, where each $\hat\sigma_k\in \sigma(Y_1,\dots,Y_k)$ depends only on the first $k$ observations.  Let $([\hat\sigma_k^L,\hat\sigma_k^U])$ be a sequence of intervals such that $0\le \hat\sigma_k^L\le \hat\sigma_k^U\le \frac{R}{2}$, with $\hat\sigma_k^L, \hat\sigma_k^U\in \sigma(Y_1,\dots,Y_k)$ for all $k$. Assume this sequence satisfies $$\mathbb{P}\Big(\exists k~\textrm{s.t}~\sigma\not\in[\hat\sigma_k^L,\hat\sigma_k^U] \Big)\le \delta.$$ 
\end{assumption}
Note that if we chose $\hat\sigma_k^2:=\frac{c_1+\sum_{i\le k-1}(Y_i-\hat\mu_i)^2}{k}$ and $\hat\mu_i:=\frac{c_2+\sum_{i\le i-1}Y_i}{i}$, where $c_1,c_2\in [0,1]$ are constants, then \cite{martinez2025sharp} provides a way to build  intervals $([\hat\sigma_k^L,\hat\sigma_k^U])$ that  satisfy \Cref{avril}.
\begin{proposition}\label{empirical} Let $\alpha, \rho\in (0,1)$. 
  Let  $(\hat\sigma_k)_{k \ge 1}$ be a sequence of estimators, and let $([\hat\sigma_k^L,\hat\sigma_k^U])_{k \ge 1}$ be a sequence of intervals satisfying \Cref{avril} with confidence level $\delta=\rho\alpha$. Define $\tilde R_k:=\frac{R}{{\hat\sigma_k^L}}$. 
  Then there exist centered Gaussian vectors $(\tilde Z_k)$ and $(Z_k)$ with the same covariance structure as $(W_k)$ and $(S_k)$ respectively, such that 
  \begin{align}&
     \mathbb{P}\Big(\exists k\le n~\textrm{s.t}~|W_k-\tilde Z_k|\ge \hat\sigma_k^U~\Delta_k(\alpha(1-\rho),\tilde R_k,1)\Big)\le \alpha,
 \\&
     \mathbb{P}\Big(\exists k\le n~\textrm{s.t}~|S_k-Z_k|\ge \hat\sigma_k^U~\mathcal{D}_k(\alpha(1-\rho),\tilde R_k,1)\Big)\le \alpha.
  \end{align}
\end{proposition} The proof can be found in \Cref{empirical_proof}.
\begin{remark}
    For any time step $k$, the thresholds $ \hat\sigma_k^U~\Delta_k(\alpha(1-\rho),\tilde R_k,1)$ and $ \hat\sigma_k^U~\mathcal{D}_k(\alpha(1-\rho),\tilde R_k,1)$ only depend on the first $k$ observations $Y_1,\dots,Y_k.$ This means they can be computed sequentially as new observations arrive.
\end{remark}

\section{Applications}\label{applications}
In this subsection, we present applications of our bounds to online change point detection and first hitting time probabilities.

\subsection{Online change point detection}\label{change_oint}
Change point detection, the problem of detecting whether there is a change in the distribution of a sequence of independent random variables $(Y_i)$, is a classical problem in statistics with many applications, notably in manufacturing \cite{montgomery2020introduction, hawkins2012cumulative, woodall2014some}, finance \cite{habibi2021bayesian,banerjee2020change}, and healthcare \cite{yang2006adaptive, malladi2013online, STAUDACHER2005582, bosc2003automatic}. 

We focus on detecting a shift in the mean in an online setting. Formally, we define a {$(T, \theta)$ change point} to have occurred at an unknown time $T \ge 0$ with magnitude $\theta \ne 0$ if the sequence $(W_i)$---where $W_i := Y_i$ for $i \le T$ and $W_i := Y_i - \theta$ for $i > T$---is identically distributed. As data arrives sequentially, our goal is to test the following hypotheses:
\begin{align*}
    H_0:~\text{No change point exists.} \qquad H_1: \text{A } (T,\theta) \text{ change point exists for some } T>1, \theta\ne 0.
\end{align*}
One of the most classical statistics for this problem is the CUSUM statistic \cite{page1954continuous}, defined as $$\mathcal{T}_{s,t}:=\Big| \frac{1}{s}\sum_{i\le s}Y_i-\frac{1}{t-s}\sum_{t\ge i>s}Y_i\Big|.$$ In an online setting, 
The null hypothesis is rejected at current
time $t$ if there exists some $s\le t$ such that $\mathcal{T}_{s,t}$ exceeds a certain threshold $(\mathcal{C}_{s,t})$:$$\text{Reject the null hypothesis if } \exists s\le t~\textrm{s.t.}~\mathcal{T}_{s,t}>\mathcal{C}_{s,t} .$$ However, the literature generally does not specify how to select these thresholds to control the Type I error, i.e., the probability of false detection.

One approach, proposed by \cite{maillard:tel-02162189}, uses double time-uniform concentration inequalities to construct valid thresholds. We instead use our computable KMT inequality. We assume only that the range of $(Y_i)$ is known, and do not require prior knowledge of the mean or standard deviation.
\begin{proposition}\label{detection}
    Suppose that $(Y_i)$ satisfies \assumptionRS. Let $\delta,\delta_1,\delta_2\in (0,1)$ be such that $\delta_1+\delta_2<\delta$. Let $\beta>1$, and let $(\hat\sigma_k)$ and $([\hat\sigma_k^L, \hat\sigma_k^U])$ be a sequence of estimators and confidence sets satisfying \Cref{avril} with confidence level $\delta_1.$ 
    Let $g_\beta(t,s,\delta_2)$ be the output of \Cref{alg:change_point}, and define $$\mathcal{C}_{s,t}:=\hat\sigma_t^U\sqrt{\frac{2(1+\frac{1}{t-s})}{t-s}\log(\frac{t (\log t)^2\sqrt{t+1-s}}{(\delta-\delta_1-\delta_2)\log 2})}+\frac{t}{s(t-s)}g_\beta(t,s,\delta_2).$$
    Then, under the null hypothesis, we have $$\mathbb{P}_{H_0}\Big(\exists 0< s<t~\textrm{s.t.}~|\mathcal{T}_{s,t}|\ge \mathcal{C}_{s,t}\Big)\le \delta.$$
\end{proposition}
The proof can be found in \Cref{marre_proof}. \begin{remark} The thresholds $(\mathcal{C}_{s,t})_{s<t}$ can be chosen to guarantee that the probability of false detection is less than any significance level $\delta$. Moreover, computing these thresholds only requires knowledge of the range $R=\|Y_1\|_{\infty}$. \end{remark} 

Before presenting \Cref{alg:change_point}, we define the following notation: Given an increasing sequence $\mathbf{L}:=(L_i)$ of integers, we denote $\ell_{\mathbf{L}}(k):=\sup_i\{i~\textrm{s.t.}~k\ge \sum_{j\le i}2^{L_j}\}$, $u_{\mathbf{L}}(k):=\inf_i\{i~\textrm{s.t.}~k\le\sum_{j\le i}2^{L_j}\}$, and $N_i=\sum_{j\le i}2^{L_j}$.  We are now ready to introduce the algorithm. 

\begin{algorithm}[H]
\caption{ Computing $g_\beta(t, s, \delta_2)$ }\label{alg:change_point}
\DontPrintSemicolon
\KwIn{ Integers $t>s>0$, $\sigma>0$, $R>0$, $\delta_2 > 0$, $\beta>1$, an increasing sequence $\mathbf{L} = (L_i)$, a sequence of intervals $([\hat\sigma_k^L,\hat\sigma_k^U])$}
\KwOut{$g_\beta(t, s, \delta_2)$}

\For{$i\ge 1$}{
$\delta_{2,i}\gets\delta_2\frac{\beta^{-i}}{\sum_{j\ge 0}\beta^{-j}}$. 

$\Big((\tilde\Delta^{N_i}_{k}(\delta_{2,i}))_{k\in(N_{i-1}, N_i]},\tilde \delta_{N_i}^*(\delta_{2,i})\Big)\gets$ output of \Cref{ts13}
with input $(N_{i-1},N_i]$, $R$, $\sigma$, $\delta_{2,i}$, and $([\hat\sigma_k^L,\hat\sigma_k^U])$

Shorthand $\tilde \Delta^i_{k}:=\tilde \Delta^{N_i}_{k}(\delta_{2,i})$ and $\tilde \delta_{i}^*:=\tilde \delta_{N_i}^*(\delta_{2,i})$
}


\textbf{If $u_{\mathbf{L}}(s)=u_{\mathbf{L}}(t)$: } $g_\beta(t, s, \delta_2)\gets\sum_{k\le \ell_{\mathbf{L}}(s)}\tilde \delta_k^*\big(1-\frac{s}{t}\big)+\frac{s}{t}\tilde \delta^*_{u_{\mathbf{L}(s)}}+\tilde \Delta^{u_{\mathbf{L}}(s)}_{s}+\frac{s}{t}\tilde \Delta^{u_{\mathbf{L}}(s)}_{t}$

\textbf{If $u_{\mathbf{L}}(s)\ne u_{\mathbf{L}}(t)$: } 
\begin{align*}
    g_\beta(t, s, \delta_2)\gets&\sum_{k\le \ell_{\mathbf{L}}(s)}\tilde \delta_k^*\big(1-\frac{s}{t}\big)+\frac{s}{t}\sum_{u_{\mathbf{L}}(s)+1\le k\le \ell_{\mathbf{L}}(t)}\tilde \delta_k^*+\frac{s}{t}\frac{t-N_{\ell_{\mathbf{L}}(t)}}{N_{u_{\mathbf{L}}(t)}-N_{\ell_{\mathbf{L}}(t)}}\tilde \delta^*_{u_{\mathbf{L}}(t)}\\&+\Big(\frac{s-N_{\ell_{\mathbf{L}}(s)}}{N_{u_{\mathbf{L}}(s)}-N_{\ell_{\mathbf{L}}(s)}}-\frac{s}{t}\Big
  )\tilde \delta^*_{u_{\mathbf{L}(s)}}+\tilde \Delta^{u_{\mathbf{L}}(s)}_{s}+\frac{s}{t}\tilde \Delta^{u_{\mathbf{L}}(t)}_{t}
\end{align*}

\textbf{Return} $g_\beta(t, s, \delta_2)$
\end{algorithm}

Proposition \ref{detection} shows that the Type I error of our proposed test is well controlled. To assess its power, we compare its detection performance against the method of \cite{maillard:tel-02162189}. 
\Cref{fig:detection-rates} presents an empirical comparison of detection rates under varying mean shifts. The observations $(Y_i)$ are defined as the average of $\ell$ independent $\Unif[0,1]$ variables, with a change point at $T = 2000$ and a post-change mean increased by the specified shift.

\begin{figure}[h]
  \centering
  \includegraphics[width=\textwidth]{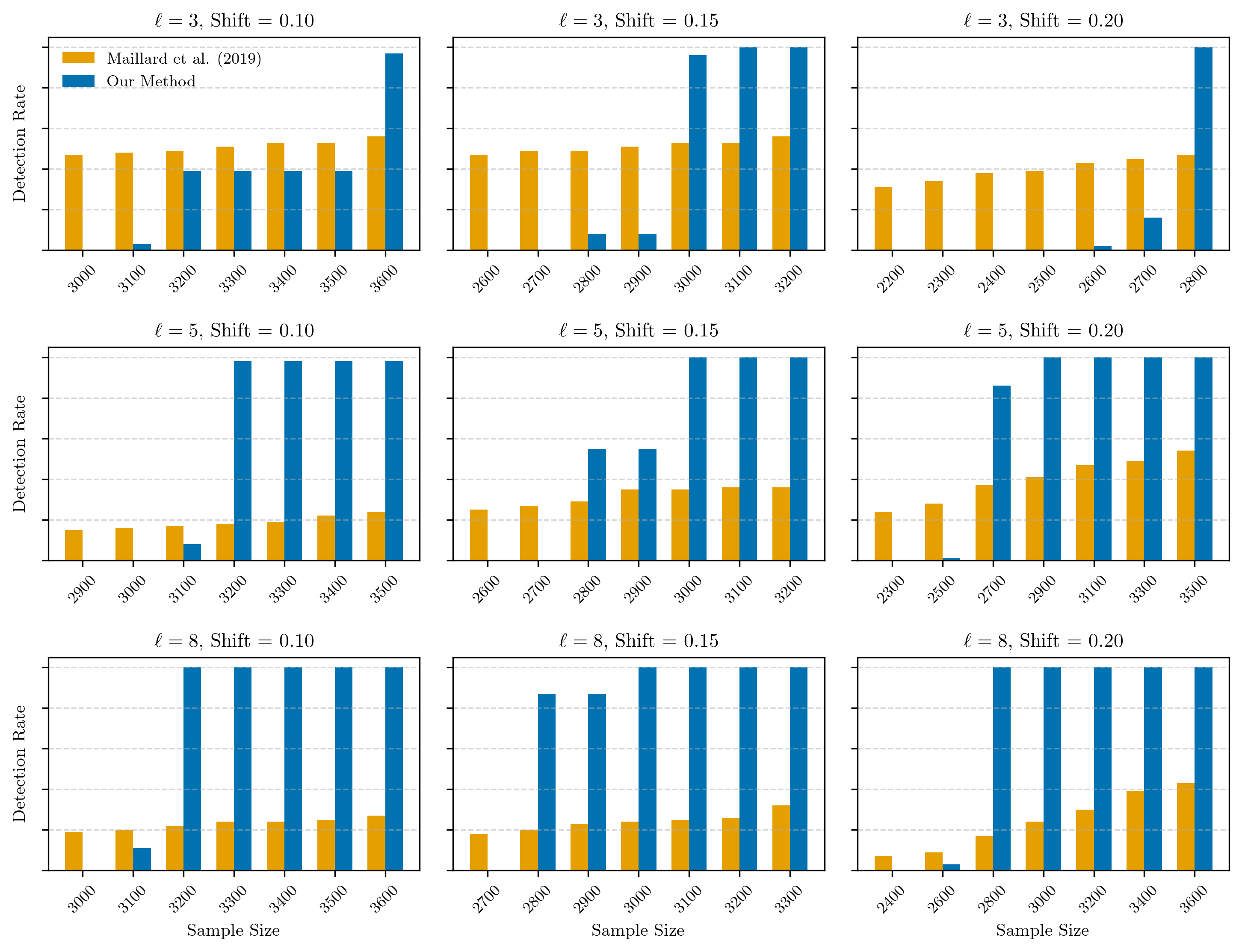}
  \caption{
Detection rate comparison across mean shifts, with each $Y_i$ an average of $\ell$ independent $\Unif[0,1]$ random variables. A change point occurs at $T = 2000$, with a post-change mean increased by the specified shift.
}
\label{fig:detection-rates}
\end{figure}
\begin{remark}
    While we focus on controlling the Type I error, another possible metric for algorithms in online change point detection is the average run length (ARL), defined as the average number of observations before a false alarm occurs \cite{page1954continuous}. Recent work \cite{shekhar2023reducing} provides methods that explicitly control the ARL by reducing the change detection problem to sequential estimation. The ARL of our procedure in Proposition \ref{detection} can be controlled by the following: let $\tau$ be the time of detection, then we have 
    \begin{equation}
        \mathbb{E}[\tau]                                                                                                                                                                                                                                                                                                                                                                      = \sum_{i \ge 1}\mathbb{P}(\tau \ge i)\\
        \ge \sum_{k\ge 1} (N_k-N_{k-1})(1-(\delta - \delta_2) - \sum_{i=1}^{k}\delta_2\frac{\beta^{-i}}{\sum_{j\ge 0}\beta^{-j}}).
    \end{equation}
\end{remark}

\subsection{First hitting time probabilities}
The first passage or hitting time $\tau$ of a Markov chain $(W_i)$ is the first time it crosses a certain boundary. Studying the distribution of $\tau$ 
plays an important role in many applications, such as finance \cite{shreve2004stochastic, kou2002jump, chicheportiche2013applicationsfirstpassageideasfinance}, neuroscience\cite{gerstein1964random,gerstner2014neuronal}, queuing theory \cite{asmussen2003applied}, and reliability analysis \cite{trivedi2001probability, howard2012dynamic, rausand2003system}. 

In the asymptotic regime, the KMT inequality has been a key tool, allowing one to reduce the study of $\tau$ to the first hitting time of a Brownian motion (e.g., \cite{portnoy1978probability,lai1979first}). 
 
 Beyond asymptotics, non-asymptotic results are also essential, particularly when dealing with triangular arrays. For instance, motivated by applications to directed polymers,  \cite{busani2020bounds} studied the running maximum of partial sums with a small negative drift that scales with the sample size. Another example arises in survival bandit problems  \cite{veroutis2024survivalmultiarmedbanditsbootstrapping}, where agents aim to maximize rewards while minimizing the probability of ruin, i.e., the exhaustion of a budget. 
The KMT inequality has not proven as powerful in that regime, notably because it depends on unknown constants, as noted in \cite{busani2020bounds}. In this subsection, we demonstrate that $
(\mathcal{D}_k(\alpha))$ in \cref{thm:main2}, which depends only on known constants, enables us to obtain computable non-asymptotic upper bounds for first hitting time probabilities. 

To be more precise, in this subsection, $(X_i^N)$ denotes a triangular array of i.i.d. random variables with mean $\mathbb{E}[X_1^N]=\mu_N$ and bounded support  $X_1^N\in [-\frac{R}{2},\frac{R}{2}]$. We define $W_i^N:=\sum_{j\le i}X_j^N$, and set 
\begin{equation}\label{eqn:def_triang_stopping_time}
    \tau_N:=\inf\{i~\textrm{s.t.}~W_i^N\ge g_i^N\},
\end{equation} where $(g_i^N)$ is a pre-specified boundary. For example, \cite{busani2020bounds} considers the case where $\mu_N=\mu\sqrt{N}^{-1/2}$ for some $\mu<0$ and $g_i^N=x$. Our goal is to bound $\mathbb{P}\big(\tau_N\ge N\big)$. 

\begin{proposition}\label{hitting_prob}
Assume that $(X_i^N)$ is a triangular array of i.i.d random variables satisfying $\mathbb{E}[X_1^N]=\mu_N$, and that for all $N$, we have $X_1^N\in [-\frac{R}{2},\frac{R}{2}]$. Let $\tau_N$ be the stopping time defined in \cref{eqn:def_triang_stopping_time}. Let $\alpha>0$, and let $(\mathcal{D}_k^N(\alpha))$ be the output of \cref{ts12} with $n=N$, $R=R$, $\sigma_N=\sqrt{\mathrm{Var}(X_1^N)}$, and confidence $\alpha$. Then we have 
\begin{align*}
    \mathbb{P}(\tau_N\ge N)\le\alpha+\mathbb{P}\Big(\forall i\le N,~ B_i\le \sigma_N^{-1}\left(g_i^N-i\mu_N+\mathcal{D}_i^N(\alpha)\right)\Big),
\end{align*} 
where $(B_i)$ is a Brownian motion. Moreover, suppose that $\sigma^*=\inf_N\sigma_N>0$, then there exists a constant $\kappa_R$ that only depends on $R$ such that for all $\alpha>0$, 
  \begin{align}\label{grr}
    \mathbb{P}(\tau_N\ge N)\le\alpha+\mathbb{P}\Big(\forall i\le N,~ B_i\le \sigma_N^{-1}\left(g_i^N-i\mu_N+\log N \kappa_R(\log N -\log \alpha )\right)\Big).
\end{align}
\end{proposition}
The proof can be found in \cref{morgane_push_proof}.
We note that the proposed upper bound only depends on the distribution of $(X_i^N)$ through its range, drift, and variance. In particular, this means the exact form of the distribution of $(X_i^N)$ is not required to compute this upper bound. When the variance is unknown, the bound can be adapted using the empirical bounds derived in \cref{empi}, and when the drift is unknown, it can be handled using valid anytime confidence sets (e.g \cite{waudby2024estimating}). 
\begin{remark}
To further bound \cref{grr}, one needs to approximate the crossing probability of the Brownian motion $\mathbb{P}\Big(\forall i\le N,~ B_i\le \sigma_N^{-1}\left(g_i^N-i\mu_N+\mathcal{D}_i^N(\alpha)\right)\Big)$, which can either be done numerically or analytically \cite{burq2008simulation,di2005simulation,karatzas2012brownian}.
\end{remark}
As an example, we contrast our bound with the one obtained  by another method in  \cite{busani2020bounds}. In that work, they assume a negative drift satisfying $-D_\mu(\log N)^{-3} \le \mu_N \le 0$ for some finite constant $D_\mu > 0$,  and the boundary is a constant $g_i^N=x>0$. They provide a computable upper bound for $\mathbb{P}(\tau_N\ge N)$: Theorem 2.2 of  \cite{busani2020bounds} states that, under general conditions, for sufficiently large $N$ and $x\ge (\log N)^2$, $$\mathbb{P}\Big(\tau_N\ge N\Big)\le C x(\log N)(\left|\mu_N\right| \vee N^{-1 / 2}),$$ where $C$ is a \emph{computable} constant that depends on $R$ and $\sigma$. Notably, if $\sigma=1$, it is easy to show that $C>4e^{56}$, which means the upper bound is trivial for any reasonable sample size.  
If we choose $\alpha=1/N$, we can recover a similar rate to that of \cite{busani2020bounds}. However, the constants are significantly smaller. 

To demonstrate this, we derive the smallest $N$ for which the upper bound in \cref{grr} becomes non-trivial when $\mu_N=\mu \sqrt{N}^{-1/2}$ for $\mu=-0.1, -0.25, -0.5$ and $g_i^N=10, 20$.
 \begin{figure}[H]
    \centering
    \includegraphics[width=0.6\textwidth]{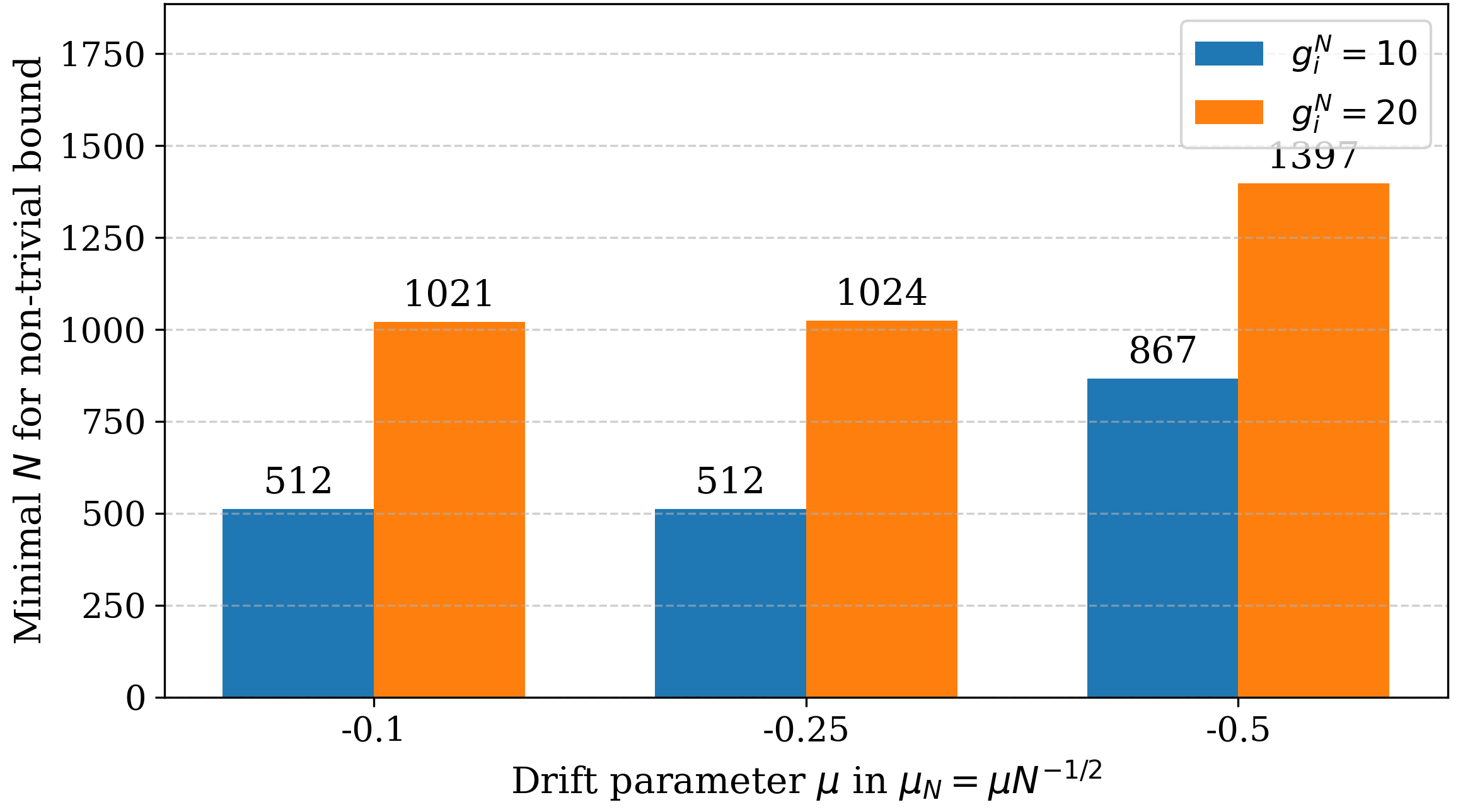}  
    \caption{Minimum value of $N$ for which the upper bound in \cref{grr} becomes non-trivial.}
    \label{fig:nontrivial-hitting-bound}
\end{figure}

\section{A Wasserstein-$p$ bound and moderate deviation bound for sequences sampled without replacement}\label{wass_sec}A main ingredient of \Cref{ts12} is our computable upper bound, $\omega_p^R(n,\sigma)$,  for the Wasserstein-$p$ distance between a centered partial sum and a normal distribution. 

To achieve the Wasserstein-$p$ bound, we use Theorem 5 of \cite{bonis2015rates}, together with an exchangeable pair method argument, to bound this distance between centered partial sums of random variables sampled \emph{without} replacement and a normal. 
More precisely, let $\mathcal{A}:=\{a_1,\dots,a_n\}$ be a finite alphabet, and let $(X^A_i)_{i\le k}$ designate a sequence of observations randomly drawn without replacement from $\mathcal{A}$. We define the centered average $W_k^A:=\sum_{i\le k}X^A_i-\frac{k}{n}\sum_{i\le n}a_i$. For a large sample size $k$ and alphabet size $n$, we can expect $W_k^A$ to be almost normally distributed, and this can be quantified using \Cref{lemma:Wp_bound_using_Bonis_basic2}. More precisely, define $(\tilde W_k^A , W_k^A)$ to be the exchangeable pairs such that 
$$\tilde W_k^A = \sum_{i=1}^{k}X^A_{(I,J)(i)}- \frac{k}{n}\sum_{i=1}^{n}X^A_{(I,J)(i)},$$
where $I\sim\Unif\{1,\dots,k\}$ and $J\sim\Unif\{k+1,\dots,n\}$. Using this we obtain the following upper-bound.
\begin{lemma}\label{lemma:Wp_bound_using_Bonis_basic2}
  Let $\tilde{\kappa}>0$ and $\eta^2>0$. Let $H_k(\cdot)$ denote the $k$-th Hermite polynomial. Let $\tilde Z\sim \mathcal{N}(0,1)$, and shorthand $H_k:=H_k(\tilde Z)$. Then 
 for all $p\ge1$, the following holds:
\begin{equation}
\begin{aligned}
   & \mathcal{W}_p(W^A_k,\mathcal{N}(0,\eta^2))\\&\le  \int_0^{\tilde\kappa}\Big\|e^{-t} W^A_k-\frac{e^{-2t}}{\sqrt{1-e^{-2t}}}\eta\tilde Z\Big\|_pdt
    +\int_{\tilde\kappa}^\infty
 e^{-t}\|S_p(t)\|_pdt,
\end{aligned} 
\end{equation} where \begin{align*}
    S_p(t):= &\mathbb{E}\big[\frac{W_k^A-\tilde W_k^A }{s}-
W_k^A|W_k^A\big]
\\&+\frac{e^{-2t}\|H_1\|_p}{\sqrt{1-e^{-2t}}}\Big(\frac{1}{2s}\eta^{-1}\mathbb{E}\big[(\tilde W_k^A -W^A_k)^2|W^A_k\big]-1\Big)
\\&+\sum_{\ell\ge 3}\frac{e^{-\ell t}\|H_{\ell-1}\|_p}{s\ell!(\sqrt{1-e^{-2t}})^{\ell-1}}\eta^{-\ell+1}\mathbb{E}\big[(\tilde W_k^A -W^A_k)^\ell|W_k^A\big].
\end{align*} 
\end{lemma}
The proof can be found in \Cref{derek}.

Note that Lemma \ref{lemma:Wp_bound_using_Bonis_basic2} builds on  Theorem 5 of \cite{bonis2020stein} but uses a specific Stein exchangeable pair. The proposed upper-bound can be further simplified if $\eta^2\approx\mathrm{Var}(W_k^A)$. 

When $\mathcal{A}:=\{Y_1,\dots,Y_n\}$, $W_k^A$ has the same distribution as $W_k$ conditioned on $\mathcal{U}(X_{1:n})$; that is, $W_k^A\overset{d}{=}W_k|\mathcal{U}(X_{1:n})$. Hence Lemma \ref{lemma:Wp_bound_using_Bonis_basic2} can be used to obtain an upper bound $\omega_p^R(n,k,\sigma)$ for the conditional Wasserstein distance between $W_k$ and a normal random variable.  Moreover, we show that $\omega_p^R(n,k,\sigma)$ grows linearly in $p$ and is bounded as $n,k$ increase, meaning that $$\sup_{n,k}\omega_p^R(n,k,\sigma)<\infty.$$ This matches the results obtained for the Wasserstein distance in the i.i.d. setting \cite{bonis2020stein,austern2023efficient}, where no additional assumptions are made about the distribution of the random variables $(Y_i)$, and shows that the dependence introduced by sampling without replacement does not change the rate of convergence of $W_k$.

\begin{theorem}\label{thm:main1}
 Suppose that $(X_i)$ satisfy \assumptionRS.  For any positive integers $1\le k\le n$, the following inequality holds:
    \begin{equation}\label{eqn:main1}
        \Big\|\mathcal{W}_p(W_k,\mathcal{N}(0,\frac{k(n-k)}{n}\sigma^2)|\mathcal{U}(X_{1:n})\Big)\Big\|_p \le \omega^R_p(n,k,\sigma),
    \end{equation}
    where $ \omega^R_p(n,k,\sigma)$ is defined in \Cref{theorem:pfWp}, \cref{definition_omega}. If $n$ is even, we define $\omega^R_p(n,\sigma):=\omega^R_p(n,\frac{n}{2},\sigma)$ and obtain \begin{equation}\label{eqn:main122}
      \Big\|\mathcal{W}_p(W_{n/2},\mathcal{N}(0,\frac{n\sigma^2}{4})|\mathcal{U}(X_{1:n})\Big)\Big\|_p \le \omega^R_p(n,\sigma).
    \end{equation}

Moreover, there exists a constant $\tilde K_{R,\sigma}<\infty$ defined in \cref{definition_k} that only depends on $R$ and $\sigma$ such that for all $p\ge 2$,
 \begin{equation}\label{eqn:main122b}
 \omega_p^R(n,k,\sigma) \le {\tilde K_{R,\sigma}p}.
    \end{equation}
    In addition, for all $0<\sigma_1\le \sigma_2\le \frac{R}{2}$, we have $$\sup_{\sigma\in [\sigma_1,\sigma_2]}\tilde K_{R,\sigma}<\infty.$$
\end{theorem} 

The proof can be found in \Cref{tperr} and \Cref{dearreader}, and the exact expression of $\omega_p^R(n,k,\sigma)$ is given  in \Cref{definition_omega} within \Cref{theorem:pfWp}. 

Interestingly, we note that the Wasserstein-$p$ distance $\Big\|\mathcal{W}_p(W_{k},\mathcal{N}(0,\sigma_k^2)|\mathcal{U}(X_{1:n})\Big)\Big\|_p$ is uniformly upper bounded for all random variables $(Y_i)$ that are bounded by $R$ and have a variance $\sigma$ bounded away from $0.$ This fact is crucial in the proof of \Cref{lemma:Delta_order}.  

In addition, the conditional Wasserstein-$p$ bound in \Cref{thm:main1} allows us to 
obtain a Cramér-type moderate deviation result for the normalized centered partial sum $\frac{1}{\sigma_{n,k}}W_k$. 

\begin{corollary}
\label{lemma:main_moderate}
The following inequality holds for all $t \ge \frac{4e\tilde K_{R,\sigma}}{\sigma_{n,k}} + 2$:
    \begin{align}\label{eqn:moderate}
   {\mathbb{P}(\frac{1}{\sigma_{n,k}} W_k \ge t)}\le{\Phi^c(t)}\Big[1+\frac{e\tilde K_{R,\sigma}\varphi(t)}{\Phi^c(t)\sigma_{n,k}}\Big]\Big(\frac{e\tilde K_{R,\sigma}\varphi(t)}{\sigma_{n,k}}\Big)^{-\frac{e\tilde K_{R,\sigma}(t+1)}{e\tilde K_{R,\sigma}(t+1) + \sigma_{n,k}}},
\end{align}
where $\tilde K_{R,\sigma}$ is the constant in \Cref{thm:main1}.
\end{corollary}
The proof can be found in \Cref{sec:app_moderate_deviation}.

Cramér-type moderate deviation result have been well-studied, see e.g. in \cite{chen2013stein, cramer1994nouveau, feller2015retracted, petrov1975sums, fang2023wass, jing2003self}. In the classical i.i.d. setting, it is shown in \cite{petrov2012sums} that for the standardized sum
$S_n/\sqrt{n}$ the relative error is of order $1+O((1+t^3)/\sqrt{n})$ for
some $t = o(n^{1/6})$, and this is known to be optimal. Most Cramér-type relative error bounds has unidentified constants
\cite{cramer1994nouveau, petrov1975sums, chen2013stein, fang2023wass}, whereas \cite{feller2015retracted} and \cite{austern2023efficient} obtain explicit constants.
We remark that \cite{fang2023wass} used Wasserstein-$p$ bounds to derive Cramér-type moderate deviation for some dependent sum. 
Our result also follows from this principle and is of the same order as the ones obtained in \cite{fang2023wass}. 

\section{Discussion}\label{discussion}

This paper establishes a new computable version of the Koml\'{o}s-Major-Tusn\'{a}dy (KMT) inequality, providing explicit bounds that depend only on the range and standard deviation of the random variables. We also derive an empirical version for cases where the standard deviation is unknown. The derived computable bounds thus make the KMT inequality a practical tool for deriving explicit, finite-sample guarantees.

Our proof relies on two key components: a novel inductive construction and a computable conditional Wasserstein-$p$ bound. The inductive construction is inspired by a recent line of research that uses Stein's method to provide simpler alternatives to the notoriously complex original KMT proofs, particularly the work of \cite{chatterjee2012new, Bhattacharjee16}. However, a significant difference is that \cite{Bhattacharjee16,chatterjee2012new} rely on bounding the moment generating function of the approximation error $|Z_n-S_n|$, which is done by constructing a bounded Stein kernel for $S_n$. This technique has not been successfully generalized to arbitrary bounded random variables. Our approach, in contrast, relies on controlling the conditional Wasserstein-$p$ distance (for all $p\ge 2$) between the sum of variables sampled without replacement and a Gaussian random variable. We derive these bounds using an adaptation of the Stein's method of exchangeable pairs from \cite{bonis2015rates}. A byproduct of our conditional Wasserstein-$p$ distance bound is a Cramér-type moderate deviation result (\Cref{lemma:main_moderate}).

The main drawback of our bound is a suboptimal asymptotic rate. Although in practice tighter than other computable bounds for any reasonable sample size (see \cref{section:compare}), our bounds are asymptotically larger by a logarithmic factor. This suboptimality stems from our method for bounding the conditional Wasserstein-$p$ distance. Our bounds for the Wasserstein-$p$ distance grow linearly with $p$ (see \Cref{thm:main1}); the rate could be improved to $\log(n)^{3/2}$ if this growth were reduced to $\sqrt{p}$, as is the case when the partial sum admits a bounded Stein kernel \cite{chatterjee2012new}. A further improvement to the optimal $\log(n)$ rate is possible if an almost sure bound on the conditional Wasserstein-$p$ distance could be established, rather than the $L_p$ norm bound we currently obtain. We leave these potential refinements for future work. Despite this asymptotic suboptimality, our bounds offer better finite-sample performance than existing computable alternatives \cite{castelle1998strong} and \cite{Bhattacharjee16}, as shown in \Cref{section:compare_clb_bg}.

Although this work is limited to univariate bounded random variables, our framework is extensible to several other settings. An extension to the multivariate and sub-Gaussian settings would require developing a corresponding conditional Wasserstein-$p$ bound for multivariate and sub-gaussian random variables, which could build upon recent work \cite{bonis2023improved}. Furthermore, our framework could be adapted for variables that only have a finite number of moments. For variables with $q+2$ finite moments, where strong approximation rates of $O(n^{1/q})$ are known \cite{major1976approximation,komlos1975approximation,komlos1976approximation}, deriving computable bounds would require adapting our Wasserstein-$p$ bounds for $p \le q$. We leave these extensions for future investigation.

\section*{Acknowledgments}
The authors were supported by ONR N000142112664 and NSF CAREER DMS 2441652. HY was also partially supported by a research grant from the Citadel Securities PhD Summit.

\bibliographystyle{plain}
\bibliography{ref}

 \appendix

 \crefalias{section}{appendix} 

\section{Notations}\label{section:proof_notation}
We recall some notations used throughout the paper. Let $(Y_i)_{i\ge 1}$ be a sequence of i.i.d.\ random variables with $\mathbb{E}[Y_1] < \infty$, $\mathrm{Var}(Y_1) = \sigma^2$, and $0 \le Y_i \le R$ almost surely. 
Define $X_i=Y_i-\mathbb{E}[Y_i]$ for $i\ge 1$.
Then $(X_i)_{i\ge1}$ is a sequence of i.i.d. random variables with $\mathbb{E}[X_1] = 0$ and $\mathrm{Var}(X_1) = \sigma^2$. Let $(Z_i)_{i\ge1}$ be a sequence of i.i.d. standard normal random variables, independent of $(X_i)_{i\ge1}$. 

In addition, we use the following notation:
\begin{itemize}
    \item Denote $\mathcal{U}(X_{1:n})$ to be the unordered set of $\{X_1,\dots, X_n\}$.
\item Let $Z,Z'$ be two random variables on the same probability space as $(X_i)$.
We define the conditional Wasserstein-$p$ distance as
\begin{equation*}
    \mathcal{W}_p(Z,Z'|\mathcal{U}(X_{1:n})) := \left( \inf_{\gamma \in \Gamma(Z, Z' \mid \mathcal{U}(X_{1:n}))} \int |x - y|^p \, d\gamma(x, y) \right)^{1/p},
\end{equation*}
where $\Gamma(Z, Z' \mid \mathcal{U}(X_{1:n}))$ is the set of all couplings whose marginals are the conditional distributions of $Z,Z'$ given $\mathcal{U}(X_{1:n})$.
    \item For $n >0$ and $0<k<n$, define
    $$S_n := \sum_{i=1}^{n}X_i\qquad \textrm{and}\qquad W_k:=S_k-\frac{k}{n}S_n.$$
We adopt the standard convention that $S_0 = 0$ and $W_0 = 0$.
\item Let $(U_i)_{i\ge1}$ denote a sequence of i.i.d $\rm{unif}[0,1]$ random variables, independent of $(Y_i)_{i\ge1}$.
    \item For $\ell\in \mathbb{N}$, denote $H_\ell(\cdot)$ to be the one-dimensional $\ell$-th Hermite polynomial $H_\ell(x) := e^{x^2 / 2} \frac{\partial^\ell}{\partial x} e^{-x^2 / 2}$ and shorthand $H_{\ell}:=H_{\ell}(Z)$ where $Z\sim \mathcal{N}(0,1)$. 
    \item Define $R_s,\tilde R_s,\tilde R$ as $R_s:=\frac{1}{2}(R+\sqrt{R^2-4\sigma^2})$, $\tilde R_s:=R_s/\sigma$ and $\tilde R:=R/\sigma$.
    \item For $p>2$, define  $A_p:=2^{1 / p} \sqrt{p / 2+1} e^{\frac{1}{2}+\frac{1}{p}}$ and $A_{n,p}^{*}:= \left(\frac{p}{2}+1\right) n^{1 / p-1 / 2}$. 
    \item For $\sigma>0$, denote $\Phi_\sigma$ the CDF of $N(0,\sigma^2)$.
\end{itemize}

\section{Preliminary Lemmas}\label{prelim}
\begin{lemma}\label{lemma:R_s_bound}
Assume that $X_1$ satisfies \assumptionRS.    The following inequality holds 
    $$|X_1|\le R_s.$$
\end{lemma}
\begin{proof}
    We remark that the following holds
    \begin{align*}
        \mathbb{E}[Y_1^2]=\sigma^2+\mathbb{E}[Y_1]^2
    \end{align*}Moreover, we also note that as $Y_1$ is almost surely positive, so it follows that $\mathbb{E}[Y_1^2]\le R\mathbb{E}[Y_1].$ Hence we have that $$\sigma^2+\mathbb{E}[Y_1]^2\le R\mathbb{E}[Y_1].$$ 
 This directly implies a lower bound for $\mathbb{E}[Y_1]$ which used on the definition of $X_1$ directly gives us the desired result.
\end{proof}

\begin{lemma}\label{lemma:PITdiscrete}
Let $X$ be a random variable with finite support and CDF $F(\cdot)$, and probability mass function $p(\cdot)$. Let $F(x^-)$ denote the left-hand limit of $F(x)$. Let $U \sim \mathrm{Uniform}[0, 1]$ be independent of $X$. Define 
\begin{equation}\label{eqn:defrandCDF}
{\tilde{F}^{U}}(x) = F(x^-) + U \cdot p(x),
\end{equation}Let $\sigma>0$, 
then the following holds:
$$\sigma\Phi^{-1}({\tilde{F}^{U}}(X)) \sim \mathcal{N}(0,\sigma^2).$$ 
\end{lemma}
\begin{proof}
Suppose that the finite support of $X$ is $\Omega=\{x_1, x_2, \dots, x_N\}$, where we take $x_1<\dots<x_N$. We denote $p_i:=\mathbb{P}(X = x_i)$ and remark that $\sum_{i=1}^N p_i = 1$. For any $t\in [0, 1]$, let $x_i$ be the largest element of $\Omega$ such that $F_X(x_i^-) \leq t$. Then, by law of total probability, 
\begin{align*}
    \mathbb{P}({\tilde{F}_{X}}(X)\le t) = &\mathbb{P}(X<x_i) + \mathbb{P}(F_X(X^-) + U \cdot p_X(X)\le t|X = x_{i})\mathbb{P}(X=x_{i})\\
    = &\sum_{k=1}^{i-1}p_k + \mathbb{P}(\sum_{i=1}^{i-1}p_k + U \cdot p_i\le t)\cdot p_i\\
    = & \sum_{k=1}^{i-1}p_k + \frac{t - \sum_{k=1}^{i-1}p_k}{p_i}\cdot p_i\\
    = &t
\end{align*}
Hence ${\tilde{F}_{X}}(X) \sim \text{Uniform}(0, 1)$.
\end{proof}

\begin{lemma}
    \label{lemma:bound_S_and_Z} Let $\mu$ and $\nu$ be probability measures on $\mathbb{R}$, and let $X\sim \mu$. Then there exists a random variable $Y\sim \nu$, measurable with respect to $X$ and an independent $U\sim \Unif[0,1]$, such that for all $\alpha>0$ and all $p\ge 2$,  $$\mathbb{P}\Big(|X-Y|\ge \frac{\mathcal{W}_p(\mu,\nu)}{\alpha^{1/p}}\Big)\le \alpha.$$
\end{lemma}
\begin{proof} 
     Define $Y:=\tilde F^U(X)$, where $\tilde F^U(\cdot)$ is defined in \cref{eqn:defrandCDF}. By \cref{lemma:PITdiscrete}, we have $Y\sim \nu$. Moreover, by Theorem 2.18 of \cite{villani2003topics}, we know that $$\|X-Y\|_p=\mathcal{W}_p(\mu,\nu).$$ Now by Markov's inequality, we obtain that  \begin{align*}
         \mathbb{P}\Big(|X-Y|\ge \frac{\mathcal{W}_p(\mu,\nu)}{\alpha^{1/p}}\Big)&\le \frac{\alpha\|X-Y\|_p^p}{\mathcal{W}_p(\mu,\nu)^p}=\alpha.
     \end{align*}
\end{proof}
\begin{lemma}\label{boy_breaks}
    Define $\tilde Z_{n/2}:=\frac{\sqrt{n}\sigma}{2}\Phi^{-1}\Big(\tilde F_{n/2}^{U_{n/2}}(W_{n/2}))\Big)$, where $\tilde F_{n/2}^{U_{n/2}}(\cdot)$ is defined in \cref{uc}. Then, for all $p \ge 1$, the following holds:
    \begin{align*}
        \Big\|W_{n/2}-\tilde Z_{n/2}\Big\|_p=\Big\|\mathcal{W}_p(W_{n/2}, \mathcal{N}(0,\sigma^2n/4)|\mathcal{U}(X_{1:n}))\Big\|_p.
    \end{align*}
\end{lemma}
\begin{proof} By \cref{lemma:PITdiscrete}, conditionally on $\mathcal{U}(X_{1:n})$, we have that $\tilde Z_{n/2}\sim \mathcal{N}\Big(0, \frac{n\sigma^2}{4}\Big)$. Moreover according to Theorem 2.18 of \cite{villani2003topics} we obtain that $$\mathbb{E}\Big[\Big|W_{n/2}-\tilde Z_{n/2}\Big|^p\Big|\mathcal{U}(X_{1:n})\Big]^{1/p}=\mathcal{W}_p(W_{n/2}, \mathcal{N}(0,\sigma^2n/4)|\mathcal{U}(X_{1:n})). $$ Hence, by the tower law, we obtain that 
    \begin{align*}
   \Big\|W_{n/2}-\tilde Z_{n/2}\Big\|_p=\Big\|\mathcal{W}_p(W_{n/2}, \mathcal{N}(0,\sigma^2n/4)|\mathcal{U}(X_{1:n}))\Big\|_p.
\end{align*}
\end{proof}

\begin{lemma}[Rosenthal's inequality with explicit constants.]\label{lemma:rosenthal_explicit-const}
Let $(\tilde{X}_i)_{i \geq 1}$ be a sequence of centered i.i.d. observations. If $\|\tilde{X}_1\|_p<\infty$ for some $p \geq 2$, then
$$
\left\|\frac{1}{\sqrt{n}} \sum_{i \leq n} \tilde{X}_i\right\|_p \leq\left(\frac{p}{2}+1\right) n^{1 / p-1 / 2}\left\|\tilde{X}_1\right\|_p+2^{1 / p} \sqrt{p / 2+1} e^{\frac{1}{2}+\frac{1}{p}}\left\|\tilde{X}_1\right\|_2 .
$$
\end{lemma}
\begin{proof}
    See Lemma 14 in \cite{austern2023efficient}.
\end{proof}

\section{Proofs of results in Section \ref{wass_sec}}\label{section:KMTproof}

\subsection{Proof of Theorem \ref{lemma:Wp_bound_using_Bonis_basic2}}\label{derek}
\subsubsection{Score bound for the conditional Wasserstein distance}

Recall that $\mathcal{A} := \{a_1, \dots, a_n\}$ is a finite alphabet, and $(X_i^A)_{i \le k}$ denotes a sequence drawn uniformly without replacement from $\mathcal{A}$. We define the centered sum $W_k^A := \sum_{i=1}^k X_i^A - \frac{k}{n} \sum_{i=1}^n a_i$, and its exchangeable counterpart $\tilde{W}_k^A := \sum_{i=1}^k X^A_{(I,J)(i)} - \frac{k}{n} \sum_{i=1}^n X^A_{(I,J)(i)}$, where $I \sim \Unif\{1,\dots,k\}$ and $J \sim \Unif\{k+1,\dots,n\}$.

In this subsection, we will derive an upper bound for $\big\|\mathcal{W}_p(W^A_k,\mathcal{N}(0,\eta^2))\big\|_p$ in terms of a score function. In this goal, we first remark that  \begin{align*}
    \mathcal{W}_p(W^A_k,\mathcal{N}(0,\eta^2))&\le \eta    \mathcal{W}_p(\eta^{-1}W^A_k,\mathcal{N}(0,1)).
\end{align*}Let $\tilde Z\sim \mathcal{N}(0,1)$ be a standard normal, independent of $W^A_k$. We define the process $(X_t)$ to interpolate between $\eta^{-1}W_k^A$ and $\tilde Z$ in the following way $$X_t:= e^{-t}\eta^{-1}W^A_k+\sqrt{1-e^{-2t}}\tilde Z.$$ 
Denote $f_t$ the density of $X_t$, meaning that for all $A\in \mathcal{B}(\mathbb{R})$, we have $\mathbb{P}(X_t\in A)=\int_A f_t(x)dx$. 
Define $
\xi_t := (\log h_t)'(F_t),
$
where $h_t(\cdot) = \frac{f_t(\cdot)}{\varphi(\cdot)}$ represents the conditional density of $F_t$ with respect to the standard Gaussian measure. Then we use Lemma 2 of \cite{otto2000generalization}.

\begin{lemma}[Lemma 2 of \cite{otto2000generalization}]\label{lemma:Wp_Lp_bound_by_int_ht}
For any $p>2$, under the above notations, the following inequality holds:
$$\mathcal{W}_p(\eta^{-1}W^A_k,\mathcal{N}(0,1))\le \int_0^\infty \|\xi_t\|_p dt.$$
\end{lemma}
Using Lemma 5 from \cite{bonis2020stein}, we directly obtain the following bound
\begin{lemma}[Theorem 5 from \cite{bonis2020stein}]\label{lemma:a_version_of_score}
Let $t>0$. The following equality holds:
$$\xi_t\overset{a.s}{=}\mathbb{E}\Big[e^{-t}\eta^{-1}W^A_k-\frac{e^{-2t}}{\sqrt{1-e^{-2t}}}\tilde Z\Big|X_t\Big].$$ 
\end{lemma}
Those combined directly imply that $$\mathcal{W}_p(W_k^A,\mathcal{N}(0,\eta^2))\le \eta \int_0^{\infty}\Big\|\mathbb{E}\Big[e^{-t}\eta^{-1}W^A_k-\frac{e^{-2t}}{\sqrt{1-e^{-2t}}}\tilde Z\Big|X_t\Big]\Big\|_pdt.$$
\subsubsection{Bounding the score bound using Stein's exchangeable pair}In this subsection we will provide a further upper-bound by using Stein's exchangeable pair method. Throughout this section we will write $$s:=\frac{n}{k(n-k)}.$$Moreover we define 
$\tilde W^A_k$ as
$$\tilde W^A_k = \sum_{i=1}^{k}X^A_{(I,J)(i)}- \frac{k}{n}\sum_{i=1}^{n}a_{(I,J)(i)},$$
where $I\sim\mathrm{Unif}\{1,\dots,k\}$ and $J\sim\mathrm{Unif}\{k+1,\dots,n\}$. Note that $(W^A_k,\tilde W^A_k)$ form an exchangeable pair as we have $$(W^A_k,\tilde W^A_k)\overset{d}{=}(\tilde W^A_k,W^A_k).$$ 
Recall that we denote by $H_\ell(\cdot)$ the one-dimensional $\ell$-th Hermite polynomial.

\begin{lemma}
Let $\tilde\kappa>0$. Assume that the conditions of \Cref{lemma:Wp_bound_using_Bonis_basic2} hold and shorthand $H_k:=H_k(\tilde Z)$. Then 
 for all $p\ge2$, we have:
\begin{equation}\label{eqn:W_p_bound_intermediate2}
\begin{aligned}
   &\eta^{-1} \mathcal{W}_p(W^A_k,\mathcal{N}(0,\eta^2))\\&\le \int_0^{\tilde\kappa} \mathbb{E}\Big[\Big|e^{-t} \eta^{-1}W^A_k-\frac{e^{-2t}}{\sqrt{1-e^{-2t}}}\tilde Z\Big|^p\Big]^{\frac{1}{p}}dt
    \\&+\eta^{-1}\int_{\tilde\kappa} ^{\infty}e^{-t}\Big\|\frac{1}{s}\mathbb{E}\big[W^A_k-\tilde W^A_k|W^A_k\big]-
W^A_k\Big\|_pdt
\\&+\int_{\tilde\kappa} ^{\infty}\frac{e^{-2t}\|H_1\|_p}{\sqrt{1-e^{-2t}}}\Big\|\frac{1}{2s}\eta^{-2}\mathbb{E}\big[(\tilde W^A_k-W^A_k)^2|W^A_k\big]-1\Big\|_pdt
\\&+\sum_{l\ge 3}\int_{\tilde\kappa} ^{\infty}\frac{e^{-lt}\|H_{l-1}\|_p}{sl!(\sqrt{1-e^{-2t}})^{l-1}}\eta^{-\ell}\Big\|\mathbb{E}\big[(\tilde W^A_k-W^A_k)^\ell|W^A_k\big]\Big\|_pdt.
\end{aligned} 
\end{equation}
\end{lemma}
\begin{proof}
    We introduce $\tau_t$ similarly as in \cite{bonis2020stein}:
    \begin{equation}\label{eqn:def_tau_t}
        \tau_t := \sum_{\ell>0} \frac{e^{-\ell t}}{s^\ell \ell! \sqrt{1-e^{-2t}}} \eta^{-\ell}\mathbb{E}[(\tilde W_k^A - W_k^A)^\ell |W_k^A] H_{\ell-1}(\tilde{Z}).
    \end{equation}
    Adapting the proof of Lemma 8 from \cite{bonis2020stein}, we have
    $$\mathbb{E}[\tau_t|X_t] = 0 ~ a.s.$$
    This directly implies that 
    \begin{align*}
        \xi_t=&\mathbb{E}\Big[e^{-t}\eta^{-1}W^A_k-\frac{e^{-2t}}{\sqrt{1-e^{-2t}}}\tilde Z\Big| X_t\Big]\\
        =&\mathbb{E}\Big[\tau_t + e^{-t}\eta^{-1}W^A_k-\frac{e^{-2t}}{\sqrt{1-e^{-2t}}}\tilde Z\Big| X_t\Big]. 
    \end{align*}
Using \Cref{lemma:Hermite_bound_Bonis}, we then obtain that
$$
\|\xi_t\|_p\le e^{-t}\mathbb{E}
[S_p(t)^{\frac{p}{2}}]^{1/p},$$
where
\begin{align*}
    S_p(t):= &\eta^{-1}\mathbb{E}\big[\frac{W^A_k-\tilde W_k^A}{s}-
W^A_k|W_k^A\big]
\\&+\frac{e^{-2t}\|H_1\|_p}{\sqrt{1-e^{-2t}}}\Big(\frac{1}{2s}\eta^{-2}\mathbb{E}\big[(\tilde W_k^A-W_k^A)^2|W_k^A\big]-1\Big)
\\&+\sum_{\ell\ge 3}\frac{e^{-\ell t}\|H_{\ell-1}\|_p}{s\ell!(\sqrt{1-e^{-2t}})^{\ell-1}}\eta^{-\ell}\mathbb{E}\big[(\tilde W_k^A-W_k^A)^\ell|W_k^A\big]
\end{align*}and where for ease of notations we shorthanded $H_k:=H_k(\tilde Z).$
    Finally, by combining this with  \Cref{lemma:Wp_Lp_bound_by_int_ht}, we directly obtain that  for all $\tilde{\kappa}>0$,
    \begin{align*}
&\mathcal{W}_p(\eta^{-1}W_k^A,\mathcal{N}(0,1))\\
\le&\int_0^{\tilde\kappa} \|\xi_t\|_p dt+\int_{\tilde\kappa} ^\infty \|\xi_t\|_p dt\\\le& \int_0^{\tilde\kappa} \mathbb{E}\Big[\Big|e^{-t} \eta^{-1}W^A_k-\frac{e^{-2t}}{\sqrt{1-e^{-2t}}}\tilde Z\Big|^p\Big]^{\frac{1}{p}}dt
    \\&+\eta^{-1}\int_{\tilde\kappa} ^{\infty}e^{-t}\Big\|\frac{1}{s}\mathbb{E}\big[W^A_k-\tilde W_k^A|W_k^A\big]-
W^A_k\Big\|_pdt
\\&+\int_{\tilde\kappa} ^{\infty}\frac{e^{-2t}\|H_1\|_p}{\sqrt{1-e^{-2t}}}\Big\|\frac{1}{2s}\eta^{-2}\mathbb{E}\big[(\tilde W_k^A-W_k^A)^2|W_k^A\big]-1\Big\|_pdt
\\&+\sum_{\ell\ge 3}\int_{\tilde\kappa} ^{\infty}\frac{e^{-\ell t}\|H_{\ell-1}\|_p}{s\ell!(\sqrt{1-e^{-2t}})^{\ell-1}}\eta^{-\ell}\Big\|\mathbb{E}\big[(\tilde W_k^A-W_k^A)^\ell|W_k^A\big]\Big\|_pdt.
    \end{align*}
    Therefore \cref{eqn:W_p_bound_intermediate} follows. 
\end{proof}
\subsection{Proof of Theorem \ref{thm:main1}}\label{tperr}In this subsection we apply \Cref{lemma:Wp_bound_using_Bonis}, to bound the conditional Wasserstein-$p$ distance between $W_k$ and $\mathcal{N}(0,\frac{k(n-k)}{n}\sigma^2)$.

\subsubsection{Bounding the conditional Wasserstein-$p$ distance}
Throughout this section we will write $$s:=\frac{n}{k(n-k)}.$$Moreover we define 
$W_k^\prime$ as
$$W_k^\prime = \sum_{i=1}^{k}X_{(I,J)(i)}- \frac{k}{n}\sum_{i=1}^{n}X_{(I,J)(i)}.$$
where $I\sim\mathrm{Unif}\{1,\dots,k\}$ and $J\sim\mathrm{Unif}\{k+1,\dots,n\}$. Note that conditionally on $\mathcal{U}(X_{1:n})$, $(W_k,W'_k)$ form an exchangeable pair, as we have $$(W_k,W_k',\mathcal{U}(X_{1:n}))\overset{d}{=}(W'_k,W_k,\mathcal{U}(X_{1:n})).$$ 
Recall that we denote by $H_\ell(\cdot)$ the one-dimensional $\ell$-th Hermite polynomial.

\begin{lemma}\label{lemma:Wp_bound_using_Bonis}
Let $\kappa>0$. Let $\tilde Z\sim\mathcal{N}(0,1)$ be an independent normal. Assume that the conditions of \Cref{thm:main1} hold and shorthand $H_k:=H_k(\tilde Z)$. Then if we define $\sigma_{n,k}^2:=\sigma^2\frac{k(n-k)}{n}$  
 for all $p\ge2$ we have:
\begin{equation}\label{eqn:W_p_bound_intermediate}
\begin{aligned}
   &\sigma_{n,k}^{-1} \Big\|\mathcal{W}_p(W_k,\mathcal{N}(0,\sigma_{n,k}^2)|\mathcal{U}(X_{1:n}))\Big\|_p\\
   \le &  \int_0^{-\frac{1}{2}\log(1-\frac{R^2}{\kappa\sigma_{n,k}^2})} \mathbb{E}\Big[\Big|e^{-t} \sigma_{n,k}^{-1}W_k-\frac{e^{-2t}}{\sqrt{1-e^{-2t}}}\tilde Z\Big|^p\Big]^{\frac{1}{p}}dt
    \\&+\sigma_{n,k}^{-1}\int_{-\frac{1}{2}\log(1-\frac{R^2}{\kappa \sigma_{n,k}^2})}^{\infty}e^{-t}\Big\|\frac{1}{s}\mathbb{E}\big[W_k-W_k'|W_k,\mathcal{U}(X_{1:n})\big]-
W_k\Big\|_pdt
\\&+\int_{-\frac{1}{2}\log(1-\frac{R^2}{\kappa \sigma_{n,k}^2})}^{\infty}\frac{e^{-2t}\|H_1\|_p}{\sqrt{1-e^{-2t}}}\Big\|\frac{1}{2s}\sigma_{n,k}^{-2}\mathbb{E}\big[(W'_k-W_k)^2|W_k,\mathcal{U}(X_{1:n})\big]-1\Big\|_pdt
\\&+\sum_{\ell\ge 3}\int_{-\frac{1}{2}\log(1-\frac{R^2}{\kappa \sigma_{n,k}^2})}^{\infty}\frac{e^{-\ell t}\|H_{\ell-1}\|_p}{s\ell!(\sqrt{1-e^{-2t}})^{\ell-1}}\sigma_{n,k}^{-\ell}\Big\|\mathbb{E}\big[(W_k'-W_k)^\ell|W_k,\mathcal{U}(X_{1:n})\big]\Big\|_pdt\\
=:& (a_0)+(a_1)+(a_2)+(a_3).
\end{aligned} 
\end{equation}
\end{lemma}
\begin{proof}
    According to \Cref{lemma:Wp_bound_using_Bonis_basic2} applied with the alphabet $\mathcal{A}:=\{X_1,\dots,X_n\}$ and $\tilde\kappa:=-\frac{1}{2}\log(1-\frac{R^2}{\kappa \sigma_{n,k}^2})$, we obtain that \begin{equation}\label{eqn:W_p_bound_intermediate}
\begin{aligned}
   &\sigma_{n,k}^{-1} \mathcal{W}_p(W_k,\mathcal{N}(0,\sigma_{n,k}^2)\big|\mathcal{U}(X_{1:n}))\\&\le  \int_0^{-\frac{1}{2}\log(1-\frac{R^2}{\kappa \sigma_{n,k}^2})} \mathbb{E}\Big[\Big|e^{-t} \sigma_{n,k}^{-1}W_k-\frac{e^{-2t}}{\sqrt{1-e^{-2t}}}\tilde Z\Big|^p\Big|\mathcal{U}(X_{1:n})\Big]^{\frac{1}{p}}dt
    \\&+\sigma_{n,k}^{-1}\int_{-\frac{1}{2}\log(1-\frac{R^2}{\kappa \sigma_{n,k}^2})}^{\infty}e^{-t}\mathbb{E}\Big[\frac{1}{s^p}\Big|\mathbb{E}\big[W_k-W'_k|W_k\big]-
W_k\Big|^p\Big|\mathcal{U}(X_{1:n})\Big]^{1/p}dt
\\&+\int_{-\frac{1}{2}\log(1-\frac{R^2}{\kappa \sigma_{n,k}^2})}^{\infty}\frac{e^{-2t}\|H_1\|_p}{\sqrt{1-e^{-2t}}}\mathbb{E}\Big[\Big|\frac{1}{2s}\sigma_{n,k}^{-2}\mathbb{E}\big[( W'_k-W_k)^2|W_k\big]-1\Big|^p\Big|\mathcal{U}(X_{1:n})\Big]^{1/p}dt
\\&+\sum_{\ell\ge 3}\int_{-\frac{1}{2}\log(1-\frac{R^2}{\kappa \sigma_{n,k}^2})}^{\infty}\frac{e^{-\ell t}\|H_{\ell-1}\|_p}{s\ell!(\sqrt{1-e^{-2t}})^{\ell-1}}\sigma_{n,k}^{-\ell}\mathbb{E}\Big[\Big|\mathbb{E}\big[( W'_k-W_k)^\ell|W_k\big]\Big|^p\Big|\mathcal{U}(X_{1:n})\Big]^{1/p}dt.
\end{aligned} 
\end{equation}The desired result directly follows from Jensen inequality.
\end{proof}

In the following subsections, we present lemmas that establish bounds for each term in \Cref{lemma:Wp_bound_using_Bonis}.

\subsubsection{Bounding $(a_0)$}
\begin{lemma}\label{lemma:a0_bound}
Assume that the conditions of \Cref{lemma:Wp_bound_using_Bonis} hold.  Then the following bound holds for all $p\ge2$:
    \begin{align*}
        (a_0)&\le \mathcal{W}_{p}(\sigma_{n,k}^{-1}W_k,\mathcal{N}(0,1))(1-\sqrt{1-\frac{R^2}{\kappa\sigma_{n,k}^2}}) + \sqrt{2}\left(\frac{\Gamma((p+1) / 2)}{\Gamma(1 / 2)}\right)^{1 / p}\arccos{(\sqrt{1-\frac{R^2}{\kappa\sigma_{n,k}^2}})}
        \\&\le \sigma_{n,k}^{-1}\mathcal{S}^p_{\sigma,R}(n,k)\times(1-\sqrt{1-\frac{R^2}{\kappa\sigma_{n,k}^2}}) + \sqrt{2}\left(\frac{\Gamma((p+1) / 2)}{\Gamma(1 / 2)}\right)^{1 / p}\arccos{(\sqrt{1-\frac{R^2}{\kappa\sigma_{n,k}^2}})}.
    \end{align*} where $ \mathcal{S}^p_{\sigma,R}(n,k)$ is defined in \Cref{nelson2}.
\end{lemma}
We note that the Wasserstein distance in \Cref{lemma:a0_bound} is the traditional Wasserstein-p distance and that we derive a numerical upper bound $\mathcal{S}^p_{\sigma,R}(n,k)$ for it in \Cref{nelson2}. This is similar to the upper bound already derived in \cite{austern2023efficient}.
\begin{proof}
Note that by definition of the Wasserstein-$p$ distance, we know that for all $\varepsilon>0$, there exists a normal random variable $Z_{W_k}\sim\mathcal{N}(0,\sigma_{n,k}^2)$ such that
\begin{align}\label{true}\Big\|W_k-Z_{W_k}\Big\|_{p}\le\mathcal{W}_p(W_k,Z_{W_k})+\varepsilon.\end{align}
Therefore, we have
\begin{align*}
    &\Big\|e^{-t} \sigma_{n,k}^{-1}W_k-\frac{e^{-2t}}{\sqrt{1-e^{-2t}}}\tilde Z\Big\|_{p}\\
     = &\Big\|e^{-t} \sigma_{n,k}^{-1}W_k-e^{-t}\sigma_{n,k}^{-1}Z_{W_k}+e^{-t}\sigma_{n,k}^{-1}Z_{W_k}-\frac{e^{-2t}}{\sqrt{1-e^{-2t}}}\tilde Z\Big\|_{p}
     \\&\le  e^{-t} \sigma_{n,k}^{-1}\Big\|W_k-Z_{W_k}\Big\|_p+\Big\|e^{-t}Z_{W_k}\sigma_{n,k}^{-1}-\frac{e^{-2t}}{\sqrt{1-e^{-2t}}}\tilde Z\Big\|_{p}
     \\
    \le & e^{-t}\sigma_{n,k}^{-1}(W_{p}(\sigma_{n,k}^{-1}W_k,Z_{W_k})+\varepsilon)+\frac{e^{-t}}{\sqrt{1-e^{-2t}}}\|\tilde{Z}\|_{p}.
\end{align*}
where the last inequality follows from \cref{true} combined with  the fact that $Z_{W_k}$ and $\tilde{Z}$ are independent. As $\varepsilon>0$ is arbitrary, it follows that \begin{align*}
    &\Big\|e^{-t} \sigma_{n,k}^{-1}W_k-\frac{e^{-2t}}{\sqrt{1-e^{-2t}}}\tilde Z\Big\|_{p}\\
    \le & e^{-t}\sigma_{n,k}^{-1}W_{p}(W_k,Z_{W_k})+\frac{e^{-t}}{\sqrt{1-e^{-2t}}}\|\tilde{Z}\|_{p}.
\end{align*}
Therefore, we obtain that 
\begin{equation*}
\begin{aligned}
    (a_0)=&  \int_0^{-\frac{1}{2}\log(1-\frac{R^2}{\kappa \sigma_{n,k}^2})} \mathbb{E}\Big[\Big|e^{-t} \sigma_{n,k}^{-1}W_k-\frac{e^{-2t}}{\sqrt{1-e^{-2t}}}\tilde Z\Big|^p\Big]^{\frac{1}{p}}dt\\
    \le&\sigma_{n,k}^{-1}\mathcal{W}_p(W_k,Z_{W_k})\int_0^{-\frac{1}{2}\log(1-\frac{R^2}{\kappa \sigma_{n,k}^2})} e^{-t}dt +\|\tilde{Z}\|_{p}\int_0^{-\frac{1}{2}\log(1-\frac{R^2}{\kappa \sigma_{n,k}^2})} \frac{e^{-t}}{\sqrt{1-e^{-2t}}} dt\\
    =: & (a_{0,1}) + (a_{0,2}).
\end{aligned}
\end{equation*}
We have
\begin{equation*}
    (a_{0.1}) =\sigma_{n,k}^{-1}\mathcal{W}_p(W_k,Z_{W_k})(1-{\sqrt{1-\frac{R^2}{\kappa\sigma_{n,k}^2}}}),
\end{equation*}
and
\begin{equation*}
    (a_{0,2}) = \|\tilde{Z}\|_{p}\Big(\frac{\pi}{2} - \arcsin{(e^{-\kappa R})}\Big)= \sqrt{2}\left(\frac{\Gamma((p+1) / 2)}{\Gamma(1 / 2)}\right)^{1 / p}\arccos{(\sqrt{1-\frac{R^2}{\kappa\sigma_{n,k}^2}})}.
\end{equation*}
Hence the result follows.
\end{proof}

\subsubsection{Bounding $(a_1)$}
Recall the notation $R_s := \frac{1}{2}(R + \sqrt{R^2 - 4\sigma^2})$, and that in \Cref{lemma:R_s_bound}, we proved $|X_1|\overset{a.s}{\le} R_s$. We also define $\tilde R_s:=\sigma^{-1}R_s.$

\begin{lemma}\label{lemma:Wk_exchangeable_pair}
  Assume that the conditions of \Cref{lemma:Wp_bound_using_Bonis} hold then we have  $$\frac{1}{s}\mathbb{E}[W'_k-W_k|W_k,\mathcal{U}(X_{1:n})]= -W_k.$$
\end{lemma}
\begin{proof} As $I\sim \Unif\{1,\dots,k\}$ and $J\sim \Unif\{k+1,\dots,n\}$, we have
    \begin{align*}&
        \frac{1}{s}\mathbb{E}\left[W'_k-W_k|\mathcal{U}(X_{1:k}),\mathcal{U}(X_{k+1:n})\right]\\=&\frac{k(n-k)}{n}\mathbb{E}\bigg[(\sum_{i\le k}X_{(I,J)(i)}-\frac{k}{n}S_n)-(\sum_{i\le k}X_{i}-\frac{k}{n}S_n)\bigg|\mathcal{U}(X_{1:k}),\mathcal{U}(X_{k+1:n})\bigg]\\
        \overset{(a)}{=}&\frac{k(n-k)}{n}\mathbb{E}\Big[X_{J}-X_{I}\Big|\mathcal{U}(X_{1:k}),\mathcal{U}(X_{k+1:n})\Big]\\
        =&\frac{k(n-k)}{n}\Big(\frac{1}{n-k}\sum_{k+1\le j\le n}X_{j} - \frac{1}{k}\sum_{i\le k}X_{i}\Big)\\
        =& \frac{k}{n}(S_n-\sum_{i\le k}X_{i})-\frac{n-k}{n}\sum_{i\le k}X_{i}\\
        =& -W_k.
    \end{align*}where (a) holds because $X_{(I,J)(i)}=X_{i}$ if $i\ne I,J$. The desired result is a direct consequence of the tower property.
\end{proof}

\begin{lemma}\label{lemma:a1_bound}
   Assume that the conditions of \Cref{lemma:Wp_bound_using_Bonis} hold. Then the following identity holds for all $p\ge2$:
    \begin{align*}
        (a_1) =0.
    \end{align*}
\end{lemma}
\begin{proof}
    Recall that 
    $$
    (a_1) = \sigma_{n,k}^{-1}\int_{-\frac{1}{2}\log(1-\frac{R^2}{\kappa \sigma_{n,k}^2})}^{\infty}e^{-t}\Big\|\frac{1}{s}\mathbb{E}\big[W_k-W_k'|W_k,\mathcal{U}(X_{1:n})\big]-
W_k\Big\|_pdt.
    $$
    Note that, by \Cref{lemma:Wk_exchangeable_pair}, we have$\Big\|\frac{1}{s}\mathbb{E}\big[W_k-W_k'|W_k,\mathcal{U}(X_{1:n})\big]-
W_k\Big\|_p = 0$. The result hence follows. 
\end{proof}

\subsubsection{Bounding $(a_2)$}
\begin{lemma}\label{lemma:a1_integrand_Wk_exchgb_bound}
 Assume that the assumptions of \Cref{lemma:Wp_bound_using_Bonis} hold. Then the following upper-bounds holds:
   \begin{align*}
   & \Big\|\frac{1}{2s}\sigma_{n,k}^{-2}\mathbb{E}\Big[(W'_k-W_k)^2\Big|W_k\Big]-1\Big\|_p\\
         &\le  \frac{\sqrt{n}}{\sqrt{k(n-k)}}\Big(\frac{1}{2}\min\Big\{  \sqrt{p-1}(\tilde R^2_s-1)^{1-1/p},A_p\sqrt{\tilde R^2_s-1}+A^*_{n,p}(\tilde R^2_s-1)^{1-1/p} \Big\}\\&\qquad\qquad+\frac{1}{\sqrt{n}}\min\Big\{\sqrt{p-1}\tilde R_s^{1-2/p},A_p +\tilde R_s^{1-2/p}A^*_{k,p}\Big\}^2\Big).  
    \end{align*}
\end{lemma}
\begin{proof}
By definition of $W'_k$, and recall that $s=\frac{n}{k(n-k)}$, we remark that 
\begin{align*}
      &\frac{1}{2s}\mathbb{E}\Big[(W'_k-W_k)^2\Big|\mathcal{U}(X_{1:k}),\mathcal{U}(X_{k+1:n})\Big]
        \\   =&\frac{k(n-k)}{2n}\mathbb{E}\Big[(X_{J}-X_{I})^2\Big|\mathcal{U}(X_{1:k}),\mathcal{U}(X_{k+1:n})\Big]
        \\
        =&\frac{k(n-k)}{2n}\Big(\frac{1}{k(n-k)}\sum_{\substack{1\le i \le k\\k<j\le n}}(X_{j}-X_{i})^2\Big)\\
        =&\frac{1}{2n}\sum_{\substack{1\le i \le k\\k<j\le n}}(X_{j}^2+X_{i}^2-2X_{j}X_{i})\\
        =& \frac{1}{2n}\Big((n-k)\sum_{1\le i\le k}X_{i}^2+k\sum_{k< j\le n}X_{j}^2-2\sum_{\substack{1\le i \le k\\k<j\le n}}X_{j}X_{i}\Big).
    \end{align*} 
    Hence we have
    \begin{align*}
    & \Big\|  \frac{1}{2n}\Big((n-k)\sum_{1\le i\le k}X_{i}^2 + k\sum_{k\le i\le n}X_i^2-2\sum_{\substack{1\le i \le k\\k<j\le n}}X_{j}X_{i}\Big)-\sigma_{n,k}^2\Big\|_p
    \\\le&   \Big\|\frac{1}{2n}\Big((n-k)\sum_{1\le i\le k}X_{i}^2 + k\sum_{k\le i\le n}X_i^2-\sigma_{n,k}^2\Big\|_p+\frac{1}{n}\Big\|\sum_{\substack{1\le i \le k\\k<j\le n}}X_{j}X_{i}\Big\|_p
    \\=& \Big\|  \frac{1}{2n}\Big((n-k)\sum_{1\le i\le k}(X_{i}^2-\sigma^2) + k\sum_{k<i\le n}(X_i^2-\sigma^2)\Big)\Big\|_p+\frac{1}{n}\Big\|\sum_{\substack{1\le i \le k\\k<j\le n}}X_{j}X_{i}\Big\|_p.
    \end{align*}
We will bound each terms on the right hand side successively. Firstly, we note that $$\|X_{i}^2-\sigma^2\|_p\le (R^2_s-\sigma^2)^{1-1/p}~  \textrm{for $p\ge2$.}$$
Hence according to \Cref{lemma:rosenthal_indep_nonidentical}, the following inequality holds:
\begin{align*}& \Big\|  \frac{1}{2n}\Big((n-k)\sum_{1\le i\le k}(X_{i}^2-\sigma^2) + k\sum_{k<i\le n}(X_i^2-\sigma^2)\Big)\Big\|_p
    \\&\le \frac{\sqrt{k(n-k)}}{2\sqrt{n}}\min\begin{cases}
        \sqrt{p-1}\sigma^{2/p}(R^2_s-\sigma^2)^{1-1/p}
        \\A_p\sigma\sqrt{R^2_s-\sigma^2}+A^*_{n,p}\sigma^{2/p}(R^2_s-\sigma^2)^{1-1/p}.
    \end{cases}
    \end{align*}
where recall that $A_p:=2^{1 / p} \sqrt{p / 2+1} e^{\frac{1}{2}+\frac{1}{p}}$ and $A_{n,p}^{*}:= \left(\frac{p}{2}+1\right) n^{1 / p-1 / 2}$.

Now we need to upper bound $\|\sum_{\substack{1\le i \le k\\k<j\le n}}X_{j}X_{i}\|_p$. We remark that as $(X_1,\dots,X_k)$ and $(X_{k+1},\dots,X_n)$ are independent, using  according to Lemma 11 and Lemma 14 in \cite{austern2023efficient}, the following holds \begin{align*}
\sigma^{-2}\Big\|\sum_{\substack{1\le i \le k\\k<j\le n}}X_{j}X_{i}\Big\|_p&=\Big\|\sum_{\substack{1\le i \le k}}X_{i}/\sigma\Big\|_p\Big\|\sum_{\substack{k+1\le i \le n}}X_{i}/\sigma\Big\|_p   
\\&\le \Big\{\sqrt{k}
\min\begin{cases}
    \sqrt{p-1}\tilde R_s^{1-2/p}\\A_p +\tilde R_s^{1-2/p}A^*_{k,p}
\end{cases}\Big\}
\times  \Big\{\sqrt{n-k}\min\begin{cases}
    \sqrt{p-1}\tilde R_s^{1-2/p}\\A_p +\tilde R_s^{1-2/p}A^*_{n-k,p}\end{cases}\Big\}
    \\&\le \sqrt{k(n-k)}\min\Big\{\sqrt{p-1}\tilde R_s^{1-2/p},A_p +\tilde R_s^{1-2/p}A^*_{k,p}\Big\}^2.  \end{align*}
    
    Therefore, if we combine everything together we obtain that 
     \begin{align*}&
   \sigma^{-2} \Big\|\frac{1}{2s}\mathbb{E}\Big[(W'_k-W_k)^2\Big|W_k\Big]-\sigma_{n,k}^{2}\Big\|_p
  \\  \le &  \frac{\sqrt{k(n-k)}}{\sqrt{n}}\Big(\frac{1}{2}\min\Big\{  \sqrt{p-1}(\tilde R^2_s-1)^{1-1/p},A_p\sqrt{\tilde R^2_s-1}+A^*_{n,p}(\tilde R^2_s-1)^{1-1/p} \Big\}\\&\qquad\qquad+\frac{1}{\sqrt{n}}\min\Big\{\sqrt{p-1}\tilde R_s^{1-2/p},A_p +\tilde R_s^{1-2/p}A^*_{k,p}\Big\}^2\Big).  
    \end{align*}
    This immediately implies the desired result by Jensen's inequality.
\end{proof}

\begin{lemma}\label{lemma:a2_bound}
    The following equality holds for all $p\ge2$:
    \begin{align*}
        (a_2)\le &\frac{\sqrt{n}\|H_1\|_p}{\sqrt{k(n-k)}}\Big[\frac{1}{2}\min\Big\{  \sqrt{p-1}(\tilde R^2_s-1)^{1-1/p},A_p\sqrt{\tilde R^2_s-1}+A^*_{n,p}(\tilde R^2_s-1)^{1-1/p} \Big\}\\&\qquad\qquad+\frac{1}{\sqrt{n}}\min\Big\{\sqrt{p-1}\tilde R_s^{1-2/p},A_p +\tilde R_s^{1-2/p}A^*_{k,p}\Big\}^2\Big]\Big(1 - \frac{R}{\sqrt{\kappa}\sigma_{n,k}}\Big). 
    \end{align*}
\end{lemma}
\begin{proof}
Recall that 
$$
    (a_2) = \int_{-\frac{1}{2}\log(1-\frac{R^2}{\kappa \sigma_{n,k}^2})}^{\infty}\frac{e^{-2t}\|H_1\|_p}{\sqrt{1-e^{-2t}}}\Big\|\frac{1}{2s}\sigma_{n,k}^{-2}\mathbb{E}\big[(W'_k-W_k)^2|W_k,\mathcal{U}(X_{1:n})\big]-1\Big\|_pdt.
$$
By \Cref{lemma:a1_integrand_Wk_exchgb_bound}, we have
\begin{align*}
    (a_2) 
    \le& \frac{\sqrt{n}}{\sqrt{k(n-k)}}\Big[\frac{1}{2}\min\Big\{  \sqrt{p-1}(\tilde R^2_s-1)^{1-1/p},A_p\sqrt{\tilde R^2_s-1}+A^*_{n,p}(\tilde R^2_s-1)^{1-1/p} \Big\}\\&+\frac{1}{\sqrt{n}}\min\Big\{\sqrt{p-1}\tilde R_s^{1-2/p},A_p +\tilde R_s^{1-2/p}A^*_{k,p}\Big\}^2\Big]\|H_1\|_p\int_{-\frac{1}{2}\log(1-\frac{R^2}{\kappa \sigma_{n,k}^2})}^{\infty}\frac{e^{-2t}}{\sqrt{1-e^{-2t}}}dt\\
    \overset{}{\le} & \|H_1\|_p\frac{\sqrt{n}}{\sqrt{k(n-k)}}\Big[\frac{1}{2}\min\Big\{  \sqrt{p-1}(\tilde R^2_s-1)^{1-1/p},A_p\sqrt{\tilde R^2_s-1}+A^*_{n,p}(\tilde R^2_s-1)^{1-1/p} \Big\}\\&\qquad\qquad+\frac{1}{\sqrt{n}}\min\Big\{\sqrt{p-1}\tilde R_s^{1-2/p},A_p +\tilde R_s^{1-2/p}A^*_{k,p}\Big\}^2\Big]\Big(1 -\frac{R}{\sqrt{\kappa}\sigma_{n,k} }\Big).  
\end{align*}
\end{proof}

\subsubsection{Bounding $(a_3)$}

\begin{lemma}\label{lemma:odd_Lp_term_bound}
    For all odd integer $\ell\ge 3$, the following inequality holds for $p>2$:
\begin{align*}
    &
    s^{-1}\sigma_{n,k}^{-\ell}\Big\|\mathbb{E}[(W_k-W_k')^\ell|W_k,\mathcal{U}(X_{1:n})]\Big\|_p
      \\\le & \frac{1}{n}\sigma_{n,k}^{-\ell}\sqrt{p-1}\sqrt{k(n-k)}R^{\ell-\frac{4}{p}} \sigma^{2/p}\bigg( (n-k)(R^2+3\sigma^2)^{\frac{2}{p}}\\& +k\min\Big\{(2^{\frac{1}{p}}R^{\frac{2}{p}} (R^2+3\sigma^2)^{\frac{1}{p}})^2,\Big((R^2+3\sigma^2)^{\frac{1}{p}}~+\frac{R^{\frac{4}{p}}}{\sqrt{k}\sigma^{2/p}}\min\{
    \sqrt{p-1},A_p+A^*_{k,p}\}\Big)^2 \Big\}\bigg)^{1/2}.
\end{align*}
Moreover, for $p\ge 4$, the following inequality also holds:
\begin{align*}
&s^{-1}\sigma_{n,k}^{-\ell}\Big\|\mathbb{E}[(W_k-W_k')^\ell|W_k,\mathcal{U}(X_{1:n})]\Big\|_p\\
\le& \frac{1}{n}\sigma_{n,k}^{-\ell}C_p\bigg\{R^{\ell-2}\sqrt{(n-k)k}\sqrt{(n-k)\sigma^2(R^2+3\sigma^2)+k{(R^2\sigma^2+R^{4-\frac{4}{p}})}}\\&+k^{\frac{1}{p}}(n-k)^{\frac{1}{p}}R^{\ell-\frac{4}{p}}\bigg(k^{p-1} \Big(\sigma^{2/p}(R^2+3\sigma^2)^{1/p}+2^{1/p}k^{-\frac{1}{2}}R^{2/p}\Big)^p\\&\qquad \qquad \qquad \qquad\qquad +(n-k)^{p-1}\sigma^2(R^2+3\sigma^2)\bigg)^{\frac{1}{p}}\bigg\},\end{align*}
where $C_p:=2\sqrt{2}\Big(\frac{p}{4}+1\Big)^{\frac{1}{p}}(1+\frac{p}{\log(p/2)})$.
\end{lemma}
\begin{proof}
First, note that by the definition of $W_k'$, the difference $W_k - W_k'$ is equal to $X_{I} - X_{J}$. Since $I \sim \mathrm{Unif}\{1, \dots, k\}$ and $J \sim \mathrm{Unif}\{k+1, \dots, n\}$, we have $\mathbb{P}(I=i, J=j)= \frac{1}{k(n-k)}$ for any $i \leq k$ and $k<j\leq n$. Consequently, the following inequality holds:
\begin{align*}
    s^{-1}\sigma_{n,k}^{-\ell}\Big\|\mathbb{E}[(W_k-W_k')^\ell|W_k,\mathcal{U}(X_{1:n})]\Big\|_p
\le&  s^{-1}\frac{1}{k(n-k)}\sigma_{n,k}^{-\ell}\Big\|\sum_{\substack{i\le k\\k<j\le n}}(X_{i}-X_{j})^\ell\Big\|_p
    \\\le&  \frac{1}{n}\sigma_{n,k}^{-\ell}\Big\|\sum_{\substack{i\le k\\k<j\le n}}(X_{i}-X_{j})^\ell\Big\|_p
    \\=&    \frac{1}{n}\sigma_{n,k}^{-\ell}\Big\|\sum_{\substack{i\le k\\k<j\le n}}(Y_{i}-Y_{j})^\ell\Big\|_p.
\end{align*}
To bound the right hand side, we define the following filtration $$\mathcal{F}_m:=\sigma(Y_1,\dots,Y_m),$$
and define $(D_m)_m$ to be the following martingale difference $$D_m=\mathbb{E}\Big[\sum_{\substack{i\le k\\k<j\le n}}(Y_{i}-Y_{j})^\ell\Big|\mathcal{F}_m\Big]-\mathbb{E}\Big[\sum_{\substack{i\le k\\k<j\le n}}(Y_{i}-Y_{j})^\ell\Big|\mathcal{F}_{m-1}\Big].$$ We remark that, as $\ell$ is odd, it follows that $\sum_{\substack{i\le k\\k<j\le n}}\mathbb{E}\big[(Y_{i}-Y_{j})^\ell\big]=0$. Hence by a telescopic sum argument, we notice that   $$\sum_{\substack{i\le k\\k<j\le n}}(Y_{i}-Y_{j})^\ell=\sum_{m\le n}D_m.$$
 Now using the Marcinkiewicz-Zygmund inequality (\Cref{lemma:MZineq_rio}), we obtain that 
 $$\Big\|\sum_{\substack{i\le k\\k<j\le n}}(Y_{i}-Y_{j})^\ell\Big\|_p\le \sqrt{p-1}\sqrt{\sum_{m=1}^n\|D_m\|_p^2}.$$
 Moreover according to Theorem 1 of \cite{osekowski2012note}, for $p\ge 4$ we have:
 \begin{align}
     \Big\|\sum_{\substack{i\le k\\k<j\le n}}(Y_{i}-Y_{j})^\ell\Big\|_p&\le C_p\Big(\mathbb{E}\Big[\Big(\sum_{m\le n}\mathbb{E}[D_m^2|\mathcal{F}_{m-1}]\Big)^{\frac{p}{2}}\Big]^{\frac{1}{p}}+\mathbb{E}\Big[\sum_{m\le n}D_m^p\Big]^{\frac{1}{p}}\Big)\\&\le C_p\Big(\sqrt{\sum_{m\le n}\Big\|\mathbb{E}[D_m^2|\mathcal{F}_{m-1}]\Big\|_{p/2}}+\Big(\sum_{m\le n}\|D_m\|_p^p\Big)^{\frac{1}{p}}\Big).
 \end{align} where $C_p:=2\sqrt{2}\Big(\frac{p}{4}+1\Big)^{\frac{1}{p}}(1+\frac{p}{\log(p/2)})$.

Each one of those inequalities will be further bounded. This is done by getting a bound on both $\|D_m\|_p$ and $\mathbb{E}[D_m^2|\mathcal{F}_{m-1}]$ for all $m\le n.$ 

In this goal we first notice that for all $m\le k$, by triangle inequality, we have\begin{align*}
    \|D_m\|_p&\le \Big\|\sum_{j> k}\mathbb{E}[(Y_m-Y_j)^\ell|Y_m]\Big\|_p
    \\
    &\le \sum_{j> k}\Big\|\mathbb{E}[(Y_m-Y_j)^\ell|Y_m]\Big\|_p
    \\&= (n-k)\Big\|\mathbb{E}[(Y_m-Y_n)^\ell|Y_m]\Big\|_p,
 \end{align*} where the last equality is due to the fact that $(Y_i)$ are identically distributed.
Moreover this can be further bounded by noticing that  
\begin{align*}
  \Big\|\mathbb{E}[(Y_m-Y_n)^\ell|\mathcal{F}_m]\Big\|_p
      &\overset{(a)}{\le }
         R^{\ell-\frac{4}{p}}\mathbb{E}\Big[\mathbb{E}[(Y_m-Y_n)^2|Y_m]^2\Big]^{\frac{1}{p}}       \\&\overset{(b)}{\le }R^{\ell-\frac{4}{p}}\mathbb{E}\Big[(X_m^2+\sigma^2)^2\Big]^{\frac{1}{p}}
          \\&\overset{(c)}{\le }\sigma^{2/p}R^{\ell-\frac{4}{p}}(R^2+3\sigma^2)^{\frac{1}{p}},
\end{align*} 
where step (a) exploits the assumption that the random variables $(Y_i)$ are positive, ensuring $|Y_m-Y_n|\overset{a.s.}{\le} R$. Step $(b)$ holds because $$
\mathbb{E}[(Y_m-Y_n)^2|Y_m] = (Y_m-\mathbb{E}[Y_m])^2+ \mathbb{E}[(Y_n-\mathbb{E}[Y_m])^2|Y_m]= X_m^2+ \mathrm{Var}(Y_m) = X_m^2+\sigma^2.$$ Step $(c)$ is due to the fact that $$\mathbb{E}[(X_m^2+\sigma^2)^2]=\mathbb{E}[X_m^4+2\sigma^2 X_m^2+\sigma^4]\le R^2\mathbb{E}[X_m^2]+2\sigma^2\mathrm{Var}(X_m)+\sigma^4 =\sigma^2(  R^2+3\sigma^2).$$ Hence we obtain that for $m\le k$,
\begin{align}\label{freefood}
    \Big\|\mathbb{E}[(Y_m-Y_n)^\ell|Y_m]\Big\|_p\le R^{\ell-\frac{4}{p}} \sigma^{2/p}(R^2+3\sigma^2)^{\frac{1}{p}}
\end{align} and so that 
\begin{align}
    \Big\|D_m\Big\|_p\le (n-k)R^{\ell-\frac{4}{p}} \sigma^{2/p}(R^2+3\sigma^2)^{\frac{1}{p}}.
\end{align}

Now if $m> k$, we have 
\begin{align*}\|D_m\|_p&\le \Big\|\sum_{i\le k}(Y_i-Y_m)^\ell-\mathbb{E}[(Y_i-Y_m)^\ell|Y_i]\Big\|_p\\
    &\le k\Big\|(Y_1-Y_m)^\ell-\mathbb{E}[(Y_1-Y_m)^\ell|Y_1]\Big\|_p\\
    &\le k\Big(\Big\|(Y_1-Y_m)^\ell\Big\|_p+\Big\|\mathbb{E}[(Y_1-Y_m)^\ell|Y_1]\Big\|_p\Big)\\
    &\overset{(a)}{\le} kR^{\ell-2/p}\mathbb{E}[(Y_1-Y_m)^2]^\frac{1}{p}+k\|\mathbb{E}[(Y_i-Y_1)^\ell|Y_1]\Big\|_p
    \\&\overset{(b)}{\le}2^{\frac{1}{p}}k\sigma^{2/p}R^{\ell-2/p}+kR^{\ell-\frac{4}{p}}\sigma^{2/p}(R^2+3\sigma^2)^{\frac{1}{p}}\\
    &= k\sigma^{2/p}R^{\ell-\frac{4}{p}}\Big(2^{\frac{1}{p}}R^{\frac{2}{p}} + (R^2+3\sigma^2)^{\frac{1}{p}}\Big),
\end{align*}
where again to get $(a)$ we used the fact that $|Y_1-Y_m|\overset{a.s}{\le }R$ and to get $(b)$ we used the fact that $\mathbb{E}[(Y_1-Y_m)^2]=2\mathrm{Var}(X_1)=2\sigma^2$ and \cref{freefood}.
\\In addition, for $m>k$, the following bound also holds: 
\begin{align*}\|D_m\|_p&\le \Big\|\sum_{i\le k}(Y_i-Y_m)^\ell-\mathbb{E}[(Y_i-Y_m)^\ell|Y_i]\Big\|_p\\&\le k\|\mathbb{E}[(Y_i-Y_m)^\ell|Y_m]\|_p+\Big\|\sum_{i\le k}(Y_i-Y_m)^\ell-\mathbb{E}[(Y_i-Y_m)^\ell|Y_m]-\mathbb{E}[(Y_i-Y_m)^\ell|Y_i]\Big\|_p
\\&\le kR^{\ell-\frac{4}{p}}\sigma^{2/p}(R^2+3\sigma^2)^{\frac{1}{p}}+\Big\|\sum_{i\le k}(Y_i-Y_m)^\ell-\mathbb{E}[(Y_i-Y_m)^\ell|Y_m]-\mathbb{E}[(Y_i-Y_m)^\ell|Y_i]\Big\|_p .
\end{align*}
Note that conditionally on $Y_m$, the random variables $(Y_i-Y_m)^\ell-\mathbb{E}[(Y_i-Y_m)^\ell|Y_m]-\mathbb{E}[(Y_i-Y_m)^\ell|Y_i]$, $i=1, \dots, k$ are centered i.i.d. random variables. Hence we obtain that got $m>k$,
\begin{align*}\|D_m\|_p&
\le kR^{\ell-\frac{4}{p}}\sigma^{2/p}(R^2+3\sigma^2)^{\frac{1}{p}}+\Big\|\sum_{i\le k}(Y_i-Y_m)^\ell-\mathbb{E}[(Y_i-Y_m)^\ell|Y_m]-\mathbb{E}[(Y_i-Y_m)^\ell|Y_i]\Big\|_p \\&\le kR^{\ell-\frac{4}{p}}\sigma^{2/p}(R^2+3\sigma^2)^{\frac{1}{p}}+\Big\|\sum_{i\le k}\mathbb{E}\big[(Y_i-Y_m)^\ell-(Y_{n+1}-Y_m)^\ell-(Y_i-Y_{n+1})^\ell|Y_i,Y_m]\Big\|_p  \\&\overset{(a)}{\le} kR^{\ell-\frac{4}{p}}\sigma^{2/p}(R^2+3\sigma^2)^{\frac{1}{p}}+\sqrt{k}R^{\ell}\min\{
    \sqrt{p-1},A_p+A^*_{k,p}\},
\end{align*} 
where to obtain $(a)$ we used \Cref{lemma:rosenthal_indep_nonidentical} together with the fact that, since $Y_i-Y_m=(Y_i-Y_{n+1})+(Y_{n+1}-Y_m)$,  for $\ell$ odd, the following inequality holds:
\begin{align*}
    \Big|(Y_i-Y_m)^\ell-(Y_{n+1}-Y_m)^\ell-(Y_i-Y_{n+1})^\ell\Big|
    &\le \max\{|Y_i-Y_m|^\ell, |Y_{n+1}-Y_m|^\ell, |Y_i-Y_{n+1}|^\ell\} \\&\le R^\ell.
\end{align*}
Hence by Theorem 2.1 of \cite{rio2009moment} (\Cref{lemma:MZineq_rio}), we obtain that
\begin{align*}
    &\Big\|\sum_{\substack{i\le k\\k<j\le n}}(Y_{i}-Y_{j})^\ell\Big\|_p\\
    \le& \sqrt{p-1}\sqrt{k(n-k)}R^{\ell-\frac{4}{p}} \sigma^{2/p}\Big( (n-k)(R^2+3\sigma^2)^{\frac{2}{p}}\\
    &~+k\min\Big\{(2^{\frac{1}{p}}R^{\frac{2}{p}} (R^2+3\sigma^2)^{\frac{1}{p}})^2, ~\Big((R^2+3\sigma^2)^{\frac{1}{p}}+\frac{R^{\frac{4}{p}}}{\sqrt{k}\sigma^{2/p}}\min\{
    \sqrt{p-1}, A_p+A^*_{k,p}\}\Big)^2 \Big\}\Big)^{1/2}.
\end{align*}
If $p\ge 4$ we derive an alternative bound using \cite{osekowski2012note}.
In this goal we remark that if $m\le k$ we have\begin{align*}
    \mathbb{E}[D_m^2|\mathcal{F}_{m-1}]&=(n-k)^2\mathbb{E}[\mathbb{E}[(Y_m-Y_n)^\ell|Y_m]^2]\\&\le (n-k)^2R^{2\ell-4}\mathbb{E}\big[\mathbb{E}[(Y_m-Y_n)^2|Y_m]^2\big]\\&\le (n-k)^2R^{2\ell-4}\mathbb{E}[(\sigma^2+X_m^2)^2]
    \\&\le (n-k)^2R^{2\ell-4}\sigma^2(R^2+3\sigma^2). 
\end{align*} which implies that \begin{align}
    \| \mathbb{E}[D_m^2|\mathcal{F}_{m-1}]\|_{p/2}&\le(n-k)^2R^{2\ell-4}\sigma^2(R^2+3\sigma^2).
\end{align}

Moreover, if, instead, we have $m>k$, then we have \begin{align*}
   \mathbb{E}[D_m^2|\mathcal{F}_{m-1}]&\le \mathbb{E}\Big[\Big(\sum_{i\le k}(Y_i-Y_m)^\ell-\mathbb{E}[(Y_i-Y_m)^\ell|Y_i]\Big)^2\Big|\mathcal{F}_{m-1}\Big]\\
   &\le k\sum_{i\le k}\mathbb{E}\Big[\Big((Y_i-Y_m)^{\ell}-\mathbb{E}[(Y_i-Y_m)^{\ell}|Y_i]\Big)^2\Big|Y_i\Big]\\
   &=k\sum_{i\le k}\mathbb{E}\Big[(Y_i-Y_m)^{2\ell}|Y_i]-\mathbb{E}[(Y_i-Y_m)^{\ell}|Y_i]^2\\&\le kR^{2\ell-2}\sum_{i\le k}\mathbb{E}\Big[(Y_i-Y_m)^{2}\Big|Y_i\Big] \\&\le  kR^{2\ell-2}\sum_{i\le k}(\sigma^2+X_i^2)\\&\le  R^{2\ell-2}(k\sigma^2+\sum_{i\le k}X_i^2).
\end{align*}Hence we obtain that \begin{align*}
  \| \mathbb{E}[D_m^2|\mathcal{F}_{m-1}]\|_{p/2}&\le R^{2\ell-2}k^2\sigma^2+k^2R^{2\ell-4/p}\sigma^{4/p}.
\end{align*} 
Hence it follows that  for $p\ge 4$, the following inequality also holds:
\begin{align*}
&s^{-1}\sigma_{n,k}^{-\ell}\Big\|\mathbb{E}((W_k-W_k')^\ell|W_k,\mathcal{U}(X_{1:n}))\Big\|_p\\
\le& \frac{1}{n}\sigma_{n,k}^{-\ell}C_p\bigg\{R^{\ell-2}\sqrt{(n-k)k}\sqrt{(n-k)\sigma^2(R^2+3\sigma^2)+k{(R^2\sigma^2+R^{4-\frac{4}{p}})}}\\&+k^{\frac{1}{p}}(n-k)^{\frac{1}{p}}R^{\ell-\frac{4}{p}}\bigg[k^{p-1} \Big(\sigma^{2/p}(R^2+3\sigma^2)^{1/p}+2^{1/p}k^{-\frac{1}{2}}R^{2/p}\Big)^p\\&\qquad \qquad \qquad \qquad\qquad +(n-k)^{p-1}\sigma^2(R^2+3\sigma^2)\bigg]^{\frac{1}{p}}\bigg\}.\end{align*}

\end{proof}
\begin{lemma}\label{lemma:even_Lp_term_bound} For all even $\ell\ge 4$ the following holds for $p>2$:
\begin{align*}
    &s^{-1}\sigma_{n,k}^{-\ell}\Big\|\mathbb{E}[(W_k-W_k')^\ell|W_k,\mathcal{U}(X_{1:n})]\Big\|_p
      \\\le&  \frac{1}{n}\sigma_{n,k}^{-\ell}{\sqrt{k(n-k)}}R^{\ell-2}\bigg[\sqrt{p-1}R^{2-\frac{4}{p}}\Big((n-k)\sigma^{4/p}(R^2+3\sigma^2)^\frac{2}{p}
   \\&\quad\quad\quad\quad\quad\quad\quad\quad\quad\quad\quad\quad+kR^{\frac{2}{p}}\min\Big\{(\frac{1}{4})^{\frac{1}{p}}R^{\frac{2}{p}},2^{\frac{1}{p}}\sigma^{2/p}+ R^{-2/p}\sigma^{4/p}(R^2+3\sigma^2)^{\frac{1}{p}}\Big\}^2\Big)^{\frac{1}{2}}\\
      &\quad\quad\quad\quad\quad\quad\quad\quad+{2\sqrt{k(n-k)}\sigma^2}\bigg].
\end{align*}
Moreover, for $p\ge 4$, the following inequality also holds:
\begin{align*}
&s^{-1}\sigma_{n,k}^{-\ell}\Big\|\mathbb{E}[(W_k-W_k')^\ell|W_k,\mathcal{U}(X_{1:n})]\Big\|_p\\
&\le \frac{1}{n}\sigma_{n,k}^{-\ell}C_p\Big\{R^{\ell-2}\sqrt{(n-k)k}\sqrt{2(n-k)\sigma^{4/p}(R^2+3\sigma^2)^{\frac{2}{p}}+ (n-k)R^{2} +k(\sigma^2+\sigma^{4/p}R^{2-\frac{4}{p}})}\\&+k^{\frac{1}{p}}(n-k)^{\frac{1}{p}}R^{\ell-\frac{4}{p}}\Big((n-k)^{p-1}\sigma^2(R^2+3\sigma^2)\\
&+k^{p-1}\min\big\{\frac{R^4}{4},R^2\sigma^2(2^{1/p}+(R^2+3\sigma^2)^{1/p})^p\big\}\Big)^{1/p}\Big\}\\
&+ \frac{2}{n}\sigma_{n,k}^{-\ell}k(n-k)R^{\ell-2}\sigma^2.\end{align*}
Alternatively the following bound also holds
\begin{align*}
    &
    s^{-1}\sigma_{n,k}^{-\ell}\|\mathbb{E}((W_k-W_k')^\ell|W_k,\mathcal{U}(X_{1:n}))\|_p
      \\ \le& \frac{\sigma_{n,k}^{-\ell}}{n}\Big\{k R^{\ell}\|\mathrm{binom}(k,\frac{2\sigma^2}{R^2})\|_p+(n-k)k2^{\frac{1}{p}}R^{\ell-2/p}\sigma^{2/p}\Big\}.
\end{align*}
\end{lemma}
\begin{proof}
    First, note that by the definition of $W_k'$, the following holds
\begin{align*}&
    s^{-1}\sigma_{n,k}^{-\ell}\|\mathbb{E}((W_k-W_k')^\ell|W_k,\mathcal{U}(X_{1:n}))\|_p
\\\le&  s^{-1}\frac{1}{k(n-k)}\sigma_{n,k}^{-\ell}\Big\|\sum_{\substack{i\le k\\k<j\le n}}(X_{i}-X_{j})^\ell\Big\|_p
    \\\le&  \frac{1}{n}\sigma_{n,k}^{-\ell}\Big\|\sum_{\substack{i\le k\\k<j\le n}}(X_{i}-X_{j})^\ell\Big\|_p
\\=&\frac{1}{n}\sigma_{n,k}^{-\ell}\Big\|\sum_{\substack{i\le k\\k<j\le n}}(Y_{i}-Y_{j})^\ell\Big\|_p.
\end{align*}
As $\ell$ is even, we have
\begin{align*}
    \Big\|\sum_{\substack{i\le k\\k<j\le n}} (Y_i-Y_j)^\ell\Big\|_p\le&\Big\|\sum_{\substack{i\le k\\k<j\le n}}\Big( (Y_i-Y_j)^\ell-\mathbb{E}[(Y_i-Y_j)^\ell]\Big)\Big\|_p+k(n-k) \mathbb{E}[(Y_i-Y_j)^\ell]
    \\\le& \Big\|\sum_{\substack{i\le k\\k<j\le n}} \Big((Y_i-Y_j)^\ell-\mathbb{E}[(Y_i-Y_j)^\ell]\Big)\Big\|_p+2k(n-k)R^{\ell-2}\sigma^2,
\end{align*}
where $\|\sum_{\substack{i\le k\\k<j\le n}} ((Y_i-Y_j)^\ell-\mathbb{E}[(Y_i-Y_j)^\ell])\|_p$ can be again bounded using the Marcinkiewicz-Zygmund type inequality and Theorem 1 of \cite{osekowski2012note}. 
Again, in this goal, we define the following filtration $$\mathcal{F}_i=\sigma(Y_1,\dots,Y_i),$$ and $$D_i=\mathbb{E}\Big[\sum_{\substack{m\le k\\k<j\le n}}(Y_{m}-Y_{j})^\ell\Big|\mathcal{F}_i\Big]-\mathbb{E}\Big[\sum_{\substack{m\le k\\k<j\le n}}(Y_{m}-Y_{j})^\ell\Big|\mathcal{F}_{i-1}\Big].$$ By a telescopic sum argument we notice that   $$\sum_{\substack{i\le k\\k<j\le n}} \Big((Y_i-Y_j)^\ell-\mathbb{E}[(Y_i-Y_j)^\ell]\Big)=\sum_{m\le n}D_m.$$For $i\le k$, we have
\begin{align*}
    \|D_i\|_p&\le \Big\|\sum_{j> k}\mathbb{E}[(Y_i-Y_j)^\ell|Y_i]-\mathbb{E}[(Y_i-Y_j)^\ell]\Big\|_p
    \\&\le (n-k)\Big\|\mathbb{E}[(Y_i-Y_n)^\ell|Y_i]-\mathbb{E}[(Y_n-Y_j)^\ell]\Big\|_p
      \\&= (n-k)\mathbb{E}\Big[\Big|\mathbb{E}[(Y_i-Y_n)^\ell|Y_i]-\mathbb{E}[(Y_n-Y_j)^\ell]\Big|^p\Big]^{\frac{1}{p}}
       \\&\overset{(a)}{\le} (n-k)R^{\ell -2\ell/p}\mathbb{E}\Big[\Big(\mathbb{E}[(Y_i-Y_n)^\ell|Y_i]-\mathbb{E}[(Y_n-Y_j)^\ell]\Big)^2\Big]^{\frac{1}{p}},
\end{align*}
where to get $(a)$ we used the fact that $|Y_i-Y_n|\overset{a.s}{\le }R$. 
Note that we have
\begin{align*}
  &\mathbb{E}\Big[\Big(\mathbb{E}[(Y_i-Y_n)^\ell|Y_i]-\mathbb{E}[(Y_n-Y_j)^\ell]\Big)^2\Big]\\
  =&\mathbb{E}\Big[\mathbb{E}[(Y_i-Y_n)^\ell|Y_i]^2\Big]-\mathbb{E}[(Y_n-Y_j)^\ell]^2\\
  \le& \mathbb{E}\Big[\mathbb{E}[(Y_i-Y_n)^\ell|Y_i]^2\Big]\\
  \le& R^{2\ell-4}\mathbb{E}\Big[\mathbb{E}[(Y_i-Y_n)^2|Y_i]^2\Big] \\
  \overset{(a)}{\le}& R^{2\ell-4}\sigma^2(R^2+3\sigma^2),
\end{align*}
where $(a)$ is because $\mathbb{E}\Big[\mathbb{E}[(Y_i-Y_n)^2|Y_i]^2\Big]\le \mathbb{E}[(X_i^2+\sigma^2)^2] \le \sigma^2(R^2+3\sigma^2)$. Therefore,
$$
\mathbb{E}\Big[\Big(\mathbb{E}[(Y_i-Y_n)^\ell|Y_i]-\mathbb{E}[(Y_n-Y_j)^\ell]\Big)^2\Big]^{\frac{1}{p}}\le R^{\ell-\frac{4}{p}} \sigma^{2/p}(R^2+3\sigma^2)^\frac{1}{p}.
$$
Hence it follows that for $i\le k$,
$$
\|D_i\|_p{\le}(n-k)R^{\ell-\frac{4}{p}}\sigma^{2/p}(R^2+3\sigma^2)^\frac{1}{p}.$$

Now  if $i> k$, by a triangle inequality argument we obtain that \begin{align*}
    \|D_i\|_p&\le \Big\|\sum_{j<k}(Y_i-Y_j)^\ell-\mathbb{E}[(Y_i-Y_j)^\ell|Y_j]\Big\|_p\\
    &\le k\Big\|(Y_i-Y_1)^\ell-\mathbb{E}[(Y_i-Y_1)^\ell|Y_1]\Big\|_p,
\end{align*}
where we have
\begin{align*}
&\Big\|(Y_i-Y_1)^\ell-\mathbb{E}[(Y_i-Y_1)^\ell|Y_1]\Big\|_p\\
    =&
    \mathbb{E}\Big[\Big|(Y_i-Y_1)^\ell-\mathbb{E}[(Y_i-Y_1)^\ell|Y_1]\Big|^p\Big]^{\frac{1}{p}}\\
    \overset{(a)}{\le}& R^{\ell-2\ell/p}\mathbb{E}\Big[\Big((Y_i-Y_1)^\ell-\mathbb{E}[(Y_i-Y_1)^\ell|Y_1]\Big)^2\Big]^{\frac{1}{p}}
    \\\overset{(b)}{\le}&(\frac{1}{4})^{\frac{1}{p}}R^{\ell},
\end{align*} where $(a)$ is by $|Y_i-Y_i|\le R$, and $(b)$ is by Popoviciu's inequality. 
Alternatively, we have
\begin{align*}
&\Big\|(Y_i-Y_1)^\ell-\mathbb{E}[(Y_i-Y_1)^\ell|Y_1]\Big\|_p\\
    \le&
    \Big\|(Y_i-Y_1)^\ell\Big\|_p+\Big\|\mathbb{E}[(Y_i-Y_1)^\ell|Y_1]\Big\|_p\\
    {\le}& 2^{\frac{1}{p}}\sigma^{2/p}R^{\ell-2/p}+R^{\ell-\frac{4}{p}}\sigma^{2/p}(R^2+3\sigma^2)^{\frac{1}{p}}.
\end{align*}
Hence it follows that, for $i>k$, 
$$
\|D_i\|_p{\le}kR^{\ell-\frac{2}{p}}\min\Big\{R^{\frac{2}{p}}/4^{\frac{1}{p}},2^{\frac{1}{p}}\sigma^{2/p}+R^{-2/p}\sigma^{2/p}(R^2+3\sigma^2)^{\frac{1}{p}}\Big\}.
$$
Hence by Theorem 2.1 of \cite{rio2009moment}, we obtain that
\begin{align*}
    &\Big\|\sum_{\substack{i\le k\\k<j\le n}} (Y_i-Y_j)^\ell-\mathbb{E}[(Y_i-Y_j)^\ell]\Big\|^2_p\\
    \le& (p-1)\Big(k\Big((n-k)R^{\ell-\frac{4}{p}}\sigma^{2/p}(R^2+3\sigma^2)^\frac{1}{p}\Big)^2\\&+ (n-k)\Big(kR^{\ell-\frac{2}{p}}\min\Big\{(\frac{1}{4})^{\frac{1}{p}}R^{\frac{2}{p}},2^{\frac{1}{p}}\sigma^{2/p}+R^{-2/p} \sigma^{2/p}(R^2+3\sigma^2)^{\frac{1}{p}}\Big\} \Big)^2\Big).
\end{align*}
Therefore, it follows that 
\begin{align*}
    s^{-1}\sigma_{n,k}^{-\ell}\|\mathbb{E}((W_k-W_k')^\ell&|W_k,\mathcal{U}(X_{1:n}))\|_p
      \\ \le \frac{\sqrt{k(n-k)}}{n}\sigma_{n,k}^{-\ell}R^{\ell-2}\Big[&\sqrt{p-1}R^{2-\frac{4}{p}}\Big((n-k)\sigma^{4/p}(R^2+3\sigma^2)^\frac{2}{p}
  \\&\quad\quad\quad\quad +kR^{\frac{2}{p}}\min\Big\{(\frac{1}{4})^{\frac{1}{p}}R^{\frac{2}{p}},2^{\frac{1}{p}}\sigma^{2/p}+ R^{-2/p}\sigma^{2/p}(R^2+3\sigma^2)^{\frac{1}{p}}\Big\}^2\Big)^{\frac{1}{2}}\\
      &+{2\sqrt{k(n-k)}\sigma^2}\Big].
\end{align*}

Moreover if $p\ge 4$ then we can alternatively use  Theorem 1 of \cite{osekowski2012note} to bound $\|\sum_{\substack{i\le k\\k<j\le n}} ((Y_i-Y_j)^\ell-\mathbb{E}[(Y_i-Y_j)^\ell])\|_p$. In this goal we first note that if $i\le k$ we have
\begin{align*}
    \mathbb{E}[D_i^2|\mathcal{F}_{i-1}]&=(n-k)^2\mathbb{E}\Big[\Big(\mathbb{E}[(Y_i-Y_n)^\ell|Y_i]-\mathbb{E}[(Y_i-Y_n)^\ell]\Big)^2\Big|\mathcal{F}_{i-1}\Big]\\
    &\overset{(a)}{\le}(n-k)^2\mathbb{E}\Big[\mathbb{E}[(Y_i-Y_n)^\ell|Y_i]^2+\mathbb{E}[(Y_i-Y_n)^\ell]^2\Big|\mathcal{F}_{i-1}\Big]
    \\&\le (n-k)^2R^{2\ell-4}\mathbb{E}[\mathbb{E}[(Y_i-Y_n)^2|Y_m]^2] + (n-k)^2R^{2\ell}
    \\&= (n-k)^2R^{2\ell-4}\mathbb{E}[(\sigma^2+X_i^2)^2]+ (n-k)^2R^{2\ell}
    \\&\le (n-k)^2R^{2\ell-4}\sigma^2(R^2+3\sigma^2)+ (n-k)^2R^{2\ell},
\end{align*} 
where $(a)$ is due to the fact that $(Y_i-Y_n)^\ell\ge 0.$ Hence we have that 
\begin{align*}
   \big\|\mathbb{E}[D_i^2|\mathcal{F}_{i-1}]\big\|_{p/2} \le (n-k)^2R^{2\ell-4}\sigma^2(R^2+3\sigma^2)+ (n-k)^2R^{2\ell}.
\end{align*} If $i>k$ we have 
\begin{align*}
   \mathbb{E}[D_i^2|\mathcal{F}_{i-1}]&\le \mathbb{E}\Big[\Big(\sum_{m\le k}(Y_m-Y_i)^\ell-\mathbb{E}[(Y_m-Y_i)^\ell|Y_i]\Big)^2\Big|\mathcal{F}_{i-1}\Big]\\
   &\le k\sum_{m\le k}\mathbb{E}\Big[\Big((Y_m-Y_i)^{\ell}-\mathbb{E}[(Y_m-Y_i)^{\ell}|Y_m]\Big)^2\Big|Y_m\Big]\\
   &=k\sum_{m\le k}\mathbb{E}\big[(Y_m-Y_i)^{2\ell}|Y_m]-\mathbb{E}[(Y_m-Y_i)^{\ell}|Y_m]^2\\&\le kR^{2\ell-2}\sum_{m\le k}\mathbb{E}\Big[(Y_m-Y_i)^{2}\Big|Y_m\Big] \\&\le  kR^{2\ell-2}\sum_{i\le k}(\sigma^2+X_i^2)\\&\le  kR^{2\ell-2}(k\sigma^2+\sum_{i\le k}X_i^2).
\end{align*}
Hence we obtain that \begin{align}
  \| \mathbb{E}[D_i^2|\mathcal{F}_{i-1}]\|_{p/2}&\le R^{2\ell-2}k^2\sigma^2+k^2R^{2\ell-4/p}\sigma^{4/p}.
\end{align}
Hence using Theorem 1 of \cite{osekowski2012note} the following alternative bound is also valid according to 
\Cref{lemma:ose_ineq}
\begin{align*}&\Big\|\sum_{m\le n}D_m\Big\|_p\\&\le C_p\Big\{R^{\ell-2}\sqrt{(n-k)k}\sqrt{2(n-k)\sigma^{4/p}(R^2+3\sigma^2)^{\frac{2}{p}}+ (n-k)R^{2} +k(\sigma^2+\sigma^{4/p}R^{2-\frac{4}{p}})}\\&+k^{\frac{1}{p}}(n-k)^{\frac{1}{p}}R^{\ell-\frac{4}{p}}\Big((n-k)^{p-1}\sigma^2(R^2+3\sigma^2)\\&+k^{p-1}\min\big\{\frac{R^4}{4},R^2\sigma^2(2^{1/p}+(R^2+3\sigma^2)^{1/p})^p\big\}\Big)^{1/p}\Big\}.\end{align*}
Therefore it follows, for $p\ge 4$, that the following inequality also holds:
\begin{align*}
&s^{-1}\sigma_{n,k}^{-\ell}\|\mathbb{E}((W_k-W_k')^\ell|W_k,\mathcal{U}(X_{1:n}))\|_p\\
\le&\frac{1}{n}\sigma_{n,k}^{-\ell}C_p\Big\{R^{\ell-2}\sqrt{(n-k)k}\sqrt{2(n-k)\sigma^{4/p}(R^2+3\sigma^2)^{\frac{2}{p}}+ (n-k)R^{2} +k(\sigma^2+\sigma^{4/p}R^{2-\frac{4}{p}})}\\&+k^{\frac{1}{p}}(n-k)^{\frac{1}{p}}R^{\ell-\frac{4}{p}}\Big((n-k)^{p-1}\sigma^2(R^2+3\sigma^2)\\&+k^{p-1}\min\big\{\frac{R^4}{4},R^2\sigma^2(2^{1/p}+(R^2+3\sigma^2)^{1/p})^p\big\}\Big)^{1/p}\Big\}\\&+\frac{2}{n}\sigma_{n,k}^{-\ell}k(n-k)R^{\ell-2}\sigma^2.\end{align*}
We hence established the first two bounds of \Cref{lemma:even_Lp_term_bound}. Moreover
  if $k\le n/2$ we remark that the following also holds
\begin{align*}
    \|\sum_{\substack{i\le k\\k<j\le n}} (Y_i-Y_j)^\ell\|_p
    &\le k\|\sum_{i\le k} (Y_i-Y_{k+i})^\ell\|_p+(n-2k)k\|(Y_i-Y_n)^l\|_p
       \\&\le k\|\sum_{i\le k} (Y_i-Y_{k+i})^\ell\|_p+(n-2k)k2^{\frac{1}{p}}\sigma^{2/p}R^{\ell-2/p}.
\end{align*}
Now note that the random variables $( (Y_i-Y_{k+i})^\ell)_{i\le k}$ are i.i.d. positive random variables.  Moreover we remark that $\mathbb{E}((Y_i-Y_{k+i})^\ell)\le2 R^{\ell-2}\sigma^2$ and $\mathbb{E}((Y_i-Y_{k+i})^{\ell p})\le2 R^{\ell p-2}\sigma^2$. Therefore using \Cref{lemma:moment-inequality} in \cite{austern2023efficient}, we obtain that $$\|\sum_{i\le k} (Y_i-Y_{k+i})^\ell\|_p\le R^{\ell}\|\mathrm{binom}(k,\frac{2\sigma^2}{R^2})\|_p.$$
Therefore, we have 
\begin{align*}
&s^{-1}\sigma_{n,k}^{-\ell}\|\mathbb{E}((W_k-W_k')^\ell|W_k,\mathcal{U}(X_{1:n}))\|_p \\\le& \frac{1}{n}\sigma_{n,k}^{-\ell}\Big(kR^{\ell}\|\mathrm{binom}(k,\frac{2\sigma^2}{R^2})\|_p + (n-k)k2^{\frac{1}{p}}]\sigma^{2/p}R^{\ell-2/p}\Big).  
\end{align*}
\end{proof}
\begin{lemma}\label{lemma:odd_integral_bound}
For every odd integer $\ell\ge 3$, let $S_{k,\kappa}:=\frac{R^2}{\kappa \sigma_{n,k}^2}$ and $\tilde S_k:=\frac{R^2n}{k\sigma^2}$. Then for any integer $K_p\ge 1$, the following holds:
\begin{align*}
&\sum_{\substack{\ell\ge 3\\\mathrm{odd}}}\sigma_{n,k}^{-\ell}R^{\ell}\int_{-\frac{1}{2}\log(1-\frac{R^2}{\kappa\sigma_{n,k}^2})}^{\infty}\frac{e^{-\ell t}\|H_{\ell-1}\|_p}{\ell!(\sqrt{1-e^{-2t}})^{\ell-1}}dt
\\\le & \frac{1}{2(n-k)}\sum_{K_p\ge m\ge1}\frac{1}{m!}\sigma_{n,k}^{-1}R\tilde S_k^m\Big(\frac{\|H_{2m}\|_pm!}{(2m+1)!}-\frac{e^{19/300}\pi^{1/4}(K+1)^{1/4}(p-1)^{m}}{2^m(2K+3)}\Big)\\&\qquad\qquad \times\int_{\frac{1}{(n-k)}}^{\frac{1}{(n-k)S_{k,\kappa}}}\frac{1}{y^{3/2}\sqrt{y-\frac{1}{n-k}}}(y-\frac{1}{n-k})^mdy
\\&+\frac{e^{19/300}R\sigma_{n,k}^{-1}\pi^{1/4}(K_p+1)^{1/4}}{2(2K_p+3)(n-k)}\int_{\frac{1}{(n-k)}}^{\frac{1}{(n-k)S_{k,\kappa}}}\frac{1}{y^{3/2}\sqrt{y-\frac{1}{n-k}}}\Big[e^{\frac{(p-1)\tilde S_k}{2}(y-\frac{1}{n-k})}-1\Big]dy.
\end{align*}
\end{lemma}
\begin{proof}
    First, we note that 
\begin{align*}
&\sum_{\substack{\ell\ge 3\\\mathrm{odd}}}\sigma_{n,k}^{-\ell}R^{\ell}\int_{-\frac{1}{2}\log(1-\frac{R^2}{\kappa \sigma_{n,k}^2})}^{\infty}\frac{e^{-\ell t}\|H_{\ell-1}\|_p}{\ell!(\sqrt{1-e^{-2t}})^{\ell-1}}dt
\\
\overset{(a)}{=} & \sum_{\substack{\ell\ge 3\\\mathrm{odd}}}\frac{1}{\ell!}\sigma_{n,k}^{-\ell}\|H_{l-1}\|_pR^{\ell}\int_{0}^{\sqrt{1-\frac{R^2}{\kappa \sigma_{n,k}^2}}}\frac{x^{\ell-1}}{(\sqrt{1-x^2})^{\ell-1}}dx\\
\overset{(b)}{=} & \sum_{m\ge1}\frac{1}{(2m+1)!}\sigma_{n,k}^{-(2m+1)}\|H_{2m}\|_pR^{2m+1}\int_{0}^{\sqrt{1-\frac{R^2}{\kappa \sigma_{n,k}^2}}}\frac{x^{2m}}{(1-x^2)^{m}}dx,
\end{align*}
where $(a)$ is obtained by using a change of variables $x=e^{-t}$ and $(b)$ by writing $\ell=2m+1$. Using the fact that, according to \cite{robbins1955remark}, the following holds: \begin{align*}
\sqrt{(2 m) !} & \geq \sqrt{\sqrt{2 \pi(2 m)} \cdot(2 m / e)^{2 m} \cdot \exp \left(\frac{1}{12(2 m)+1}\right)} \\
& \geq e^{-19 / 300} 2^m m ! /(\pi m)^{1 / 4},
\end{align*} and by using \Cref{lemma:Hermite_bound}, we obtain that for all $K_p\ge 1$, we have\begin{align*}
&\sum_{\substack{\ell\ge 3\\\mathrm{odd}}}\sigma_{n,k}^{-\ell}R^{\ell}\int_{-\frac{1}{2}\log(1-\frac{R^2}{\kappa \sigma_{n,k}^2})}^{\infty}\frac{e^{-\ell t}\|H_{\ell-1}\|_p}{\ell!(\sqrt{1-e^{-2t}})^{\ell-1}}dt\\
\le & \sum_{K_p\ge m\ge1}\frac{1}{(2m+1)!}\sigma_{n,k}^{-(2m+1)}\|H_{2m}\|_pR^{2m+1}\int_{0}^{\sqrt{1-\frac{R^2}{\kappa \sigma_{n,k}^2}}}\frac{x^{2m}}{(1-x^2)^{m}}dx
\\&+\sum_{ m\ge K_p+1}\frac{1}{(2m+1)\sqrt{(2m)!}}\sigma_{n,k}^{-(2m+1)}(p-1)^{m}R^{2m+1}\int_{0}^{\sqrt{1-\frac{R^2}{\kappa \sigma_{n,k}^2}}}\frac{x^{2m}}{(1-x^2)^{m}}dx
\\\le & \sum_{K_p\ge m\ge1}\frac{1}{(2m+1)!}\sigma_{n,k}^{-(2m+1)}\|H_{2m}\|_pR^{2m+1}\int_{0}^{\sqrt{1-\frac{R^2}{\kappa \sigma_{n,k}^2}}}\frac{x^{2m}}{(1-x^2)^{m}}dx
\\&+\sum_{ m\ge K_p+1}\frac{e^{19/300}\pi^{1/4}m^{1/4}}{2^mm!(2m+1)}\sqrt{\frac{n}{(n-k)k}}^{2m+1}(p-1)^{m}R^{2m+1}\int_{0}^{\sqrt{1-\frac{R^2}{\kappa \sigma_{n,k}^2}}}\frac{x^{2m}}{(1-x^2)^{m}}dx
\\\le & \sum_{K_p\ge m\ge1}\frac{1}{m!}\sigma_{n,k}^{-(2m+1)}R^{2m+1}\Big(\frac{\|H_{2m}\|_pm!}{(2m+1)!}-\frac{e^{19/300}\pi^{1/4}(K_p+1)^{1/4}(p-1)^{m}}{2^m(2K_p+3)}\Big)\\&\qquad\int_{0}^{\sqrt{1-\frac{R^2}{\kappa \sigma_{n,k}^2}}}\frac{x^{2m}}{(1-x^2)^{m}}dx
\\&+\frac{e^{19/300}R\sigma_{n,k}^{-1}\pi^{1/4}(K_p+1)^{1/4}}{(2K_p+3)}\int_{0}^{\sqrt{1-\frac{R^2}{\kappa \sigma_{n,k}^2}}}\Big[e^{\frac{x^2(p-1)R^2}{2(1-x^2)\sigma_{n,k}^2}}-1\Big]dx.
\end{align*}
Now by change of variable, $y^{-1}=(n-k)(1-x^2)$ we obtain the desired result:
\begin{align*}
&\sum_{\substack{\ell\ge 3\\\mathrm{odd}}}\sigma_{n,k}^{-\ell}R^{\ell}\int_{-\frac{1}{2}\log(1-\frac{R^2}{\kappa \sigma_{n,k}^2})}^{\infty}\frac{e^{-\ell t}\|H_{\ell-1}\|_p}{\ell!(\sqrt{1-e^{-2t}})^{\ell-1}}dt\\
\le & \frac{1}{2(n-k)}\sum_{K_p\ge m\ge1}\frac{1}{m!}\sigma_{n,k}^{-1}R^{2m+1}\Big(\frac{\|H_{2m}\|_pm!}{(2m+1)!}-\frac{e^{19/300}\pi^{1/4}(K_p+1)^{1/4}(p-1)^{m}}{2^m(2K_p+3)}\Big)\\&\qquad\qquad \times\int_{\frac{1}{(n-k)}}^{\frac{1}{(n-k)S_{k,\kappa}}}\frac{n^m}{\sigma^{2m}k^m}\frac{1}{y^{3/2}\sqrt{y-\frac{1}{n-k}}}(y-\frac{1}{n-k})^mdy
\\&+\frac{e^{19/300}R\sigma_{n,k}^{-1}\pi^{1/4}(K_p+1)^{1/4}}{2(n-k)(2K_p+3)}\int_{\frac{1}{(n-k)}}^{\frac{1}{(n-k)S_{k,\kappa}}}\frac{1}{y^{3/2}\sqrt{y-\frac{1}{n-k}}}\Big[e^{\frac{(p-1)R^2n}{2k\sigma^2}(y-\frac{1}{n-k})}-1\Big]dy.
\end{align*}
\end{proof}

\begin{lemma}\label{lemma:even_integral_bound}For all $\ell\ge 4$ even integer, if we denote $S_{k,\kappa}:=\frac{R^2}{\kappa \sigma_{n,k}^2}$ and $\tilde S_k:=\frac{R^2n}{k\sigma^2}$, choose $K_p\ge 1$ then the following holds
\begin{align*}
&\sum_{\substack{\ell\ge 4\\\mathrm{even}}}\sigma_{n,k}^{-\ell}\int_{-\frac{1}{2}\log(1-\frac{R^2}{\kappa \sigma_{n,k}^2})}^{\infty}\frac{e^{-\ell t}\|H_{\ell-1}\|_p}{\ell!(\sqrt{1-e^{-2t}})^{\ell-1}}R^{\ell-1}dt\\
\le &  \frac{Rn}{2\sigma^2 k(n-k)^{3/2}}\sum_{K_p\ge m\ge1}\frac{\tilde S_k^m}{m!}\Big(\frac{\|H_{2m+1}\|_pm!}{(2m+2)!}-\frac{e^{19/300}\pi^{1/4}(K_p+1)^{1/4}(p-1)^{m+1/2}}{2^m(2K_p+4)\sqrt{2K_p+3}}\Big)\\&\qquad\qquad \times\int_{\frac{1}{(n-k)}}^{\frac{1}{(n-k)S_{k,\kappa}}}\frac{1}{y^{3/2}}(y-\frac{1}{n-k})^{m}dy
\\&+\frac{Rn}{2\sigma^2 k(n-k)^{3/2}}\frac{e^{19/300}\pi^{1/4}(K_p+1)^{1/4}\sqrt{p-1}}{2(2K_p+4)\sqrt{2K_p+3}}\int_{\frac{1}{(n-k)}}^{\frac{1}{(n-k)S_{k,\kappa}}}\frac{1}{y^{3/2}}\Big[e^{\frac{(p-1)\tilde S_k}{2}(y-\frac{1}{n-k})}-1\Big]dy.
\end{align*}
\end{lemma}

\begin{proof}

First, we note that 
\begin{align*}
&\sum_{\substack{\ell\ge 4\\\mathrm{even}}}\sigma_{n,k}^{-\ell}\int_{-\frac{1}{2}\log(1-\frac{R^2}{\kappa \sigma_{n,k}^2})}^{\infty}\frac{e^{-\ell t}\|H_{\ell-1}\|_p}{\ell!(\sqrt{1-e^{-2t}})^{\ell-1}}R^{\ell-1}dt\\
\overset{(a)}{=} & \sum_{\substack{\ell\ge 4\\\mathrm{even}}}\sigma_{n,k}^{-\ell}\|H_{\ell-1}\|_pR^{\ell-1}\int_{0}^{\sqrt{1-\frac{R^2}{\kappa \sigma_{n,k}^2}}}\frac{x^{\ell-1}}{\ell!(\sqrt{1-x^2})^{\ell-1}}dx\\
\overset{(b)}{=} & \sum_{m\ge1}\frac{1}{(2m+2)!}\sigma_{n,k}^{-(2m+2)}\|H_{2m+1}\|_pR^{2m+1}\int_{0}^{\sqrt{1-\frac{R^2}{\kappa \sigma_{n,k}^2}}}\frac{x^{2m+1}}{(1-x^2)^{m+\frac{1}{2}}}dx,
\end{align*} where (a) is due to a change of variable $x=e^{-t}$, (b) is due to the fact that any even integer $\ell\ge 4$ can be rewritten as $2m+2$ for $m\ge 1$ and finally where (c) is due to a change of variable $y=(n-k)(1-x^2)$. We then remark that according to \cite{robbins1955remark}, we have $$\sqrt{(2 m+1) !}=\sqrt{2 m+1} \sqrt{2 m !} \geq \sqrt{2 m+1} e^{-19 / 300} 2^m m ! /(\pi m)^{1 / 4}.$$ 
Moreover, by \Cref{lemma:Hermite_bound}, the following inequality for the Hermite polynomials holds
$$
\left\|H_{2m+1}\right\|_p \leq \sqrt{(2m+1) !} \sqrt{p-1}^{2m+1}.
$$
Hence for all $K_p\ge 1$, we obtain that
\begin{align*}
&\sum_{\substack{\ell\ge 4\\\mathrm{even}}}\sigma_{n,k}^{-\ell}\int_{-\frac{1}{2}\log(1-\frac{R^2}{\kappa \sigma_{n,k}^2})}^{\infty}\frac{e^{-\ell t}\|H_{\ell-1}\|_p}{\ell!(\sqrt{1-e^{-2t}})^{\ell-1}}R^{\ell-1}dt\\
\le & \sum_{K_p\ge m\ge1}\frac{1}{(2m+2)!}\sigma_{n,k}^{-(2m+2)}\|H_{2m+1}\|_pR^{2m+1}\int_{0}^{\sqrt{1-\frac{R^2}{\kappa \sigma_{n,k}^2}}}\frac{x^{2m+1}}{(1-x^2)^{m+\frac{1}{2}}}dx\\
+& \sum_{m\ge K_p+1}\frac{1}{(2m+2)\sqrt{(2m+1)!}}\sigma_{n,k}^{-(2m+2)}(p-1)^{m+1/2}R^{2m+1}\int_{0}^{\sqrt{1-\frac{R^2}{\kappa \sigma_{n,k}^2}}}\frac{x^{2m+1}}{(1-x^2)^{m+\frac{1}{2}}}dx\\
\le & \sum_{K_p\ge m\ge1}\frac{1}{(2m+2)!}\sigma_{n,k}^{-(2m+2)}\|H_{2m+1}\|_pR^{2m+1}\int_{0}^{\sqrt{1-\frac{R^2}{\kappa \sigma_{n,k}^2}}}\frac{x^{2m+1}}{(1-x^2)^{m+\frac{1}{2}}}dx\\
&+ \sum_{m\ge K_p+1}\frac{m^{1/4}e^{19/300}\pi^{1/4}(p-1)^{m+1/2}R^{2m+1}}{2^m m!(2m+2)\sqrt{2m+1}}\sigma_{n,k}^{-(2m+2)}\int_{0}^{\sqrt{1-\frac{R^2}{\kappa \sigma_{n,k}^2}}}\frac{x^{2m+1}}{(1-x^2)^{m+\frac{1}{2}}}dx\\
\le & \sum_{K_p\ge m\ge1}\frac{1}{m!}\sigma_{n,k}^{-(2m+2)}R^{2m+1}\Big(\frac{\|H_{2m+1}\|_pm!}{(2m+2)!}-\frac{e^{19/300}\pi^{1/4}(K_p+1)^{1/4}(p-1)^{m+1/2}}{2^m(2K_p+4)\sqrt{2K_p+3}}\Big)\\&\qquad\int_{0}^{\sqrt{1-\frac{R^2}{\kappa \sigma_{n,k}^2}}}\frac{x^{2m+1}}{(1-x^2)^{m+\frac{1}{2}}}dx\\
&+ \frac{e^{19/300}\pi^{1/4}(K_p+1)^{1/4}\sigma_{n,k}^{-2}R\sqrt{p-1}}{(2K_p+4)\sqrt{2K_p+3}}\int_{0}^{\sqrt{1-\frac{R^2}{\kappa \sigma_{n,k}^2}}}\frac{x}{\sqrt{1-x^2}}\Big[e^{\frac{x^{2}(p-1)R^2}{2(1-x^2)\sigma_{n,k}}}-1\Big]dx.
\end{align*}
Finally, by the change of variable $y^{-1}=(n-k)(1-x^2)$, we obtain that 
\begin{align*}
&\sum_{\substack{\ell\ge 4\\\mathrm{even}}}\sigma_{n,k}^{-\ell}\int_{-\frac{1}{2}\log(1-\frac{R^2}{\kappa \sigma_{n,k}^2})}^{\infty}\frac{e^{-\ell t}\|H_{\ell-1}\|_p}{\ell!(\sqrt{1-e^{-2t}})^{\ell-1}}R^{\ell-1}dt
\\\le &  \frac{Rn}{2\sigma^2 k(n-k)^{3/2}}\sum_{K_p\ge m\ge1}\frac{\tilde S_k^m}{m!}\Big(\frac{\|H_{2m+1}\|_pm!}{(2m+2)!}-\frac{e^{19/300}\pi^{1/4}(K_p+1)^{1/4}(p-1)^{m+1/2}}{2^m(2K_p+4)\sqrt{2K_p+3}}\Big)\\&\qquad\qquad \times\int_{\frac{1}{(n-k)}}^{\frac{1}{(n-k)S_{k,\kappa}}}\frac{1}{y^{3/2}}(y-\frac{1}{n-k})^{m}dy
\\&+\frac{Rn}{2\sigma^2 k(n-k)^{3/2}}\frac{e^{19/300}\pi^{1/4}(K_p+1)^{1/4}\sqrt{p-1}}{2(2K_p+4)\sqrt{2K_p+3}}\int_{\frac{1}{(n-k)}}^{\frac{1}{(n-k)S_{k,\kappa}}}\frac{1}{y^{3/2}}\Big[e^{\frac{(p-1)\tilde S_k}{2}(y-\frac{1}{n-k})}-1\Big]dy.
\end{align*}
\end{proof}

Now we may use \Cref{lemma:odd_Lp_term_bound,lemma:even_Lp_term_bound,lemma:odd_integral_bound,lemma:even_integral_bound} to bound the term $(a_3)$. In this goal, we define
\begin{equation}\label{eqn:def_C_odd}
C_{\mathrm{odd}}:=
\begin{cases}
C_{\mathrm{odd},1} & \text{if } p < 4,\\
\min\{C_{\mathrm{odd},1}, C_{\mathrm{odd},2}\} & \text{if } p \geq 4,
\end{cases}
\end{equation}
where
\begin{align*}
C_{\mathrm{odd},1}:= &\frac{1}{n}\sqrt{p-1}\sqrt{k(n-k)}R^{-\frac{4}{p}} \sigma^{2/p}\bigg( (n-k)(R^2+3\sigma^2)^{\frac{2}{p}}+k\min\Big\{(2^{\frac{1}{p}}R^{\frac{2}{p}} (R^2+3\sigma^2)^{\frac{1}{p}})^2,\\& ~\Big((R^2+3\sigma^2)^{\frac{1}{p}}~+\frac{R^{\frac{4}{p}}}{\sqrt{k}\sigma^{2/p}}\min\{
    \sqrt{p-1},A_p+A^*_{k,p}\}\Big)^2 \Big\}\bigg)^{1/2},
\end{align*}
and
\begin{align*}
C_{\mathrm{odd},2}:=& \frac{1}{n}C_p\bigg\{R^{-2}\sqrt{(n-k)k}\sqrt{(n-k)\sigma^2(R^2+3\sigma^2)+k{(R^2\sigma^2+R^{4-\frac{4}{p}})}}\\&+k^{\frac{1}{p}}(n-k)^{\frac{1}{p}}R^{\ell-\frac{4}{p}}\bigg(k^{p-1} \Big(\sigma^{2/p}(R^2+3\sigma^2)^{1/p}+2^{1/p}k^{-\frac{1}{2}}R^{2/p}\Big)^p\\&\qquad \qquad \qquad \qquad\qquad +(n-k)^{p-1}\sigma^2(R^2+3\sigma^2)\bigg)^{\frac{1}{p}}\bigg\},\end{align*}
and 
\begin{equation}\label{eqn:def_C_even}
C_{\mathrm{even}}:=
\begin{cases}
\min\{C_{\mathrm{even},1}, C_{\mathrm{even},3}\} & \text{if } p < 4,\\
\min\{C_{\mathrm{even},1}, C_{\mathrm{even},2},C_{\mathrm{even},3}\} & \text{if } p \geq 4,
\end{cases}
\end{equation}
where
\begin{align*}
    C_{\mathrm{even},1}:=&  \frac{1}{n}{\sqrt{k(n-k)}}R^{-2}\bigg[\sqrt{p-1}R^{2-\frac{4}{p}}\Big((n-k)\sigma^{4/p}(R^2+3\sigma^2)^\frac{2}{p}
   \\&\quad\quad\quad\quad+kR^{\frac{2}{p}}\min\Big\{(\frac{1}{4})^{\frac{1}{p}}R^{\frac{2}{p}},2^{\frac{1}{p}}\sigma^{\frac{2}{p}}+ R^{-2/p}\sigma^{4/p}(R^2+3\sigma^2)^{\frac{1}{p}}\Big\}^2\Big)^{\frac{1}{2}}\\
      &+{2\sqrt{k(n-k)}\sigma^2}\bigg],
\end{align*}
\begin{align*}
 C_{\mathrm{even},2}&:= \frac{1}{n}C_p\Big\{R^{-2}\sqrt{(n-k)k}\sqrt{2(n-k)\sigma^{4/p}(R^2+3\sigma^2)^{\frac{2}{p}}+ (n-k)R^{2} +k(\sigma^2+\sigma^{4/p}R^{2-\frac{4}{p}})}\\&+k^{\frac{1}{p}}(n-k)^{\frac{1}{p}}R^{\ell-\frac{4}{p}}\Big((n-k)^{p-1}\sigma^2(R^2+3\sigma^2)\\&+k^{p-1}\min\big\{\frac{R^4}{4},R^2\sigma^2(2^{1/p}+(R^2+3\sigma^2)^{1/p})^p\big\}\Big)^{1/p}\Big\}\\
&+ \frac{2}{n}\sigma_{n,k}^{-\ell}k(n-k)R^{\ell-2}\sigma^2,\end{align*}
and 
\begin{align*}
    C_{\mathrm{even},3}:= & \frac{1}{n}\Big\{k \|\mathrm{binom}(k,\frac{2\sigma^2}{R^2})\|_p+(n-k)k2^{\frac{1}{p}}R^{-2/p}\sigma^{2/p}\Big\}.
\end{align*}

\begin{lemma}\label{lemma:a3_bound}
The following bound holds for all $p\ge2$:
    \begin{align*}
        (a_3) 
        \le &C_{\mathrm{odd}}\bigg\{\frac{1}{2(n-k)}\sum_{K_p\ge m\ge1}\frac{1}{m!}\sigma_{n,k}^{-1}R\tilde S_k^m\Big(\frac{\|H_{2m}\|_pm!}{(2m+1)!}-\frac{e^{19/300}\pi^{1/4}(K_p+1)^{1/4}(p-1)^{m}}{2^m(2K_p+3)}\Big)\\&\qquad\qquad \times\int_{\frac{1}{(n-k)}}^{\frac{1}{(n-k)S_{k,\kappa}}}\frac{1}{y^{3/2}\sqrt{y-\frac{1}{n-k}}}(y-\frac{1}{n-k})^mdy
\\&+\frac{e^{19/300}R\sigma_{n,k}^{-1}\pi^{1/4}(K_p+1)^{1/4}}{2(2K_p+3)(n-k)}\int_{\frac{1}{(n-k)}}^{\frac{1}{(n-k)S_{k,\kappa}}}\frac{1}{y^{3/2}\sqrt{y-\frac{1}{n-k}}}\Big[e^{\frac{(p-1)\tilde S_k}{2}(y-\frac{1}{n-k})}-1\Big]dy
        \bigg\}\\
        &+ C_{\mathrm{even}}R\bigg\{ \frac{Rn}{2\sigma^2 k(n-k)^{3/2}}\\
        &\qquad\qquad\times\sum_{K_p\ge m\ge1}\frac{\tilde S_k^m}{m!}\Big(\frac{\|H_{2m+1}\|_pm!}{(2m+2)!}-\frac{e^{19/300}\pi^{1/4}(K_p+1)^{1/4}(p-1)^{m+1/2}}{2^m(2K_p+4)\sqrt{2K_p+3}}\Big)\\&\qquad\qquad \times\int_{\frac{1}{(n-k)}}^{\frac{1}{(n-k)S_{k,\kappa}}}\frac{1}{y^{3/2}}(y-\frac{1}{n-k})^{m}dy
\\&+\frac{Rn}{2\sigma^2 k(n-k)^{3/2}}\frac{e^{19/300}\pi^{1/4}(K_p+1)^{1/4}\sqrt{p-1}}{2(2K_p+4)\sqrt{2K_p+3}}\\&\qquad\qquad \times\int_{\frac{1}{(n-k)}}^{\frac{1}{(n-k)S_{k,\kappa}}}\frac{1}{y^{3/2}}\Big[e^{\frac{(p-1)\tilde S_k}{2}(y-\frac{1}{n-k})}-1\Big]dy\bigg\}.
    \end{align*}
\end{lemma}
\begin{proof}
    We have that
\begin{align*}
    (a_3)= &\sum_{\ell\ge 3}\int_{-\frac{1}{2}\log(1-\frac{R^2}{\kappa \sigma_{n,k}^2})}^{\infty}\frac{e^{-lt}\|H_{\ell-1}\|_p}{s\ell!(\sqrt{1-e^{-2t}})^{\ell-1}}\sigma_{n,k}^{-\ell}\Big\|\mathbb{E}\big[(W_k'-W_k)^\ell|W_k,\mathcal{U}(X_{1:n})\big]\Big\|_pdt\\
    = & \sum_{\substack{\ell\ge 3\\\mathrm{odd}}}\frac{1}{s}\sigma_{n,k}^{-\ell}\Big\|\mathbb{E}\big[(W_k'-W_k)^\ell|W_k,\mathcal{U}(X_{1:n})\big]\Big\|_p\int_{-\frac{1}{2}\log(1-\frac{R^2}{\kappa \sigma_{n,k}^2})}^{\infty}\frac{e^{-\ell t}\|H_{\ell-1}\|_p}{\ell!(\sqrt{1-e^{-2t}})^{\ell-1}}dt\\
    & + \sum_{\substack{\ell\ge 4\\\mathrm{even}}}\frac{1}{s}\sigma_{n,k}^{-\ell}\Big\|\mathbb{E}\big[(W_k'-W_k)^\ell|W_k,\mathcal{U}(X_{1:n})\big]\Big\|_p\int_{-\frac{1}{2}\log(1-\frac{R^2}{\kappa \sigma_{n,k}^2})}^{\infty}\frac{e^{-\ell t}\|H_{\ell-1}\|_p}{\ell!(\sqrt{1-e^{-2t}})^{\ell-1}}dt.
\end{align*}
By \Cref{lemma:odd_Lp_term_bound}, we have that for odd $\ell>3$,
\begin{equation*}
    \frac{1}{s}\sigma_{n,k}^{-\ell}\Big\|\mathbb{E}\big[(W_k'-W_k)^\ell|W_k,\mathcal{U}(X_{1:n})\big]\Big\|_p \le C_{\mathrm{odd}}\sigma_{n,k}^{-\ell}R^\ell,
\end{equation*}
and by \Cref{lemma:even_Lp_term_bound}, we have that for even $\ell\ge4$,
\begin{equation*}
    \frac{1}{s}\sigma_{n,k}^{-\ell}\Big\|\mathbb{E}\big[(W_k'-W_k)^\ell|W_k,\mathcal{U}(X_{1:n})\big]\Big\|_p \le C_{\mathrm{even}}\sigma_{n,k}^{-\ell}R^\ell,
\end{equation*}
where $C_{\mathrm{odd}}$ and $C_{\mathrm{even}}$ are defined as in \cref{eqn:def_C_odd} and \cref{eqn:def_C_even}.
Hence we have
\begin{align*}
    (a_3)
    \le & C_{\mathrm{odd}}\sum_{\substack{\ell\ge 3\\\mathrm{odd}}}R^\ell\int_{-\frac{1}{2}\log(1-\frac{R^2}{\kappa \sigma_{n,k}^2})}^{\infty}\frac{e^{-\ell t}\|H_{\ell-1}\|_p}{\ell!(\sqrt{1-e^{-2t}})^{\ell-1}}dt \\&+ C_{\mathrm{even}}R\sum_{\substack{\ell\ge 4\\\mathrm{even}}}R^{\ell-1}\int_{-\frac{1}{2}\log(1-\frac{R^2}{\kappa \sigma_{n,k}^2})}^{\infty}\frac{e^{-\ell t}\|H_{\ell-1}\|_p}{\ell!(\sqrt{1-e^{-2t}})^{\ell-1}}dt.
\end{align*}
Note that by \Cref{lemma:odd_integral_bound}, we have
\begin{align*}
&\sum_{\substack{\ell\ge 3\\\mathrm{odd}}}\sigma_{n,k}^{-\ell}R^{\ell}\int_{-\frac{1}{2}\log(1-\frac{R^2}{\kappa \sigma_{n,k}^2})}^{\infty}\frac{e^{-\ell t}\|H_{\ell-1}\|_p}{\ell!(\sqrt{1-e^{-2t}})^{\ell-1}}dt
\\\le 
 & \frac{1}{2(n-k)}\sum_{K_p\ge m\ge1}\frac{1}{m!}\sigma_{n,k}^{-1}R\tilde S_k^m\Big(\frac{\|H_{2m}\|_pm!}{(2m+1)!}-\frac{e^{19/300}\pi^{1/4}(K_p+1)^{1/4}(p-1)^{m}}{2^m(2K_p+3)}\Big)\\&\qquad\qquad \times\int_{\frac{1}{(n-k)}}^{\frac{1}{(n-k)S_{k,\kappa}}}\frac{1}{y^{3/2}\sqrt{y-\frac{1}{n-k}}}(y-\frac{1}{n-k})^mdy
\\&+\frac{e^{19/300}R\sigma_{n,k}^{-1}\pi^{1/4}(K_p+1)^{1/4}}{2(2K_p+3)(n-k)}\int_{\frac{1}{(n-k)}}^{\frac{1}{(n-k)S_{k,\kappa}}}\frac{1}{y^{3/2}\sqrt{y-\frac{1}{n-k}}}\Big[e^{\frac{(p-1)\tilde S_k}{2}(y-\frac{1}{n-k})}-1\Big]dy.
\end{align*}
Also note that by \Cref{lemma:even_integral_bound}, we have
\begin{align*}
&\sum_{\substack{\ell\ge 4\\\mathrm{even}}}\sigma_{n,k}^{-\ell}\int_{-\frac{1}{2}\log(1-\frac{R^2}{\kappa \sigma_{n,k}^2})}^{\infty}\frac{e^{-\ell t}\|H_{\ell-1}\|_p}{\ell!(\sqrt{1-e^{-2t}})^{\ell-1}}R^{\ell-1}dt\\
\le &  \frac{Rn}{2\sigma^2 k(n-k)^{3/2}}\sum_{K_p\ge m\ge1}\frac{\tilde S_k^m}{m!}\Big(\frac{\|H_{2m+1}\|_pm!}{(2m+2)!}-\frac{e^{19/300}\pi^{1/4}(K_p+1)^{1/4}(p-1)^{m+1/2}}{2^m(2K_p+4)\sqrt{2K_p+3}}\Big)\\&\qquad\qquad \times\int_{\frac{1}{(n-k)}}^{\frac{1}{(n-k)S_{k,\kappa}}}\frac{1}{y^{3/2}}(y-\frac{1}{n-k})^{m}dy
\\&+\frac{Rn}{2\sigma^2 k(n-k)^{3/2}}\frac{e^{19/300}\pi^{1/4}(K_p+1)^{1/4}\sqrt{p-1}}{2(2K_p+4)\sqrt{2K_p+3}}\int_{\frac{1}{(n-k)}}^{\frac{1}{(n-k)S_{k,\kappa}}}\frac{1}{y^{3/2}}\Big[e^{\frac{(p-1)\tilde S_k}{2}(y-\frac{1}{n-k})}-1\Big]dy.
\end{align*}
\end{proof}
\begin{theorem}\label{theorem:pfWp}
    For $p\ge 2$, the following inequality holds for any $K\ge0$:$$\Big\|\mathcal{W}_p(W_k,\mathcal{N}(0,\sigma_{n,k}^2)|\mathcal{U}(X_{1:n}))\Big\|_p 
        \le \omega^R_p(n,k,\sigma),$$ where we set 
    \begin{align}\label{definition_omega}
       \omega^R_p(n,k,\sigma)
        :=\min\Big(\sqrt{p-1}\sigma_{n,k}(1+2\sigma^{-1}),&~ \Omega^R_p(n,k,\sigma)\Big),\end{align} where we have set
        \begin{align*}
  \Omega^R_p(n,k,\sigma)&=  \inf_{\substack{\kappa\le R^2/\sigma_{n,k}^2\\K\ge 1}} \sigma_{n,k}\Bigg\{ \sigma_{n,k}^{-1}\mathcal{S}^p_{\sigma,R}(n,k)\times(1-\sqrt{1-\frac{R^2}{\kappa\sigma_{n,k}^2}}) \\&+\sqrt{2}\left(\frac{\Gamma((p+1) / 2)}{\Gamma(1 / 2)}\right)^{1 / p}\arccos{(\sqrt{1-\frac{R^2}{\kappa\sigma_{n,k}^2}})}  \\
        & + \|H_1\|_p\frac{\sqrt{n}}{\sqrt{k(n-k)}}\Big[\frac{1}{2}\min\Big(  \sqrt{p-1}(\tilde R^2_s-1)^{1-1/p},A_p\sqrt{\tilde R^2_s-1}+A^*_{n,p}(\tilde R^2_s-1)^{1-1/p} \Big)\\&\qquad\qquad+\frac{1}{\sqrt{n}}\min\Big(\sqrt{p-1}\tilde R_s^{1-2/p},A_p +\tilde R_s^{1-2/p}A^*_{k,p}\Big)^2\Big]\Big(1 - \frac{R}{\sigma_{n,k}\sqrt{\kappa}}\Big)\\
        &+C_{\mathrm{odd}}\bigg\{\frac{1}{2(n-k)}\sum_{K\ge m\ge1}\frac{1}{m!}\sigma_{n,k}^{-1}R\tilde S_k^m\Big(\frac{\|H_{2m}\|_pm!}{(2m+1)!}-\frac{e^{19/300}\pi^{1/4}(K+1)^{1/4}(p-1)^{m}}{2^m(2K+3)}\Big)\\&\qquad\qquad \times\int_{\frac{1}{(n-k)}}^{\frac{1}{(n-k)S_{k,\kappa}}}\frac{1}{y^{3/2}\sqrt{y-\frac{1}{n-k}}}(y-\frac{1}{n-k})^mdy
\\&+\frac{e^{19/300}R\sigma_{n,k}^{-1}\pi^{1/4}(K+1)^{1/4}}{2(2K+3)(n-k)}\int_{\frac{1}{(n-k)}}^{\frac{1}{(n-k)S_{k,\kappa}}}\frac{1}{y^{3/2}\sqrt{y-\frac{1}{n-k}}}\Big[e^{\frac{(p-1)\tilde S_k}{2}(y-\frac{1}{n-k})}-1\Big]dy
        \bigg\}\\
        &+ C_{\mathrm{even}}R\bigg\{ \frac{Rn}{2\sigma^2 k(n-k)^{3/2}}\\&\qquad\qquad \times\sum_{K\ge m\ge1}\frac{\tilde S_k^m}{m!}\Big(\frac{\|H_{2m+1}\|_pm!}{(2m+2)!}-\frac{e^{19/300}\pi^{1/4}(K+1)^{1/4}(p-1)^{m+1/2}}{2^m(2K+4)\sqrt{2K+3}}\Big)\\&\qquad\qquad \times\int_{\frac{1}{(n-k)}}^{\frac{1}{(n-k)S_{k,\kappa}}}\frac{1}{y^{3/2}}(y-\frac{1}{n-k})^{m}dy
\\&+\frac{Rn}{2\sigma^2 k(n-k)^{3/2}}\frac{e^{19/300}\pi^{1/4}(K+1)^{1/4}\sqrt{p-1}}{2(2K+4)\sqrt{2K+3}}\\&\qquad\qquad \times\int_{\frac{1}{(n-k)}}^{\frac{1}{(n-k)S_{k,\kappa}}}\frac{1}{y^{3/2}}\Big[e^{\frac{(p-1)\tilde S_k}{2}(y-\frac{1}{n-k})}-1\Big]dy\Big\}\Bigg\}.
    \end{align*} where $\mathcal{S}^p_{\sigma,R}(n,k)$ is defined in \Cref{nelson2}.
    
\end{theorem}
\begin{proof}
Recall that in \cref{eqn:W_p_bound_intermediate}, we have established that
\begin{align*}
&   \sigma^{-1}_{n,k} \Big\|\mathcal{W}_p(W_k,\mathcal{N}(0,\sigma_{n,k}^2)|\mathcal{U}(X_{1:n}))\Big\|_p\\&\le  \int_0^{-\frac{1}{2}\log(1-\frac{R^2}{\kappa \sigma_{n,k}^2})} \mathbb{E}\Big[\Big|e^{-t} \sigma_{n,k}^{-1}W_k-\frac{e^{-2t}}{\sqrt{1-e^{-2t}}}\tilde Z\Big|^p\Big]^{\frac{1}{p}}dt
    \\&+\sigma_{n,k}^{-1}\int_{-\frac{1}{2}\log(1-\frac{R^2}{\kappa \sigma_{n,k}^2})}^{\infty}e^{-t}\Big\|\frac{1}{s}\mathbb{E}\big[W_k-W_k'|W_k,\mathcal{U}(X_{1:n})\big]-
W_k\Big\|_pdt
\\&+\int_{-\frac{1}{2}\log(1-\frac{R^2}{\kappa \sigma_{n,k}^2})}^{\infty}\frac{e^{-2t}\|H_1\|_p}{\sqrt{1-e^{-2t}}}\Big\|\frac{1}{2s}\sigma_{n,k}^{-2}\mathbb{E}\big[(W'_k-W_k)^2|W_k,\mathcal{U}(X_{1:n})\big]-1\Big\|_pdt
\\&+\sum_{l\ge 3}\int_{-\frac{1}{2}\log(1-\frac{R^2}{\kappa \sigma_{n,k}^2})}^{\infty}\frac{e^{-lt}\|H_{l-1}\|_p}{sl!(\sqrt{1-e^{-2t}})^{l-1}}\sigma_{n,k}^{-\ell}\Big\|\mathbb{E}\big[(W_k'-W_k)^\ell|W_k,\mathcal{U}(X_{1:n})\big]\Big\|_pdt\\
&:= (a_0)+(a_1)+(a_2)+(a_3).
\end{align*} 
Combining the result of \Cref{lemma:a0_bound,lemma:a1_bound,lemma:a2_bound,lemma:a3_bound} and reorganizing terms gives us the first bound. Secondly, we remark that by a triangle inequality we also obtain the following trivial bound:
\begin{align*}
    \|\mathcal{W}_p(W_k,\mathcal{N}(0,\sigma_{n,k}^2)|\mathcal{U}(X_{1:n}))\Big\|_p&\le \|W_k\|_p+\sigma_{n,k}\|Z\|_p
    \\&\le(1-\frac{k}{n}) \|S_k\|_p+\frac{k}{n}\|S_n-S_k\|_p+\sigma_{n,k}\|Z\|_p
    \\&\le \sqrt{p-1}\Big(\frac{n-k}{n}\sqrt{k}+\frac{k\sqrt{n-k}}{{n}}\Big)R+\sqrt{p-1}\sigma_{n,k}
    \\&\le  \sqrt{p-1}\frac{\sqrt{k(n-k)}}{\sqrt{n}}\Big(\frac{\sqrt{n-k}}{\sqrt{n}}+\frac{\sqrt{k}}{\sqrt{n}}\Big)R+\sqrt{p-1}\sigma_{n,k}
    \\&\le \sqrt{p-1}\sigma_{n,k}(1+\sqrt{2}R\sigma^{-1}).
\end{align*}
\end{proof}
\subsection{Proof of Equation (\ref{eqn:main122b})}\label{dearreader}

We define the following constant
\begin{align}\label{definition_k}
\tilde K_{R,\sigma} :=& \Big\{\KRsig +\frac{R^2}{\sigma}(1+\sigma/2)+ 2^{\frac{3}{p}}\sigma^{2/p-1}\frac{e^{19/300}\pi^{1/4}}{3}\frac{R^2}{2}\Big(e^{R^2}-1\Big)\log\Big(\frac{1+\sqrt{1/2}}{1-\sqrt{1/2}}\Big)\\&\quad+R^{-\frac{4}{p}} \bigg[\tilde R^{2/p}4^{1/p}\sigma^{-1}+{2}\bigg] \frac{e^{19/300}\pi^{1/4}}{8\sqrt{3}}\frac{R^{5/2}}{\sqrt{2}}\Big(e^{R^2}-1\Big)\Big\}\vee (1+\sqrt{2}R\sigma^{-1}),
\end{align}
where $\KRsig$ is defined in \Cref{nelson2s}.
\begin{lemma}\label{lemma:Wp_order}
Let $R > 0$ and $R/2>\sigma > 0$. Then for all $k,n \in \mathbb{N}$ and all $p \in \mathbb{N}$, we have that 
$$
\omega^R_p(n, k, \sigma) \le \tilde K_{R,\sigma}p.
$$
Moreover, for all $0<\sigma_1<\sigma_2\le \frac{R}{2}$, we have $$\max_{\sigma\in [\sigma_1,\sigma_2]}|\tilde K_{R,\sigma}|<\infty.$$
\end{lemma}
\begin{proof}
We will establish the bound in two steps: First for the case $p\le \sigma_{n,k}^2+1$, and then for $p>\sigma_{n,k}^2+1$. 

In the first case, we use the bound established in \Cref{theorem:pfWp} with $K=0$ and $\kappa=C/(p-1)$ with $C=2R^2$, which tells us that
 \begin{align*}
      & \Omega^R_p(n,k,\sigma)
        \\\le & 
    \mathcal{S}^p_{\sigma,R}(n,k)\times(1-\sqrt{1-\frac{R^2}{\kappa\sigma_{n,k}^2}}) +\sigma_{n,k}\Bigg\{ \sqrt{2}\left(\frac{\Gamma((p+1) / 2)}{\Gamma(1 / 2)}\right)^{1 / p}\arccos{(\sqrt{1-\frac{R^2}{\kappa\sigma_{n,k}^2}})}  \\
        & + \|H_1\|_p\frac{\sqrt{n}}{\sqrt{k(n-k)}}\Big[\frac{1}{2}\min\Big(  \sqrt{p-1}(\tilde R^2_s-1)^{1-1/p},A_p\sqrt{\tilde R^2_s-1}+A^*_{n,p}(\tilde R^2_s-1)^{1-1/p} \Big)\\&\qquad\qquad+\frac{1}{\sqrt{n}}\min\Big(\sqrt{p-1}\tilde R_s^{1-2/p},A_p +\tilde R_s^{1-2/p}A^*_{k,p}\Big)^2\Big]\Big(1 - \frac{R}{\sigma_{n,k}\sqrt{\kappa}}\Big)\\
        &+C_{\mathrm{odd}}\bigg\{\frac{e^{19/300}R\sigma_{n,k}^{-1}\pi^{1/4}}{6(n-k)}\int_{\frac{1}{(n-k)}}^{\frac{1}{(n-k)S_{k,\kappa}}}\frac{1}{y^{3/2}\sqrt{y-\frac{1}{n-k}}}\Big[e^{\frac{(p-1)\tilde S_k}{2}(y-\frac{1}{n-k})}-1\Big]dy
        \bigg\}\\
        &+ C_{\mathrm{even}}\bigg\{ \frac{R^2n}{2\sigma^2 k(n-k)^{3/2}}\frac{e^{19/300}\pi^{1/4}\sqrt{p-1}}{8\sqrt{3}}\int_{\frac{1}{(n-k)}}^{\frac{1}{(n-k)S_{k,\kappa}}}\frac{1}{y^{3/2}}\Big[e^{\frac{(p-1)\tilde S_k}{2}(y-\frac{1}{n-k})}-1\Big]dy\Big\}\Bigg\}.
    \end{align*} We will bound each term successively. In this goal, we first remark that
using \Cref{nelson2s},
$$\mathcal{S}_p^{\sigma,R}(n,k)(1-\sqrt{1-\frac{R^2}{\kappa\sigma_{n,k}^2}}) \le\mathcal{S}_p^{\sigma,R}(n,k) \le \KRsig  p,$$
where $\KRsig$ is a constant defined in \Cref{nelson2s}. Moreover, using the mean value theorem, we observe that 
\begin{align*}
 &  \sqrt{2}\left(\frac{\Gamma((p+1) / 2)}{\Gamma(1 / 2)}\right)^{1 / p}\arccos{(\sqrt{1-\frac{R^2}{\sigma_{n,k}^2 \kappa}})}\\& \le \sqrt{p-1}\arccos{(\sqrt{1-\frac{R^2}{\sigma_{n,k}^2 \kappa}})}\\
\le & \sqrt{p-1}\frac{R^2}{\sigma_{n,k}^2\kappa\sqrt{1-\frac{R^2}{\sigma_{n,k}^2\kappa}}}
\\\le & \sqrt{p-1}\frac{R^2}{\sigma_{n,k}^2\kappa\sqrt{0.5}},
\end{align*}
where for the last inequality we used the fact that $$\frac{R^2}{\sigma_{n,k}^2\kappa}\le 2.$$
Moreover using \Cref{lemma:Hermite_bound}, we note that $\|H_1\|_p\le \sqrt{p-1}$. This directly implies that 
\begin{align*}&
\|H_1\|_p\frac{\sqrt{n}}{\sqrt{k(n-k)}}\Big[\frac{1}{2}\min\Big\{  \sqrt{p-1}(\tilde R^2_s-1)^{1-1/p},A_p\sqrt{\tilde R^2_s-1}+A^*_{n,p}(\tilde R^2_s-1)^{1-1/p} \Big\}\\&\qquad\qquad+\frac{1}{\sqrt{n}}\min\Big\{\sqrt{p-1}\tilde R_s^{1-2/p},A_p +\tilde R_s^{1-2/p}A^*_{k,p}\Big\}^2\Big]\Big(1 - \frac{R}{\sigma_{n,k}\sqrt{\kappa}}\Big)
        \\&\le (p-1)\frac{\sqrt{n}}{\sqrt{k(n-k)}}\Big[\frac{1}{2}(\tilde R^2_s-1)^{1-1/p}+\frac{\sqrt{p-1}}{\sqrt{n}}\tilde R_s^{2-\frac{4}{p}}\Big]
        \\&\le (p-1)\frac{\tilde R^2\sqrt{n}}{\sqrt{k(n-k)}}\Big[1+\frac{\sqrt{p-1}}{\sqrt{n}}\Big].
    \end{align*}
    Moreover, we note that  
\begin{align*}
    C_{\mathrm{odd}}\le&\frac{1}{n}\sqrt{p-1}\sqrt{k(n-k)}R^{-\frac{4}{p}} \sigma^{2/p}\bigg( (n-k)(R^2+3\sigma^2)^{\frac{2}{p}}+k(2^{\frac{1}{p}}R^{\frac{2}{p}} (R^2+3\sigma^2)^{\frac{1}{p}})^2\bigg)^{1/2}
    \\ \le& \frac{\sqrt{k(n-k)}}{\sqrt{n}}\sqrt{p-1}R^{-\frac{4}{p}} 2^{\frac{3}{p}}R^{4/p}\sigma^{2/p}.
\end{align*}Now, note that for all $x\in[0, \sqrt{1-\frac{R^2}{\sigma_{n,k}^2 \kappa}}]$, we have $\frac{1}{(1-x^2)^m}\le \Big(\frac{\sigma_{n,k}^2 \kappa}{R^2}\Big)^{m-1}\frac{1}{1-x^2}.$ Hence, a change of variables implies
\begin{align*}
&\frac{e^{19/300}R\sigma_{n,k}^{-1}\pi^{1/4}}{6(n-k)}\int_{\frac{1}{(n-k)}}^{\frac{1}{(n-k)S_{k,\kappa}}}\frac{1}{y^{3/2}\sqrt{y-\frac{1}{n-k}}}\Big[e^{\frac{(p-1)\tilde S_k}{2}(y-\frac{1}{n-k})}-1\Big]dy\\
= &\frac{e^{19/300}R\sigma_{n,k}^{-1}\pi^{1/4}}{6(n-k)}\int_{0}^{\sqrt{1-S_{k,\kappa}}}2(n-k)\Big[e^{\frac{(p-1)\tilde S_k}{2}\frac{x^2}{(n-k)(1-x^2)}}-1\Big]dx\\
= &\frac{e^{19/300}R\sigma_{n,k}^{-1}\pi^{1/4}}{3}\sum_{m\ge1}\int_{0}^{\sqrt{1-S_{k,\kappa}}} \frac{1}{m!}\frac{(p-1)^m}{2^m}S_{k,\kappa}^m\kappa^m\frac{x^{2m}}{(1-x^2)^m}dx\\
\overset{(a)}{\le}  &\frac{e^{19/300}R\sigma_{n,k}^{-1}\pi^{1/4}}{3}S_{k,\kappa}\sum_{m\ge1}\int_{0}^{\sqrt{1-S_{k,\kappa}}} \frac{1}{m!}\frac{(p-1)^m}{2^m}\kappa^m\frac{x^{2m}}{1-x^2}dx\\
=&\frac{e^{19/300}R^3\sigma_{n,k}^{-3}\pi^{1/4}}{3\kappa}\int_{0}^{\sqrt{1-S_{k,\kappa}}}\Big(e^{\frac{(p-1)\kappa x^2}{2}}-1\Big)\frac{1}{1-x^2}dx\\
\le & \frac{e^{19/300}\pi^{1/4}}{3}\frac{\sigma_{n,k}^{-3}R^3}{\kappa}\Big(e^{\frac{1}{2}(p-1)\kappa (1-S_{k,\kappa})}-1\Big)\log\Big(\frac{1+\sqrt{1-S_{k,\kappa}}}{1-\sqrt{1-S_{k,\kappa}}}\Big)
\\\le& \frac{e^{19/300}\pi^{1/4}}{3}\frac{\sigma_{n,k}^{-3}R^3}{\kappa}\Big(e^{\frac{C}{2}(1-S_{k,\kappa})}-1\Big)\log\Big(\frac{1+\sqrt{1-S_{k,\kappa}}}{1-\sqrt{1-S_{k,\kappa}}}\Big)
\\\le& \frac{e^{19/300}\pi^{1/4}}{3}\frac{\sigma_{n,k}^{-3}R^3(p-1)}{C}\Big(e^{\frac{C}{2}(1-S_{k,\kappa})}-1\Big)\log\Big(\frac{1+\sqrt{1/2}}{1-\sqrt{1/2}}\Big).
\\ \le& \frac{e^{19/300}\pi^{1/4}}{6}{\sigma_{n,k}^{-3}R(p-1)}{}\Big(e^{\frac{C}{2}}-1\Big)\log\Big(\frac{1+\sqrt{1/2}}{1-\sqrt{1/2}}\Big).
\end{align*}
Moreover, we remark that 
\begin{align*}
    C_{\mathrm{even}}& \le \frac{1}{n}R^{-\frac{4}{p}}{\sqrt{k(n-k)}}\bigg[\sqrt{p-1}\Big((n-k)\sigma^{4/p}(R^2+3\sigma^2)^\frac{2}{p}
   +kR^{\frac{2}{p}}(\frac{1}{4})^{\frac{2}{p}}\Big)^{\frac{1}{2}}\\
      &+{2\sqrt{k(n-k)}\sigma^2}\bigg]
      \\&\le R^{-\frac{4}{p}} \bigg[\sqrt{p-1}\tilde R^{2/p}4^{1/p}\frac{{\sqrt{k(n-k)}}}{\sqrt{n}}+{2}\sigma_{n,k}^2\bigg].
\end{align*}
Similarly, we have
\begin{align*}&
 \frac{R^2n}{2\sigma^2 k(n-k)^{3/2}}\frac{e^{19/300}\pi^{1/4}\sqrt{p-1}}{8\sqrt{3}}\int_{\frac{1}{(n-k)}}^{\frac{1}{(n-k)S_{k,\kappa}}}\frac{1}{y^{3/2}}\Big[e^{\frac{(p-1)\tilde S_k}{2}(y-\frac{1}{n-k})}-1\Big]dy\\
 = &\frac{R^2n}{2\sigma^2 k(n-k)^{3/2}}\frac{e^{19/300}\pi^{1/4}\sqrt{p-1}}{8\sqrt{3}}\int_{0}^{\sqrt{1-S_{k,\kappa}}}2(n-k)^{1/2}\frac{x}{(1-x^2)^{1/2}}\Big[e^{\frac{(p-1)\tilde S_k}{2}\frac{x^2}{(n-k)(1-x^2)}}-1\Big]dx\\
 = &\frac{R^2n}{\sigma^2 k(n-k)}\frac{e^{19/300}\pi^{1/4}\sqrt{p-1}}{8\sqrt{3}}\sum_{m\ge1}\int_{0}^{\sqrt{1-S_{k,\kappa}}} \frac{1}{m!}\frac{(p-1)^m}{2^m}S_{k,\kappa}^m\kappa^m\frac{x^{2m+1}}{(1-x^2)^{m+1/2}}dx\\
 = &\frac{R^2n}{\sigma^2 k(n-k)}\frac{e^{19/300}\pi^{1/4}\sqrt{p-1}}{8\sqrt{3}}S_{k,\kappa}\sum_{m\ge1}\int_{0}^{\sqrt{1-S_{k,\kappa}}} \frac{1}{m!}\frac{(p-1)^m}{2^m}\kappa^m\frac{x^{2m+1}}{(1-x^2)^{3/2}}dx\\
 = &\frac{R^2n}{\sigma^2 k(n-k)}\frac{e^{19/300}\pi^{1/4}\sqrt{p-1}}{8\sqrt{3}}\frac{R^2}{\kappa \sigma_{n,k}^2}\int_{0}^{\sqrt{1-S_{k,\kappa}}} \Big(e^{\frac{(p-1)\kappa x^2}{2}}-1\Big)\frac{x}{(1-x^2)^{3/2}}dx\\
\le & \frac{e^{19/300}\pi^{1/4}}{8\sqrt{3}}\frac{R^{4}\sigma_{n,k}^{-4}\sqrt{p-1}}{\kappa}\Big(e^{\frac{1}{2}(p-1)\kappa S_{k,\kappa}}-1\Big)\frac{\sigma_{n,k}\sqrt{\kappa}}{R}.
\\
\le & \frac{e^{19/300}\pi^{1/4}}{8\sqrt{3}}\frac{R^{3}\sigma_{n,k}^{-3}({p-1})}{\sqrt{C}}\Big(e^{\frac{1}{2}C}-1\Big)
\\
\le & \frac{e^{19/300}\pi^{1/4}}{8\sqrt{3}}\frac{R^{5/2}\sigma_{n,k}^{-3}({p-1})}{\sqrt{2}}\Big(e^{R^2}-1\Big).
\end{align*}
Hence, if $p\le \sigma_{n,k}^2+1\le \frac{n\sigma^2}{4}+1$,  this implies that 
 \begin{align*}
      & \omega^R_p(n,k,\sigma)\\
      \le&  \Omega^R_p(n,k,\sigma)
 \\\le& \KRsig  p+
  (p-1)\frac{R^2}{\sigma}\Big[1+\frac{\sqrt{p-1}}{\sqrt{n}}\Big]+\sqrt{p-1}^3\frac{1}{2\sigma_{n,k}\sqrt{1/2}}
\\&+\sqrt{p-1}R^{-\frac{4}{p}} 2^{\frac{3}{p}}R^{4/p}\sigma^{2/p-1}\frac{e^{19/300}\pi^{1/4}}{3}\frac{\sigma_{n,k}^{-1}R(p-1)}{2}\Big(e^{R^2}-1\Big)\log\Big(\frac{1+\sqrt{1/2}}{1-\sqrt{1/2}}\Big)
\\&+R^{-\frac{4}{p}} \bigg[\sqrt{p-1}\tilde R^{2/p}4^{1/p}\sigma^{-1}\sigma_{n,k}^{-1}+{2}\bigg] \frac{e^{19/300}\pi^{1/4}}{8\sqrt{3}}\frac{R^{5/2}({p-1})}{\sqrt{2}}\Big(e^{R^2}-1\Big)
\\\le& p\Big\{\KRsig +\frac{R^2}{\sigma}(1+\sigma/2)+ 2^{\frac{3}{p}}\sigma^{2/p-1}\frac{e^{19/300}\pi^{1/4}}{3}\frac{R^2}{2}\Big(e^{R^2}-1\Big)\log\Big(\frac{1+\sqrt{1/2}}{1-\sqrt{1/2}}\Big)\\&\quad+R^{-\frac{4}{p}} \bigg[\tilde R^{2/p}4^{1/p}\sigma^{-1}+{2}\bigg] \frac{e^{19/300}\pi^{1/4}}{8\sqrt{3}}\frac{R^{5/2}}{\sqrt{2}}\Big(e^{R^2}-1\Big)\Big\}.
    \end{align*}
  Now in the second case where $p>\sigma_{n,k}^2+1$, we have 
  \begin{align*}
      \omega_p^R(n,k,\sigma)&\le \sigma_{n,k}\sqrt{p-1}(1+\sqrt{2}R\sigma^{-1})
      \\&\le p(1+\sqrt{2}R\sigma^{-1}).
  \end{align*}
Hence the desired result directly follows.
\end{proof}

\subsection{Moderate deviation}\label{sec:app_moderate_deviation}
\begin{lemma}\label{lemma:moderate}
The following inequality holds for all $t \ge \frac{4e\tilde K_{R,\sigma}}{\sigma_{n,k}} + 2$
:
    \begin{align*}
    \mathbb{P}(\frac{1}{\sigma_{n,k}} W_k \ge t)\le\Phi^c(t)\Big[1+\frac{e\tilde K_{R,\sigma}\varphi(t)}{\Phi^c(t)\sigma_{n,k}}\Big]\Big(\frac{e\tilde K_{R,\sigma}\varphi(t)}{\sigma_{n,k}}\Big)^{-\frac{e\tilde K_{R,\sigma}(t+1)}{e\tilde K_{R,\sigma}(t+1) + \sigma_{n,k}}},
\end{align*}
where $\tilde K_{R,\sigma}$ is the constant in \Cref{lemma:Wp_order}.
\end{lemma}
\begin{proof}
Recall that, by \Cref{lemma:Wp_order}, we have the following result:
\begin{align*}
    \Big\|\mathcal{W}_p(W_k,\mathcal{N}(0,\sigma_{n,k}^2)|\mathcal{U}(X_{1:n}))\Big\|_p 
        \le \omega^R_p(n,k,\sigma)\le \tilde K_{R,\sigma}p.
\end{align*}
Note that for any $\epsilon > 0$, conditioned on $\mathcal{U}(X_{1:n})$, there exists a $Z_k \sim \mathcal{N}(0,\sigma_{n,k}^2)$ such that 
$$
\mathbb{E}\Big[|W_k-Z_k|^p\Big|\mathcal{U}(X_{1:n})\Big]^{1/p}\le  \mathcal{W}_p(W_k,\mathcal{N}(0,\sigma_{n,k}^2)|\mathcal{U}(X_{1:n})) + \epsilon
$$
It follows that, by Jensen's inequality,
\begin{align*}
    \mathbb{P}(\frac{1}{\sigma_{n,k}} W_k \ge t)\le& \mathbb{E}\Big[\mathbb{P}(\frac{1}{\sigma_{n,k}} W_k \ge t | \mathcal{U}(X_{1:n}))\Big]\\
    =& \mathbb{E}\Big[\mathbb{P}(\frac{1}{\sigma_{n,k}} (W_k - Z_k) +  \frac{1}{\sigma_{n,k}} Z_k\ge t | \mathcal{U}(X_{1:n}))\Big]\\
    \le & \mathbb{P}(\frac{1}{\sigma_{n,k}}Z_k > \rho t) + \mathbb{E}\Big[\mathbb{P}(\frac{1}{\sigma_{n,k}} (W_k - Z_k) \ge (1-\rho)t | \mathcal{U}(X_{1:n}))\Big]\\
    \le & \Phi^c(\rho t) + \mathbb{E}\Bigg[\frac{\frac{1}{\sigma_{n,k}^p} \Big(\mathcal{W}_p(W_k,\mathcal{N}(0,\sigma_{n,k}^2)|\mathcal{U}(X_{1:n}))+\epsilon\Big)^p}{ (1-\rho)^pt^p }\Bigg]
\end{align*}
Since this holds for all $\epsilon > 0$, we have
\begin{align*}
    \mathbb{P}(\frac{1}{\sigma_{n,k}} W_k \ge t)\le \inf_{p\ge2,\rho\in (0,1)}\Phi^c(\rho t) + \frac{\omega^R_p(n,k,\sigma)^p/\sigma_{n,k}^p}{ (1-\rho)^pt^p }.
\end{align*}
For a fix $p\ge 2$, choose $\rho = 1 - e\frac{\omega^R_p(n,k,\sigma)}{\sigma_{n,k}t}$ and obtain that 
$$
\mathbb{P}(\frac{1}{\sigma_{n,k}} W_k \ge t)\le \inf_{p\ge2,\rho\in (0,1)}\Phi^c(t - e\omega^R_p(n,k,\sigma)/{\sigma_{n,k}}) + e^{-p}.
$$
We next note that the function $t \mapsto \log(\Phi^c(t))$ is concave, since it is twice differentiable with second derivative
$\frac{\varphi(t)}{\Phi^c(t)}\big(t - \frac{\varphi(t)}{\Phi^c(t)}\big) \le 0$,
where the inequality follows from bounds on the Mills ratio of a normal random variable; see \cite{mills}. This directly implies that for all $y \ge 0$, we have
$$\Phi^c(t - y) \le \Phi^c(t)\exp\big(y\varphi(t)/\Phi^c(t)\big).$$

Plugging $y = e\omega_p^R(n,k,\sigma)/{\sigma_{n,k}}$ and using the fact that, according to \Cref{lemma:Wp_order}, we have $\omega^R_p(n,k,\sigma)/{\sigma_{n,k}}\le \tilde K_{R,\sigma}p/{\sigma_{n,k}}$, we obtain
\begin{align*}
    \mathbb{P}(\frac{1}{\sigma_{n,k}} W_k \ge t) \le& \inf_{p\ge2}\Phi^c(t)e^{ \frac{e\omega^R_p(n,k,\sigma)\varphi(t)}{\sigma_{n,k}\Phi^c(t)}} + e^{-p}\\
    \le& \inf_{p\ge2}\Phi^c(t)e^{ \frac{e\tilde K_{R,\sigma}p\varphi(t)}{\sigma_{n,k}\Phi^c(t)}} + e^{-p}\\
    \le& \inf_{x\ge e^2}\Phi^c(t)x^{ \frac{e\tilde K_{R,\sigma}\varphi(t)}{\sigma_{n,k}\Phi^c(t)}} + 1/x.
\end{align*}
Choosing $x: = \Big(\frac{\sigma_{n,k}}{e\tilde K_{R,\sigma}\varphi(t)}\Big)^{\frac{\sigma_{n,k}}{e\tilde K_{R,\sigma}(t+1) + \sigma_{n,k}}} \ge e^2$, which can be chosen because $t \ge \frac{4e\tilde K_{R,\sigma}}{\sigma_{n,k}} + 2$, we obtain
\begin{align*}
    \mathbb{P}(\frac{1}{\sigma_{n,k}} W_k \ge t)
    \le& \Phi^c(t)\Big(\frac{\sigma_{n,k}}{e\tilde K_{R,\sigma}\varphi(t)}\Big)^{\frac{e\tilde K_{R,\sigma}(t+1) }{e\tilde K_{R,\sigma}(t+1) + \sigma_{n,k}}}  + \Big(\frac{e\tilde K_{R,\sigma}\varphi(t)}{\sigma_{n,k}}\Big)^{1-\frac{e\tilde K_{R,\sigma}(t+1)}{e\tilde K_{R,\sigma}(t+1) + \sigma_{n,k}}}\\
    \le& \Phi^c(t)\Big[1+\frac{e\tilde K_{R,\sigma}\varphi(t)}{\Phi^c(t)\sigma_{n,k}}\Big]\Big(\frac{e\tilde K_{R,\sigma}\varphi(t)}{\sigma_{n,k}}\Big)^{-\frac{e\tilde K_{R,\sigma}(t+1)}{e\tilde K_{R,\sigma}(t+1) + \sigma_{n,k}}}.
\end{align*}
Note that we used the Mills ratio bound $\frac{\varphi(t)}{\Phi^c(t)}\le t+1$.
\end{proof}

\section{Proofs of results in Section \ref{theoretical_guarantees}}\label{triste}
\begin{proof}
    [Proof of \Cref{ts11}]
The proof proceeds by induction. First, if $L=0$, then we remark that $n=1$ and $W_k=0$. Therefore, if we choose $\tilde Z_k=0$, we know that it satisfies \Cref{coupled} with $\mathcal{I}=\Iintv{0,2^L}$ and is such that 
$$\mathbb{P}\Big(\exists k\le 2^L~\textrm{s.t.}~|W_k-\tilde Z_k|>\delta^L_k\Big)=0=\beta_0.$$

If $L\ge 1$, suppose there exists a Gaussian vector $(\tilde Z^1_k)_{k\le 2^{L-1}}$ satisfying \Cref{coupled} with $\mathcal{I}=\Iintv{0,2^{L-1}}$,  and such that $$\mathbb{P}\Big(\exists k\le 2^{L-1}~\textrm{s.t.}~|W_k-\tilde Z_k|>\delta^{L-1}_k\Big)\le \beta_{L-1}.$$ Then according to \Cref{ausecours} there exists $(\tilde Z_k)_{k\le 2^L}$ satisfying \Cref{coupled} with $\mathcal{I}=\Iintv{1,2^L}$ such that $$\mathbb{P}\Big(\exists k\le 2^{L-1}~\textrm{s.t.}~|W_k-\tilde Z_k|>\delta^{L}_k\Big)\le 2\beta_{L-1}-\beta_{L-1}^2+\alpha_0.$$ Now, by definition, $2\beta_{L-1}-\beta_{L-1}^2+\alpha_0=\beta_L,$ which concludes the proof.\end{proof}

\begin{proof}
    [Proof of \Cref{all_too_well}]First, according to \Cref{lemma:bound_S_and_Z}, we know that there exists $Z_n\in \sigma(S_n, U_n)$ such that $$Z_n\sim \mathcal{N}(0,\sigma^2 n)$$ and, for $p\ge 2$, we have $$\mathbb{P}\Big(|S_n-Z_n|\ge \frac{\mathcal{W}_p(S_n, \mathcal{N}(0,\sigma^2n))}{\alpha_1^{1/p}}\Big) \le \alpha_1.$$ Now, by \Cref{nelson3}, we know that for all $p\ge 2$, we have $$\mathcal{W}_p(S_n, \mathcal{N}(0,\sigma^2n))\le s_p^R(n,\sigma).$$ Hence, this directly implies that $$\mathbb{P}\Big(|S_n-Z_n|\ge \delta^*(\alpha_1)\Big)\le \alpha_1.$$ Now according to \Cref{ts11}
there exists a Gaussian vector $(\tilde Z_k)_{k\le n}$ satisfying \Cref{coupled} with $\mathcal{I}=\Iintv{1,n}$ such that $$\mathbb{P}\Big(\exists k~\textrm{s.t.}~|W_k-\tilde Z_k|\ge \delta_k\Big)\le \beta_L.$$ Now note that \Cref{coupled} implies $(\tilde Z_k)|Z_n\sim \mathcal{N}(0,\tilde \Sigma)$  where $\tilde \Sigma_{i,j}=\sigma^2\frac{i\wedge j(n-j\vee i)}{n}$. Hence, using the conjugacy of Gaussian priors with Gaussian likelihoods, we obtain $$\Big(\tilde Z_k+\frac{k}{n}Z_n\Big)_{k\le n}\sim \mathcal{N}(0,\Sigma_n).$$ Hence if we define $$Z_k:=\tilde Z_k+\frac{k}{n}Z_n,$$ then by a union bound argument we have 
\begin{align*}
    \mathbb{P}\Big(\exists k\le n~|Z_k-S_k|\le \delta_k+\frac{k}{n}\delta_1^*(\alpha_1)\Big)\le&  \mathbb{P}\Big(|S_n-Z_n|\le\delta_1^*(\alpha_1)\Big)
    \\&+ \mathbb{P}\Big(\exists k\le n~|\tilde Z_k-W_k|\le \delta_k\Big)
    \\\le& \beta_L+\alpha_1.
\end{align*}\end{proof}
\begin{proof}
    [Proof of \Cref{thm:main3}]By the definition of $ \nu^*_0$, we know that $\beta_L(\nu_0^*)\le \alpha.$ Hence, according to \Cref{ts11}, there exists $\tilde Z:=(\tilde Z_k)_{0\le k\le n}$ with the same covariance as $(W_k)$, such that the following inequality holds:
\begin{equation*}
    \mathbb{P}(\exists k\le n\textrm{ s.t. }\big|W_k-\tilde {Z}_k\big|\ge \Delta_k(\alpha))\le\beta_L(\nu_0^*)\le \alpha.
\end{equation*}
\end{proof}
\begin{proof}[Proof of \Cref{thm:main2}] A direct consequence of \Cref{thm:main3} combined with \Cref{all_too_well}.
    
\end{proof}

\section{Proof of Theorem \ref{empirical}}\label{empirical_proof}\begin{proof}Define $\tilde W_k:=\sigma^{-1}W_k$ and $\tilde S_k:=\sigma^{-1}S_k.$ Using \Cref{thm:main2,thm:main3} we know that there exists centered Gaussian vectors $(\tilde Z^*_k),(Z^*_k)$ that have the same variance as respectively $(\tilde W_k)$ and $(\tilde S_k)$ such that 
    $$ \mathbb{P}\Big(\exists k\le n~\textrm{s.t.}~|\tilde S_k-Z^*_k|\ge\mathcal{D}_k(\alpha(1-\rho),\frac{R}{\sigma},1)\Big)\le \alpha(1-\rho),$$ and  $$ \mathbb{P}\Big(\exists k\le n~\textrm{s.t.}~|\tilde W_k-\tilde Z^*_k|\ge\Delta_k(\alpha(1-\rho),\frac{R}{\sigma},1)\Big)\le \alpha(1-\rho).$$
    Define $Z_k:=\sigma Z_k^*$ and $\tilde Z_k:=\sigma \tilde Z_k^*$. We remark that $(Z_k)$ and $(\tilde Z_k)$ are centered Gaussian vectors that have the same variance as $(S_k)$ and $(W_k).$ Moreover, we note that 
    \begin{align*}
        & \mathbb{P}\Big(\exists k\le n~\textrm{s.t.}~| S_k-Z_k|\ge\sigma \mathcal{D}_k(\alpha(1-\rho),\frac{R}{\sigma},1)\Big)
         \\= &\mathbb{P}\Big(\exists k\le n~\textrm{s.t.}~| \tilde S_k-Z^*_k|\ge \mathcal{D}_k(\alpha(1-\rho),\frac{R}{\sigma},1)\Big)\\
         \le& \alpha(1-\rho).\end{align*}
         Similarly, we have that  $$ \mathbb{P}\Big(\exists k\le n~\textrm{s.t.}~|\tilde W_k-\tilde Z^*_k|\ge\Delta_k(\alpha(1-\rho),\frac{R}{\sigma},1)\Big)\le \alpha(1-\rho).$$
         Moreover, according to \Cref{increasing_in_R}, we know that $R\mapsto \Delta_k(\alpha(1-\rho),R,1)$ is an increasing function of $R$. Hence the following holds
         \begin{align*}&
              \mathbb{P}\Big(\exists k\le n~\textrm{s.t.}~|S_k-Z_k|\ge \hat\sigma_k^U~\mathcal{D}_k(\alpha(1-\rho),\tilde R_k,1)\Big)
              \\\le& \mathbb{P}\Big(\exists k~\textrm{s.t.}~\sigma\not\in [\sigma_k^L,\sigma_k^U]\Big)+ \mathbb{P}\Big(\exists k\le n~\textrm{s.t.}~|S_k-Z_k|\ge \sigma~\mathcal{D}_k(\alpha(1-\rho),R/\sigma,1)\Big)
              \\\le& \rho\alpha+(1-\rho)\alpha=\alpha.
         \end{align*}Similarly, we can prove that 
          \begin{align*}&
              \mathbb{P}\Big(\exists k\le n~\textrm{s.t.}~|W_k-\tilde Z_k|\ge \hat\sigma_k^U~\Delta_k(\alpha(1-\rho),\tilde R_k,1)\Big)
              \le \alpha.
         \end{align*}
\end{proof}
\section{Proof of Theorem \ref{detection} }\label{marre_proof}
\begin{proof} Let $\delta_1,\delta_2\in (0,1)$ be such that $\delta>\delta_1+\delta_2$ and set $\beta>1$. We denote $\delta_3:=\delta-\delta_1-\delta_2,$ and define $\delta_{2,i}:=\frac{\delta_2\beta^{-i}}{\sum_{j\in\mathbb{N}}\beta^{-j}}.$

The main idea of the proof is to approximate the statistics $\mathcal{T}_{s,t}$ by the same statistics defined on Gaussians instead. The latter is then bounded using \cite{maillard:tel-02162189}. More precisely, by a union bound argument, we remark that 
\begin{align*}
    &\mathbb{P}\Big(\exists s<t~\textrm{s.t.}~|\mathcal{T}_{s,t}|\ge \mathcal{C}_{s,t}\Big)\\\le&   \mathbb{P}\Big(\exists s<t~\textrm{s.t.}~|\mathcal{T}_{s,t}|\ge \mathcal{C}_{s,t}\text{ and }\forall k,~\sigma\in [\hat\sigma_k^L,\hat\sigma_k^U]\Big)+\mathbb{P}\Big(\exists k~\textrm{s.t.}~\sigma\not\in [\hat\sigma_k^L,\hat\sigma_k^U]\Big)
    \\\le& \delta_1+\mathbb{P}\Big(\exists s<t~\textrm{s.t.}~|\mathcal{T}_{s,t}|\ge \mathcal{C}_{s,t}\text{ and } \forall k,~\sigma\in [\hat\sigma_k^L,\hat\sigma_k^U]\Big).
\end{align*}
To further bound this, we first remark that $$\mathcal{T}_{s,t}=\frac{t}{s(t-s)}\big(S_s-\frac{s}{t}S_t),$$ and hence resembles the quantities controlled in \Cref{ts12}. A difficulty, however, is that \Cref{ts12} controls the variation of $S_k-\frac{k}{n}S_n$ uniformly over $k$, but only for a fixed sample size $n$. In this setting, we instead need to control $\mathcal{T}_{s,t}$ uniformly over both $s$ and $t$. To overcome this limitation, we first control $(\mathcal{T}_{s,t})_{s<t}$ for different $t$ that are expressible using powers of $2$, and then extend this control to all $t$.

In this goal, we define $\mathcal{I}^{-}_k=\sum_{i\le \ell_{\mathbf{L}}(k)}2^{L_i}$ and $\mathcal{I}^+_k=\sum_{i\le u_{\mathbf{L}}(k)}2^{L_i}$, and remark that $\mathcal{I}^{-}_k\le k\le \mathcal{I}^{+}_k$. 
Moreover, we define the increments $$S^*_{\mathcal{I}^{+}_k}:=S_{\mathcal{I}^+_k}-S_{\mathcal{I}^-_k},\qquad \tilde W_k^{u_{\mathbf{L}}(k)}:=S_k-S_{\mathcal{I}^-_k}-\frac{k-\mathcal{I}^-_k}{\mathcal{I}^+_k-\mathcal{I}^-_k}S^*_{\mathcal{I}^{+}_k}.$$ 
For all $i\in\mathbb{N}$, we write $N_i=\sum_{j\le i}2^{L_j}$ and let $(\tilde \Delta_{k}(\delta_{2,i}))_{u_{\mathbf{L}}(k)=i}$ and $\tilde \delta^*_{N_i}(\delta_{2,i})$ the output of \Cref{ts13} with confidence $\delta_{2,i}$, $n_1=N_{i-1}$, and $n_2=N_i$. Note that by definition, $I_k^+ = N_{u_{\mathbf{L}}(k)}$ and $I_k^- = N_{\ell_{\mathbf{L}}(k)}$. For the ease of notation, we shorthand $\tilde \delta_i^*:=\tilde \delta^*_{N_i}(\delta_{2,i})$ and $\tilde \Delta_k:=\tilde \Delta_k(\delta_{2,i}).$ 
We remark that, for all $i\in \mathbb{N}$, according to  \Cref{partir} there exist Gaussian vectors $(\tilde Z^i_k)$ and $ (Z^{*}_i)$ that have the same mean and variance as  $(\tilde W^i_k)$ and $ (S^{*}_i)$  such that \begin{equation}\label{smt2}\mathbb{P}\Big(\max_{ k,~ {u^{\mathbf{L}}}(k)=i}~(\tilde \Delta_{k})^{-1}\big|\tilde W_k^i-\tilde Z_k^i\big|\le 1,\quad |S^{*}_i-Z^{*}_i|\le \tilde \delta_{N_i}^*,~\sigma\in \bigcap_k [\hat\sigma_k^L,\hat\sigma_k^U]\Big)\ge 1-\delta_{2,i}.
\end{equation} 
 For each $k$ we define $$Z_k:=\sum_{j\le \ell_{\mathbf{L}(k)}}Z^*_j+\tilde Z_k^{u_{\mathbf{L}}(k)} +\frac{s}{\mathcal{I}_k^+-\mathcal{I}_k^-}Z^*_{\mathcal{I}_k^+}.$$ We note that $(Z_k)$ is a Gaussian process that has the same mean and variance as $(S_k).$

Interestingly, both $\mathcal{T}_{s,t}$ and $\mathcal{T}_{s,t}(Z):=\frac{1}{s}Z_s-\frac{1}{t-s}(Z_t-Z_s)$ can be re-expressed in terms of $(\tilde W_k^i)$, $(S^*_i)$, $(\tilde Z_k^i)$, and $(Z^*_i)$. Indeed, we have:
\begin{align*}
  &\frac{s(t-s)}{t}\mathcal{T}_{s,t}
  \\=&S_s-\frac{s}{t}S_t
  \\=&\tilde W^{u_{\mathbf{L}}(s)}_s-\frac{s}{t}\tilde W^{u_{\mathbf{L}}(t)}_t+\sum_{k\le \ell_{\mathbf{L}}(s)}S_k^*\big(1-\frac{s}{t}\big)+\Big(\frac{s-I_s^-}{I_s^+-I_s^-}-\frac{s}{t}\Big
  )S^*_{u_{\mathbf{L}(s)}}\\&\qquad-\frac{s}{t}\sum_{u_{\mathbf{L}}(s)+1\le k\le \ell_{\mathbf{L}}(t)}S_k^*-\frac{s}{t}\frac{t-I_t^-}{I_t^+-I_t^-}S^*_{u_{\mathbf{L}}(t)},
\end{align*} and \begin{align*}&
    \frac{s(t-s)}{t}  \mathcal{T}_{s,t}(Z)
  \\=&\tilde Z^{u_{\mathbf{L}}(s)}_s-\frac{s}{t}\tilde Z^{u_{\mathbf{L}}(t)}_t+\sum_{k\le \ell_{\mathbf{L}}(s)}Z_k^*\big(1-\frac{s}{t}\big)+\Big(\frac{s-I_s^-}{I_s^+-I_s^-}-\frac{s}{t}\Big
  )Z^*_{u_{\mathbf{L}(s)}}\\&\qquad-\frac{s}{t}\sum_{u_{\mathbf{L}}(s)+1\le k\le \ell_{\mathbf{L}}(t)}Z_k^*-\frac{s}{t}\frac{t-I_t^-}{I_t^+-I_t^-}Z^*_{u_{\mathbf{L}}(t)}.
\end{align*}
Using \cref{smt2} combined with a union bound, we obtain that 
\begin{align*}
    \mathbb{P}\Big(\exists s\le t~\textrm{s.t.} \Big|\mathcal{T}_{s,t}-\mathcal{T}_{s,t}(Z)\Big|\ge g_\beta(t,s,\delta_2) \frac{t}{s(t-s)},~\sigma\in \bigcap_k [\hat\sigma_k^L,\hat\sigma_k^U]\Big)&\le \delta_2,
\end{align*} where $g_\beta(t,s,\delta_2)$ is defined as follows: if $u_{\mathbf{L}}(s)\ne u_{\mathbf{L}}(t)$, then
\begin{align*}
    g_\beta(t,s,\delta_2)=&\sum_{k\le \ell_{\mathbf{L}}(s)}\tilde \delta_{k}^*\big(1-\frac{s}{t}\big)+\frac{s}{t}\sum_{u_{\mathbf{L}}(s)+1\le k\le \ell_{\mathbf{L}}(t)}\tilde \delta_k^*+\frac{s}{t}\frac{t-I_t^-}{I_t^+-I_t^-}\tilde \delta^*_{u_{\mathbf{L}}(t)}\\&+\Big(\frac{s-I_s^-}{I_s^+-I_s^-}-\frac{s}{t}\Big
  )\tilde\delta^*_{u_{\mathbf{L}(s)}}+\tilde \Delta^{u_{\mathbf{L}}(s)}_{s}+\frac{s}{t}\tilde \Delta^{u_{\mathbf{L}}(t)}_{t},
\end{align*} and if $u_{\mathbf{L}}(s)= u_{\mathbf{L}}(t)$, then \begin{align*}
    g_\beta(t,s,\delta_2)=&\sum_{k\le \ell_{\mathbf{L}}(s)}\tilde \delta_{k}^*\big(1-\frac{s}{t}\big)+\frac{s}{t}\tilde\delta^*_{N_{u_{\mathbf{L}(s)}}}+\tilde \Delta^{u_{\mathbf{L}}(s)}_{s}+\frac{s}{t}\tilde\Delta^{u_{\mathbf{L}}(t)}_{t}.
\end{align*}
Moreover according to Theorem 5 of \cite{maillard:tel-02162189} we know that \begin{align*}&
    \mathbb{P}\Big(\exists s\le t ~\textrm{s.t.}~|\mathcal{T}_{s,t}(Z)|\ge\hat\sigma_t^U\sqrt{\frac{2(1+\frac{1}{t-s})}{t-s}\log(\frac{t (\log t)^2\sqrt{t+1-s}}{\delta_3\log 2})} , \sigma\in \bigcap_k [\hat\sigma_k^L,\hat\sigma_k^U]\Big)\\\le&    \mathbb{P}\Big(\exists s\le t ~\textrm{s.t.}~|\mathcal{T}_{s,t}(Z)|\ge\sigma \sqrt{\frac{2(1+\frac{1}{t-s})}{t-s}\log(\frac{t (\log t)^2\sqrt{t+1-s}}{\delta_3\log 2})}\Big) \\
    \le& \delta_3,
\end{align*}
where the last inequality is due to Theorem 5 of \cite{maillard:tel-02162189}.
Therefore, by a union bound argument, we obtain that 
\begin{align*}
   \mathbb{P}_{H_0}\Big(\exists s<t~\textrm{s.t.}~|\mathcal{T}_{s,t}|\ge \mathcal{C}_{s,t}\Big)\le \delta. 
\end{align*}
\end{proof}

 \section{Proof of Theorem \ref{hitting_prob}}\label{morgane_push_proof}
\begin{proof} For each $N$, we define $Y^N_i:=X_i^N+\frac{R}{2}$ and  remark that $(Y^N_i)$ are almost surely positive
 $Y^N_1\overset{a.s}{\ge }0,$ and are bounded by $R$ meaning that $Y^N_1\overset{a.s}{\le }R.$ Moreover we write$$\bar{W}_i^N:=W_i^N-i\mu_N=\sum_{j\le i}Y_j^N-\mathbb{E}[Y_j^N].$$ Using \Cref{thm:main2}, we know that there exists a Gaussian vector $(Z_i^N)$ with same mean and variance as $(\bar{W}_i^N)$ such that for all $\alpha>0$, we have
 \begin{align*}
        \mathbb{P}\Big(\forall {i\le N}\big|\bar{W}_i^N-Z_i^N\big|\ge \mathcal{D}_i(\alpha)\Big)\le \alpha.
    \end{align*}
 
    Hence we obtain that \begin{align*}
        \mathbb{P}(\tau_N\ge N)&=\mathbb{P}(\forall i\le N~ W_i^N\le g_i^N)
        \\&\le \mathbb{P}\Big(\max_{i\le N}\big|\bar{W}_i^N-Z_i^N\big|\ge \mathcal{D}_i(\alpha)\Big)\\&+\mathbb{P}\Big(\forall i\le N,~ Z_i^N+i\mu_n\le g_i^N+\mathcal{D}_i(\alpha)\Big)
        \\&\le \alpha+\mathbb{P}\Big(\forall i\le N~ Z_i^N+i\mu_N\le g_i^N+\mathcal{D}_i(\alpha)\Big).
    \end{align*}
    Now remark that $(\sigma_N^{-1}Z_i^N)\overset{d}{=}(B_t)_{t\in \mathbb{N},t\le N}$. Hence we obtain that  \begin{align*}
        \mathbb{P}(\tau_N\ge N)&=\mathbb{P}(\forall i\le N~ W_i^N\ge g_i^N)
        \\&\le \alpha+\mathbb{P}\Big(\forall i\le N,~ B_i\le \sigma_N^{-1}\{g_i^N-i\mu_N+\mathcal{D}_i(\alpha)\}\Big).
    \end{align*}This establishes the first claim of \Cref{hitting_prob}. The second claim is a direct consequence of \Cref{lemma:Delta_order} which implies that there exists a constant $\kappa^R<\infty$, that does not depend on $N$, such that \begin{align*}
       \max_{i\le N}\big|\mathcal{D}_i(\alpha)\big|\le \log(N)\kappa_R(\log(N)-\log(\alpha)).
    \end{align*}
\end{proof}
\section{Additional algorithms}\label{section:proofs_add_algo}In this section, we present some additional algorithms used in \Cref{change_oint} and mentioned in \Cref{turing}. 
\subsection{Algorithm used in Section \ref{change_oint}}
\begin{algorithm}[H]
\caption{Computing $(\tilde \Delta_k(\alpha))_{k\in (n_1,n_2]},\tilde\delta^*_{n_2}(\alpha)$ knowing $([\hat\sigma_k^L,\hat\sigma_k^U])$}\label{alg:main_delta_computation}\label{ts13}
\DontPrintSemicolon
\KwIn{Positive interval \((n_1,n_2] \), $R>0$, $\sigma>0$, total probability budget $\alpha > 0$ and intervals $([\hat\sigma_k^L,\hat\sigma_k^U])$}
\KwOut{Thresholds $(\tilde \Delta_k(\alpha))_{k\in (n_1,n_2]}$, $\tilde\delta_{n_2}^*(\alpha)$ }

\textbf{Define:} \(L:=\lceil \log_2(n_2-n_1)\rceil\)\; 

Let $\alpha_1^* \gets \alpha - \alpha_0^*$. \;

\tcp{Step 1: For any candidate split $\alpha_0 \in (0, \alpha)$, compute thresholds}
Given $\alpha_0 \in (0, \alpha)$  and $\alpha_1 \gets \alpha - \alpha_0$: 

\Indp
\tcp{Step 1a: Compute} \For{$k = 1$ \KwTo $n_2-n_1$}{
    $\Delta_{k}(\alpha_0) \gets$ output of \Cref{ts122} with input  $n_2-n_1$, $R/\hat\sigma_{k+n_1}^L$, $\sigma$, $n_2-n_1$ and $\alpha_0$\;}
\tcp{Step 1b: Calculate $\delta^\star$ using $\alpha_1$}
$\delta^\star(\alpha_1) \gets \hat\sigma_{n_2}^U\min_{2 \le p}\frac{S_p^{R/\hat\sigma_{n_2}^L}(n,1)}{(\alpha_1)^{1/p}}$\;

\tcp{Step 1c: Compute $\mathcal{D}_k(\alpha)$}
\For{$k = 1$ \KwTo $n$}{
$\mathcal{D}_k(\alpha):=\tilde \Delta_k(\alpha)+ \frac{k}{n}\delta^\star(\alpha_1)$}

\Indm

\tcp{Step 2: Optimize the split}
Find $\alpha_0^* \in (0, \alpha)$ that minimizes $\max_{k \le n} \mathcal{D}_k(\alpha)$.\;
Set $\alpha_1^* := \alpha - \alpha_0^*$\;

\tcp{Step 3: Compute Final output}
Set $\tilde \Delta_k(\alpha) := \hat\sigma_k^U \cdot \Delta_{k - n_1}(\alpha_0^*)$ for all $k \in (n_1, n_2]$.\; Set $\tilde \delta_{n_2}^*(\alpha) := \delta^\star(\alpha_1^*)$\;

\Return{$(\tilde \Delta_k(\alpha))_{k\in (n_1,n_2]}$, $\tilde\delta_{n_2}^*(\alpha)$}
\end{algorithm}

\subsection{Alternative algorithms mentioned in Theorem \ref{turing}}\label{turing_details}In this subsection we will let $(C_\ell)$ be a sequence of prespecified constants; and $(\zeta_i(\cdot))$ be any functions such that $$\mathbb{P}\big(|W_{n/2}-\tilde Z_{n/2}|\ge \zeta_{n/2}(\alpha)\big)\le \alpha\qquad \forall \alpha\in (0,1), n\in \mathbb{N}.$$ For example we can choose $C_\ell=\beta^\ell$ for some $\beta>1$ and $\zeta_{n/2}(\alpha)\le Rn/2+\sigma n/2\Phi^{-1}(\alpha).$

\begin{algorithm}[H]
\caption{ Computing $(\Delta_k(\alpha))_{k\le n}$}\label{alg:main_delta_computation_alt}\label{ts122_r}
\DontPrintSemicolon
\KwIn{An integer \(n >0\), total probability budget $\alpha > 0$, a prespecified sequence $(C_\ell)$}
\KwOut{Thresholds $(\Delta_k(\alpha))$}

\textbf{Define:}\(L:=\lceil \log_2(n)\rceil\)\;

\tcp{Step 1: Determine $\nu_0$ based on the chosen $\alpha$}
\textbf{Initialize for $\nu_0$ search:} Choose an initial guess for \( \nu_0 \) and set $\beta_0(\nu_0)=0$\;

\textbf{Compute $\beta_k$ recursively:}

\For{$k = 1$ \KwTo $L$}{
    $\beta_k(\nu_0) \gets 2\beta_{k-1}(\nu_0) - \beta_{k-1}^2(\nu_0) + C_{k} \nu_0$\;
}
\textbf{Finalize $\nu_0$:} Find the largest \( \nu_0^* \) such that \( \beta_L(\nu_0^*) < \alpha\)\;

\tcp{Step 2: Calculate intermediate $\delta^M_k$ values using $\nu_0^*$}
$\delta_1^0 \gets 0$\;
\For{$M = 1$ \KwTo $L$}{
    $\delta_{2^{M-1}}^M\gets  \min\big(\min_{2 \le p }\frac{\omega_p^R(2^M,\sigma)}{(C_M\nu_0^*)^{1/p}},\zeta_{M}(C_M\nu_0^*))$\;
    \For{$k = 1$ \KwTo $2^{M-1}$}{
        \(\delta_k^M \gets \delta_k^{M-1} + \frac{k}{2^{M-1}} \delta_{2^{M-1}}^M\)\;
    }
    \For{$k = 2^{M-1}+1$ \KwTo $2^M-1$}{
        \(\delta_k^M \gets \delta_{2^{M}-k}^{M-1} + \frac{2^{M}-k}{2^{M-1}} \delta_{2^{M-1}}^M\)\;
    }

}
Let $\delta^L_k$ denote the final values $\delta_k^L$ from this step.


\For{$k = 1$ \KwTo $2^L$}{
    $\Delta_k(\alpha) \gets \delta^L_k$\;
}

\Return{$(\Delta_k(\alpha))_{k\le n}$}
\end{algorithm}

\begin{algorithm}[H]
\caption{ Computing $(\mathcal{D}_k(\alpha))_{k\le n}$}\label{ts12_r}
\DontPrintSemicolon
\KwIn{An integer \(n>0\), $R>0$, $\sigma>0$, total probability budget $\alpha > 0$,and a pre-specified sequence of constants $(C_\ell)$}
\KwOut{Thresholds $(\mathcal{D}_k(\alpha))$ minimized by optimal choice of $\alpha_0$}

\textbf{Define:} \(L:=\lceil \log_2(n)\rceil\)\;

\tcp{Step 1: For any candidate split $\alpha_0 \in (0, \alpha)$, compute thresholds}
Given $\alpha_0 \in (0, \alpha)$ and $\alpha_1 := \alpha - \alpha_0$: 

\Indp
    \tcp{Step 1a: Compute $\Delta_k$ using $\alpha_0$}
    \For{$k = 1$ \KwTo $n$}{
        $\Delta_k(\alpha_0) \gets \text{output of \Cref{ts122_r} with input } n \text{, } \alpha_0$ and $(C_\ell)$\;
    }

    \tcp{Step 1b: Compute $\delta^\star$ using $\alpha_1$}
    $\delta^\star(\alpha_1) \gets \min_{2 \le p}\frac{s_p^R(n,\sigma)}{(\alpha_1)^{1/p}}$\;

    \tcp{Step 1c: Compute $\mathcal{D}_k(\alpha)$, values to be minimized}
    \For{$k = 1$ \KwTo $n$}{
        $\mathcal{D}_k(\alpha) \gets \Delta_k(\alpha_0) + \frac{k}{n}\delta^\star(\alpha_1)$\;
    }
\Indm
\tcp{Step 2: Optimize the split}
Find $\alpha_0^* \in (0, \alpha)$ that minimizes $\max_{k \le n} \mathcal{D}_k(\alpha)$. Let $\alpha_1^* := \alpha- \alpha_0^*$

$\mathcal{D}_k(\alpha) \gets \Delta_k(\alpha^*_0) + \frac{k}{n}\delta^\star(\alpha_1^*)$\;

\tcp{Step 3: Return optimized thresholds}
\Return{$(\mathcal{D}_k(\alpha))_{k \le n}$}
\end{algorithm}

It is a straightforward to check that the output of these algorithms give us the same guarantees than we have for \Cref{ts12,ts122}.
\begin{proposition}
    Assume that the conditions of \Cref{thm:main2} hold. Let $\alpha>0$ and $L\in\mathbb{N}$. Set $n=2^L$. Define $
(\Delta^b_k(\alpha))$ and $
(\mathcal{D}^b_k(\alpha))$ as the output of respectively \Cref{ts122_r} and \Cref{ts12_r}.
Then there exist  centered Gaussian vector $(Z_k)_{0\le k\le n}$ and $(\tilde Z_k)_{0\le k\le n}$ with covariance $\mathrm{Var}(Z)=\Sigma_n$ and $\mathrm{Var}(\tilde Z)=\mathrm{Var}(W)$ such that the following inequality holds:
\begin{align*}
    &\mathbb{P}(\exists k\le n\textrm{ s.t. }\big|S_k-{Z}_k\big|\ge \mathcal{D}_k(\alpha))\le\alpha,
   \\& \mathbb{P}(\exists k\le n\textrm{ s.t. }\big|W_k-\tilde{Z}_k\big|\ge \Delta_k(\alpha))\le\alpha.
\end{align*}
\end{proposition}\begin{proof}
    The proof of this result is a trivial adaptation of the proof of \Cref{thm:main2,thm:main3}.
\end{proof}
 \section{\texorpdfstring{Computable bound for $\mathcal{W}_p(W_k,\mathcal{N}(0,\sigma^2\frac{k(n-k)}{n}))$}
{Computable bound for Wasserstein distance}}
\label{section:proofs_computable}
In this subsection we present a computable upper bound for $\mathcal{W}_p(W_k,\mathcal{N}(0,\sigma^2\frac{k(n-k)}{n}))$.\subsection{Notations}We first present some constants that appear in \Cref{nelson2}.
\begin{align*}
    D_{n,p}:=\frac{1}{2\sqrt{n}}\min
\begin{cases}
\sqrt{p-1}\max\Big((\tilde R_s^2-1)^{1-1/p},\Big[\frac{ (\tilde R_s^2-1)^{p}+\tilde R_s^2-1}{\tilde R_s^2}\Big]^{1/p}\Big)\\
\max\Big((\tilde R_s^2-1)^{1-1/p} ,\Big[\frac{ (\tilde R_s^2-1)^{p}\tilde R_s^2-1}{\tilde R_s^2}\Big]^{1/p}\Big)A^*_{n,p}
+\sqrt{\tilde R_s^2-1}A_p\\\sqrt{2}\Big(\frac{\Gamma(\frac{p+1}{2})}{\sqrt{\pi}}\Big)^{1/p}\tilde R_s^{2(1-2/p)}(\tilde R_s^2-1)^{1/p} \\\sqrt{2}2^{-1/p}\Big(\frac{\Gamma(\frac{p+1}{2})}{\sqrt{\pi}}\Big)^{1/p}\frac{\tilde R^2}{\sqrt{n}}\sqrt{\|\Bin(n,\frac{2(\tilde R_s^2-1)}{\tilde R_s^4})\|_{p/2}}\quad\stext{if} p\ge4.\end{cases} 
\end{align*}Moreover we write \begin{align*}
   C_{n,p}:= 
\min\begin{cases}\tilde A_{n,p}\tilde R 2^{1/p}+\sqrt{2}A_{n,p}\\{\sqrt{2}}\Big(\frac{\Gamma(\frac{p+1}{2})}{\sqrt{\pi}}\Big)^{1/p}\begin{cases}2^{1/p}\tilde R^{1-2/p}\text{\quad if }p<4\\2^{1/p}\tilde R^{1-2/p} \wedge \frac{\tilde R}{\sqrt{n}}\sqrt{\|\Bin(n,\frac{2}{\tilde R^2})\|_{p/2}}\quad \text{otherwise};\end{cases} \end{cases} 
\end{align*}and $$B_{p,n}:= \min\Big(\frac{\tilde R^2}{n}\|\Bin(n,\frac{2}{\tilde R^2})\|_p, 1+\frac{1}{\sqrt{n}}\tilde U_{n,p}\tilde R\Big).$$
\subsection{Computable bound}
\begin{lemma}\label{nelson2} Let $Z\sim \mathcal{N}(0,1)$ be an independent standard random variable.
Let $p\ge 2$ be a constant. 
   Then for all $k\le n$ we have that \begin{align*}
&\mathcal{W}_p(W_k,\mathcal{N}(0,\sigma_{n,k}^2))\\&\le \mathcal{S}_p^{\sigma,R}(n,k):=\sqrt{p-1}\sqrt{n}(R+\sigma)\wedge\begin{cases}\max\big(\tilde \omega_p^R(\sigma,k),\tilde \omega_p^R(\sigma,n-k))~\rm{if}~k\ne n/2\\\frac{1}{2}\{\tilde \omega_p^R(\sigma,n)\}~\textrm{otherwise}.\end{cases}   
   \end{align*} where\begin{align*}&\sigma^{-1}\tilde \omega_p^R(\sigma,n)\\&:=\inf_{\substack{\kappa\le \tilde R^2/n\\K_p\ge 1}} \frac{\sqrt{n}}{\mnkappa\mathbb{I}(p\ne 2)+\mathbb{I}(p=2)}\Big\{\|Z\|_p\Big(\frac{\pi}{2}-\sin^{-1}(\sqrt{1-\frac{\tilde R^2}{n\kappa}})\Big)+\|Z\|_p D_{n,p} M_{n,\kappa}^2\\&+  \frac{B_{p,n}}{2} \Big\{\sum_{\substack{K_p>j\ge 1}}\frac{R^{2k}}{\sqrt{n}}\Big(\frac{\|H_{2j+1}\|_p}{(2j+2)!}-
   \frac{2^{-j}e^{19/300}\pi^{1/4}K_p^{1/4}\sqrt{p-1}^{2k+1}}{2(K_p+1)\sqrt{2K_p+1}j!}\Big)\int_{{\frac{\tilde R^2}{\kappa}}}^n\frac{1}{\sqrt{y}}\Big(\frac{1}{y}-\frac{1}{n}\Big)^jdy
   \\&\quad+\frac{1}{4}e^{19/300}\pi^{1/4} \frac{K_p^{1/4}\sqrt{p-1}}{(K_p+1)\sqrt{2K_p+1}}\frac{1}{\sqrt{n}}\int_{{\frac{\tilde R^2}{\kappa}}}^n\frac{1}{\sqrt{y}}\big(e^{\frac{1}{2}(p-1)\tilde R^2(\frac{1}{y}-\frac{1}{n})}-1\big)dy\Big\}
\\&+\frac{C_{n,p}}{2}\Big\{\sum_{\substack{1\le j\le K_p-1}}
 \frac{\tilde R^{2j}}{n}\Big(\frac{\|H_{2j}\|_p}{(2j+1)!}-\frac{K_p^{1/4}2^{-j}(p-1)^je^{19/300}\pi^{1/4}}{(2K_p+1)j!}\Big)\int_{\frac{\tilde R^2}{\kappa}}^n\frac{1}{\sqrt{y}}\big(\frac{1}{y}-\frac{1}{n}\big)^{j-\frac{1}{2}}dy
 \\&\quad +e^{19/300}\pi^{1/4}\frac{K_p^{1/4}}{(2K_p+1)n}
\int_{\frac{1}{n}}^{\frac{\kappa}{\tilde R^2}}\frac{1}{y^{3/2}\sqrt{y-\frac{1}{n}}}\big[e^{\frac{1}{2}(p-1)\tilde R^2\big(y-\frac{1}{n}\big)}-1\big]dy\Big\}\Big\},
\end{align*}
where $M_{n,\kappa}:={\sqrt{1-\frac{\tilde R^2}{n\kappa}}}$ and $\sigma^2_{n,k}:=\frac{\sig k(n-k)}{n}.$ 

\end{lemma}
\begin{proof}

If $k\ne n/2$ then we first remark that \begin{align*}
    W_k=\Big(1-\frac{k}{n}\Big)S_k-\frac{k}{n}(S_n-S_k).
\end{align*}We note that $S_k$ and $S_n-S_k$ are two independent sums of i.i.d random variables. Hence according to the definition of the Wasserstein-p distance we know that for all $\epsilon>0$ there exists $G_1\sim \mathcal{N}(0,k\sigma^2)$ and $G_2\sim \mathcal{N}(0,(n-k)\sigma^2) $ such that $$\|S_k-G_1\|_p\le \mathcal{W}_p(S_k, \mathcal{N}(0,\sigma^2k))+\epsilon,~\|S_n-S_k-G_2\|_p\le \mathcal{W}_p(S_n-S_k, \mathcal{N}(0,\sigma^2(n-k)))+\epsilon.$$ Moreover we remark that $$\Big(1-\frac{k}{n}\Big)G_1-\frac{k}{n}G_2\sim \mathcal{N}(0,\frac{k(n-k)\sigma^2}{n}).$$ Hence we obtain that $$\mathcal{W}_p(W_k, \mathcal{N}(0,\sigma^2_{n,k}))\le (1-\frac{k}{n})\mathcal{W}_p(S_k, \mathcal{N}(0,\sigma^2k))+\frac{k}{n}\mathcal{W}_p(S_n-S_k, \mathcal{N}(0,\sigma^2(n-k))) +2\epsilon.$$ As $\epsilon>0$ is arbitrary we directly obtain that $$\mathcal{W}_p(W_k, \mathcal{N}(0,\sigma^2_{n,k}))\le \Big(1-\frac{k}{n}\Big)\mathcal{W}_p(S_k, \mathcal{N}(0,\sigma^2k))+\frac{k}{n}\mathcal{W}_p(S_n-S_k, \mathcal{N}(0,\sigma^2(n-k))).$$
According to Lemma 8 \cite{austern2023efficient} we obtain that $$\mathcal{W}_p(W_k, \mathcal{N}(0,\sigma^2_{n,k}))\le \max\Big(\tilde \omega_p^R(\sigma,k),\tilde \omega_p^R(\sigma,n-k)\Big).$$
Moreover using \Cref{lemma:MZineq_rio} we note that the following trivial bound also holds:
\begin{align*}
    \mathcal{W}_p(W_k, \mathcal{N}(0,\sigma_{n,k}^2))&\le \|W_k\|_p+\sigma_{n,k}\|Z\|_p
    \\&\le \sqrt{p-1}\Big((1-\frac{k}{n})R\sqrt{k}+\frac{k\sqrt{n-k}R}{n}+\sigma_{n,k}\Big)
    \\&\le \sqrt{p-1}\frac{\sqrt{n}}{2}\Big(R\sqrt{2}+\sigma\Big).
\end{align*}

We finally prove a tighter bound if $k=n/2$. In 
Fix any $p\geq 2$, define $\tilde W_k:= \frac{W_k\sqrt{n}}{\sig \sqrt{k(n-k)}}$, and let $Z$ represent an independent standard normal variable. 
We first remark that $$\mathcal{W}_p(W_k,\mathcal{N}(0,\sigma_{n,k}^2)\le\sigma_{n,k}\mathcal{W}_p(\tilde W_k,Z). $$ We first bound the right hand side. In this goal, we note that, for all $j\ge 1$, the moments of $X_i^j$ can be upper bounded as
\begin{align}\label{moment_eq_t}&\|X_i^j\|_p=\big(\mathbb{E}(|X_i|^{jp})\big)^{1/p}\le \big(\|X_i\|_{\infty}^{kp-2}\mathbb{E}(|X_i|^2)\big)^{1/p} \le\frac{R_s^j\sig^{2/p}}{R_s^{2/p}}.
\end{align}%

Consider a random index $I\sim \Unif(\{1,\dots,n\})$ and a sequence $(X_i')_{i\geq1}\eqdist (X_i)_{i\geq1}$ with \\ $(I,(X_i')_{i\geq1},(X_i)_{i\geq1})$ mutually independent.
Define an exchangeable copy of $\tilde W_k$,
\begin{align}
W'_{k}:= \tilde W_k+(X'_I-X_I)\sigma_{n,k}^{-1}\begin{cases}
    (1-\frac{k}{n})\quad \text{if}~I\le k\\-\frac{k}{n}\quad \text{otherwise}
\end{cases},
\end{align} 
and the exchangeable pair difference,
\begin{align*}
Y:=\tilde W_k-W'_j.
\end{align*} 
 Fix any $\kappa\ge \frac{\tilde R^2}{n}$ and define
\begin{align}
\mnkappa:= {\sqrt{1-\frac{\tilde R^2}{n\kappa}}}.
\end{align}
A slight modification of Theorem 3 of \cite{bonis2020stein} shows that
\begin{align*}\label{sneeze}\wass (\tilde W_k,Z)\le &\int^{-\frac{1}{2}\log(1-\frac{\tilde R^2}{n\kappa})}_0\|e^{-t}\tilde W_k-\frac{e^{-2t}}{\sqrt{1-e^{-2t}}}Z\|_pdt \\&+\int_{-\frac{1}{2}\log(1-\frac{R^2}{n\kappa})}^{\infty}e^{-t}\|n\mathbb{E}\big(Y|\tilde W_k\big)-
\tilde W_k\|_pdt\\&+\int_{-\frac{1}{2}\log(1-\frac{\tilde R^2}{n\kappa})}^{\infty}\frac{e^{-2t}\|H_1\|_p}{\sqrt{1-e^{-2t}}}\|\frac{n}{2}\mathbb{E}\big(Y^2|\tilde W_k\big)-1\|_pdt
\\&+\sum_{j\ge 3}\int_{-\frac{1}{2}\log(1-\frac{\tilde R^2}{n\kappa})}^{\infty}\frac{e^{-jt}\|H_{j-1}\|_p}{j!(\sqrt{1-e^{-2t}})^{j-1}}n\|\mathbb{E}\big(Y^j|\tilde W_k\big)\|_pdt
\\
:= &\ (a_0)+(a_1)+(a_2)+(a_3).
\end{align*}
   
We first bound $(a_0).$ 
To this end, fix $\epsilon>0$ and select $G\sim \mathcal{N}(0,1)$ independent from $Z$ such that $\|\tilde W_k-G\|_p\le \mathcal{W}_p(\tilde W_k,Z)+\epsilon.$ 
By the triangle inequality we have
\begin{align*}&
   \int^{-\frac{1}{2}\log(1-\frac{\tilde R^2}{n\kappa})}_0\|e^{-t}\tilde W_k-\frac{e^{-2t}}{\sqrt{1-e^{-2t}}}Z\|_pdt  \\&\le\int^{-\frac{1}{2}\log(1-\frac{\tilde R^2}{n\kappa})}_{0}e^{-t}\|\tilde W_k-G\|_p+\|e^{-t}G-\frac{e^{-2t}}{\sqrt{1-e^{-2t}}}Z\|_p dt
    \\&\overset{(a)}{=} (1-\mnkappa)\|\tilde W_k-G\|_p+\|Z\|_p\int^{-\frac{1}{2}\log(1-\frac{\tilde R^2}{n\kappa})}_{0}\frac{e^{-t}}{\sqrt{1-e^{-2t}}}dt  \\&\le (1-\mnkappa)\big\{\mathcal{W}_p(\tilde W_k,Z)+\epsilon\big\} +\|Z\|_p\big[\frac{\pi}{2}-\mathrm{sin}^{-1}(\mnkappa)\big],
\end{align*}
where (a) follows from the fact $e^{-t}G-\frac{e^{-2t}}{\sqrt{1-e^{-2t}}}Z\overset{d}{=} \frac{e^{-t}}{\sqrt{1-e^{-2t}}}Z$.
Since $\epsilon>0$ was arbitrary, we have \begin{align}&
 (a_0)
 \le (1-\mnkappa)\mathcal{W}_p(\tilde W_k,Z)+\|Z\|_p\big[\frac{\pi}{2}-\mathrm{sin}^{-1}(\mnkappa)\big].
\end{align} 
In addition, if $p=2$, by independence of $Z$ and $S_n$ we obtain that 
\begin{align}\label{eq:02}
   \int^{-\frac{1}{2}\log(1-\frac{\tilde R^2}{n\kappa})}_0\|e^{-t}\tilde W_k-\frac{e^{-2t}}{\sqrt{1-e^{-2t}}}Z\|_pdt &\le\int^{-\frac{1}{2}\log(1-\frac{\tilde R^2}{n\kappa})}_0\sqrt{e^{-2t}\|\tilde W_k\|_2^2+\frac{e^{-4t}}{{1-e^{-2t}}}\|Z\|_2^2} dt
    \\&\le\int^{-\frac{1}{2}\log(1-\frac{\tilde R^2}{n\kappa})}_0\frac{e^{-t}}{\sqrt{1-e^{-2t}}}dt
    \le
 \frac{\pi}{2}-\mathrm{sin}^{-1}(\mnkappa).
\end{align} 

Therefore, according to \cref{sneeze}, if $p>2$ we have 
\begin{align*}\mnkappa\wass (\tilde W_k,Z)\le & \|Z\|_p\Big(\frac{\pi}{2}-\mathrm{sin}^{-1}(\sqrt{1-\frac{\tilde R^2}{n\kappa}})\Big)+(a_1)+(a_2)+(a_3).\end{align*} and if $p=2$ we have\begin{align*}\wass (\tilde W_k,Z)\le & \|Z\|_p\Big(\frac{\pi}{2}-\mathrm{sin}^{-1}(\sqrt{1-\frac{\tilde R^2}{n\kappa}})\Big)+(a_1)+(a_2)+(a_3).\end{align*} 

 We will next bound $(a_1)$ of \cref{sneeze}. 
 Note that, since $I\sim\Unif(\{1,\dots,n\})$, 
\begin{align*}\label{merur2_t}
\mathbb{E}\big(Y|\tilde W_k\big)=\frac{\sigma_{n,k}^{-1}}{2n}\sum_{i\le k}\mathbb{E}(X_i-X'_i|\tilde W_k)-\frac{\sigma_{n,k}^{-1}}{n}\sum_{i\le k}\mathbb{E}(X_i-X'_i|\tilde W_k)=\frac{1}{n}\tilde W_k,
\end{align*} 
and hence $\|n\mathbb{E}\big(Y|\tilde W_k\big)-\tilde W_k\|_p=0$. 
Therefore $(a_1)=0$.

We now turn to bounding $(a_2)$ of \cref{sneeze}. 
By Jensen's inequality, we have
\begin{align*}\|\frac{n}{2}\mathbb{E}\big(Y^2|\tilde W_k\big)-1\|_p
&\le \|\frac{1}{2\sigma_{n,k}^2}\sum_{i\le k}(1-\frac{k}{n})^2(\mathbb{E}(X_i^2)+X_i^2)+\frac{1}{2\sigma_{n,k}^2}\sum_{i>k}\frac{k^2}{n^2}(\mathbb{E}(X_i^2)+X_i^2)-1\|_p
\\&=\frac{1}{8} \|\frac{1}{\sigma_{n,k}^2}\sum_{i\le k}X_i^2+\frac{1}{\sigma_{n,k}^2}\sum_{i>k}\sum_{i\le n}X_i^2-1\|_p\\&=\frac{1}{2n} \|\sum_{i\le n}\sigma^{-2}X_i^2-1\|_p.
\end{align*} 
To better bound this we first obtain an upper bound for $\|\sigma^{-2}X_i^2-1\|_p$. In this goal, we remark that   $|\sigma^{-2}X_i^2-1|\overset{a.s.}{\le}\max(\tilde R_s^2-1,1\big)$. We will get a different upper bound depending on if $\tilde R_s^2-1\le 1$ or not. 
First suppose $2\ge\tilde R_s^2$.
Since $\mathbb{E}(X_i^2)=\sigma^2$, we know $\|\sigma^{-2}X_i^2-1\|_p$  is maximized when  $\frac{1}{R_s^2} X_i^2\sim \Ber(\frac{1}{\tilde R_s^2}).$
Hence
\begin{align}
\|\sigma^{-2}X_i^2-1\|_p
\le  
\Big[\frac{ \tilde R_s^2-1}{ \tilde R_s^2}\Big]^{1/p}\big( (\tilde R_s^2-1)^{p-1}+1\big)^{1/p}.
\end{align} 
Now suppose $\tilde R_s^2\geq 2$. We instead obtain that 
\begin{align}
\|\sigma^{-2}X_i^2-1\|_p
\le  
(\|X_i^2\sigma^{-2}-1\|_\infty^{p-2} \|X_i^2\sigma^{-2}-1\|_2^2)^{1/p}
\le
\big( \tilde R_s^2-1\big)^{1-1/p}.
\end{align} 
Therefore using the Marcinkiewicz-Zygmund inequality (\Cref{lemma:MZineq_rio}) we have if $\tilde R_s^2\ge 2$
 \begin{align*}\|\frac{n}{2}\mathbb{E}\big(Y^2|\tilde W_k\big)-1\|_p
\overset{}{\le}\frac{1}{2\sqrt{n}}\sqrt{p-1}(\tilde R_s^2-1)^{1-1/p}
\end{align*} and if $\tilde R_s^2\le 2$ we instead obtain that \begin{align*}&\|\frac{n}{2}\mathbb{E}\big(Y^2|\tilde W_k\big)-1\|_p
\\&\overset{}{\le}\frac{1}{2\sqrt{n}}\sqrt{p-1}\Big[\frac{ R_s^2-1}{ R_s^2}\Big]^{1/p}\big( (R_s^2-1)^{p-1}+1\big)^{1/p}
\end{align*}
Moreover using the Rosenthal inequality \Cref{lemma:rosenthal_indep_nonidentical} we find that 
\begin{align*}\|\frac{n}{2}\mathbb{E}\big(Y^2|\tilde W_k\big)-1\|_p
\overset{}{\le}\frac{1}{2\sqrt{n}}\Big[(\tilde R_s^2-1)^{1-1/p}A_{n,p}^*+\sqrt{\tilde R_s^2-1}A_p\Big].
\end{align*}
for  $ \tilde R_s^2\ge 2$ and
\begin{align*}&\|\frac{n}{2}\mathbb{E}\big(Y^2|\tilde W_k\big)-1\|_p
\\&\overset{}{\le}\frac{1}{2\sqrt{n}}\Big[\frac{ R_s^2-1}{ R_s^2}\Big]^{1/p}\Big[\big( (R_s^2-1)^{p-1}+1\big)^{1/p}A^*_{n,p}+\sqrt{R_s^2-1}A_p\Big]
\end{align*}
for $\tilde R_s^2\le 2$.
Alternatively, if $k=n/2$ by \cite[Eq.~(2.8)]{esseen1975bounds} for $p<4$ and  \cite[Thm.~2.6]{cox1983sharp} for $p\ge 4$, we have the symmetrized estimate
\begin{align*}
\|\sigma^{-2}\sum_{i\le n}X_i^2-1\|_p\le2^{-1/p} \|\sigma^{-2}\sum_{i\le n}X_i^2-(X'_i)^2\|_p.
\end{align*}
Since the random variables $\big(X_i^2-(X'_i)^2\big)_{i\geq 1}$ are symmetric, with
\begin{align*}
\mathbb{E}[(X_i^2-(X'_i)^2)^2]
\le \frac{2}{n^2}(R_s^2-\sigma^2)
\qtext{and} \mathbb{E}[(X_i^2-(X'_i)^2)^p]\le \frac{2}{n^{p}}(R_s^2-\sigma^2)R_s^{2(p-2)},
\end{align*}
an improvement on the  Marcinkiewicz-Zygmund inequality for symmetric random variables (\Cref{dino_veg_symm}) implies
\begin{align*}
\|\sigma^{-2}\sum_{i\le n}X_i^2-(X'_i)^2\|_p
   \le
\begin{cases}
{\sqrt{2}2^{1/p}}{\sqrt{n}} \Big(\frac{\Gamma(\frac{p+1}{2})}{\sqrt{\pi}}\Big)^{1/p}\tilde R_s^{2(1-2/p)}(\tilde R_s^2-1)^{1/p} & \stext{if} p\ge2\\
{\sqrt{2}}\Big(\frac{\Gamma(\frac{p+1}{2})}{\sqrt{\pi}}\Big)^{1/p}{\tilde R^2}\|\Bin(n,\frac{2(\tilde R_s^2-1)}{\tilde R_s^4})\|_p &\stext{if} p\ge4.
\end{cases}
\end{align*}
Hence we finally obtain that 
\begin{align*}\label{highsch}&\qquad\|\frac{n}{2}\mathbb{E}\big(Y^2|\tilde W_k\big)-1\|_p
\\&\overset{}{\le}D^k_{n,p}:=\frac{1}{2\sqrt{n}}\min
\begin{cases}
\sqrt{p-1}\max\Big((\tilde R_s^2-1)^{1-1/p},\Big[\frac{ (\tilde R_s^2-1)^{p}+\tilde R_s^2-1}{\tilde R_s^2}\Big]^{1/p}\Big)\\
\max\Big((\tilde R_s^2-1)^{1-1/p} ,\Big[\frac{ (\tilde R_s^2-1)^{p}\tilde R_s^2-1}{\tilde R_s^2}\Big]^{1/p}\Big)A^*_{n,p}
+\sqrt{\tilde R_s^2-1}A_p\\\sqrt{2}\Big(\frac{\Gamma(\frac{p+1}{2})}{\sqrt{\pi}}\Big)^{1/p}\tilde R_s^{2(1-2/p)}(\tilde R_s^2-1)^{1/p} \\\sqrt{2}2^{-1/p}\Big(\frac{\Gamma(\frac{p+1}{2})}{\sqrt{\pi}}\Big)^{1/p}\frac{\tilde R^2}{\sqrt{n}}\sqrt{\|\Bin(n,\frac{2(\tilde R_s^2-1)}{\tilde R_s^4})\|_{p/2}}\quad\stext{if} p\ge4.\end{cases} 
\end{align*}
This implies that 
{\begin{align*}
    &\int_{-\frac{1}{2}\log(1-\frac{R^2}{n\kappa})}^{\infty}\|\frac{n}{2}\mathbb{E}\big((Y)^2|\tilde W_k\big)-1\|_p\frac{e^{-2t}}{\sqrt{1-e^{-2t}}}dt
    \\&\le D_{n,p}\int_{-\frac{1}{2}\log(1-\frac{R^2}{n\kappa})}^{\infty}\frac{e^{-2t}}{\sqrt{1-e^{-2t}}}dt
     = D_{n,p}\int^{\sqrt{1-\frac{R^2}{n\kappa}}}_{0}\frac{t}{\sqrt{1-t^2}}dt
    = \mnkappa^2 D_{n,p}.
\end{align*}}
Therefore since $H_1(Z)=Z$ we obtain 
\begin{align*}
(a_2)&\le (b_{2,1}^{\kappa,p,\tilde R})
:=
\|Z\|_p D_{n,p} M_{n,\kappa}^2.
\end{align*}

Finally, we turn to bounding $(a_3)$ of \cref{sneeze}.  
Since
\begin{align*}
   n \mathbb{E}(Y^j|X_1,\dots, X_n)=\sigma_{n,k}^{-j}\sum_{i\le n}\mathbb{E}((X_i-X'_i)^j|X_1,\dots, X_n)
   \qtext{for all}
  j \geq 1,
\end{align*} 
by Jensen's inequality, we have for all $j\ge 1$

\begin{align*}
n\big\|\mathbb{E}(Y^j|\tilde W_k)\big\|_p&\le\sigma_{n,k}^{-j}\big\|\sum_{i\le k}(1-\frac{k}{n})^j(X_i-X'_i)^j+\sum_{k < i\le n}(\frac{k}{n})^j(X_i-X'_i)^j\big\|_p
\\&=\sigma_{n,k}^{-j}2^{-j}\big\|\sum_{i\le k}(X_i-X'_i)^j+\sum_{k< i\le n}(X_i-X'_i)^j\big\|_p.
\end{align*}
We will derive different bounds for odd and even $j$, so we begin by writing
\begin{align*}
    (a_3)&\le \sum_{\substack{j\ge 3\\j ~\rm{is~odd}}}\int_{-\frac{1}{2}\log(1-\frac{\tilde R^2}{n\kappa})}^{\infty}\frac{e^{-jt}\|H_{j-1}\|_p}{j!(\sqrt{1-e^{-2t}})^{j-1}}n\|\mathbb{E}\big(Y^j|\tilde W_k\big)\|_pdt
    \\&+ \sum_{\substack{j\ge 4\\j ~\rm{is~even}}}\int_{-\frac{1}{2}\log(1-\frac{\tilde R^2}{n\kappa})}^{\infty}\frac{e^{-jt}\|H_{j-1}\|_p}{j!(\sqrt{1-e^{-2t}})^{j-1}}n\|\mathbb{E}\big(Y^j|\tilde W_k\big)\|_pdt
    :=(a_{3,1})+(a_{3,2}).
\end{align*} 

Let $j\ge 3$ be an \emph{odd} integer. The random variables $((X_i-X'_i)^j)_{i\ge1}$ are symmetric and therefore have a mean of zero. 
Moreover we note that 
\begin{align*}&\|(X_i-X'_i)^j\|_p=\|(Y_i-Y'_i)^j\|_p
\\&\le\|\max(Y_i,Y_i')^{j-2/p}|Y_i-Y'_i|^{2/p}\|_p
\\&\le  R^{j-2/p}\||Y_i-Y'_i|^{2/p}\|_p
\le {2^{1/p} R^{j-2/p}\sigma^{2/p}}.
\end{align*}
Similarly, we also have \begin{align}\|(X_i-X'_i)^j\|_2\le{\sqrt{2} R^{j-1}}\sigma.\end{align}
Therefore, using \Cref{lemma:MZineq_rio,dino_veg_symm,lemma:rosenthal_indep_nonidentical}
we obtain that 
$\sigma^{-j}\big\|\sum_{i\le n}(X_i-X'_i)^j\big\|_p\le 
\frac{ \tilde R^{j-1}}{\sqrt{n}^{j-1}}C_{n,p}$ where we set \begin{align*}
   C_{n,p}:= 
\min\begin{cases}\tilde A_{n,p}\tilde R 2^{1/p}+\sqrt{2}A_{n,p}\\{\sqrt{2}}\Big(\frac{\Gamma(\frac{p+1}{2})}{\sqrt{\pi}}\Big)^{1/p}\begin{cases}2^{1/p}\tilde R^{1-2/p}\text{\quad if }p<4\\2^{1/p}\tilde R^{1-2/p} \wedge \frac{\tilde R}{\sqrt{n}}\sqrt{\|\Bin(n,\frac{2}{\tilde R^2})\|_{p/2}}\quad \text{otherwise}.\end{cases} \end{cases} 
\end{align*}
Hence we obtain that 
\begin{align*}(a_{3,1})&=\sum_{\substack{j\ge 3\\odd}}\int_{-\frac{1}{2}\log(1-\frac{\tilde R^2}{n\kappa})}^{\infty}\frac{e^{-tj}\|H_{j-1}\|_p}{j!\sqrt{{1-e^{-2t}}}^{j-1}}n\|\mathbb{E}\big(Y^j|S_n\big)\|_pdt
\\&\le\sum_{\substack{j\ge 3\\odd}}
 {C_{n,p}}
 \int^{\sqrt{1-\frac{\tilde R^2}{n\kappa}}}_0\frac{x^{j-1}\tilde R^{j-1}\|H_{j-1}\|_p}{j!\sqrt{1-x^2}^{j-1}\sqrt{n}^{j-1}}dx
 \\&\overset{(a)}{\le} \sum_{\substack{j\ge 1}}C_{n,p}
 \frac{\tilde R^{2j}\|H_{2j}\|_p}{n^j(2j+1)! }\int^{\sqrt{1-\frac{\tilde R^2}{n\kappa}}}_0\frac{x^{2j}}{(1-x^2)^j}dx
 \\&\overset{(b)}{=}\frac{1}{2}\sum_{\substack{j\ge 1}}
C_{n,p} \frac{\tilde R^{2j}\|H_{2j}\|_p}{n(2j+1)!}\int_{\frac{\tilde R^2}{\kappa}}^n\frac{1}{\sqrt{y}}\big(\frac{1}{y}-\frac{1}{n}\big)^{j-\frac{1}{2}}dy,
\end{align*}
where (a) is obtained by noting that all odd numbers $k$ can be written as $2m+1$ for an $m\in \mathbb{N}_+$ and (b) by the change of variables $y=n(1-x^2)$.
 To further upper bound the right-hand side, we  will invoke the Hermite polynomial moment bound (\Cref{lemma:Hermite_bound})  $\|H_{j-1}\|_p\le {\sqrt{p-1}^{j-1}}{\sqrt{j-1!}}$ and use two applications of Stirling's approximation \citep{robbins1955remark} to conclude that,  for all  $m\in \mathbb{N}\setminus \{0\}$,
\begin{align}\label{taylor_swift}
        \sqrt{(2m)!} &\geq \sqrt{\sqrt{2\pi(2m)} \cdot (2m/e)^{2m} \cdot \exp\left(\frac{1} {12(2m)+1}\right)} 
        \geq e^{-19/300} 2^m m! / (\pi m)^{1/4}.
    \end{align}
These estimates imply that, for all $K_p\in \mathbb{N}_+$,
\begin{align*}
  (a_{3,1})&\le \frac{C_{n,p}}{2}\Big\{\sum_{\substack{1\le j\le K_p-1}}
 \frac{\tilde R^{2j}\|H_{2j}\|_p}{n(2j+1)!}\int_{\frac{\tilde R^2}{\kappa}}^n\frac{1}{\sqrt{y}}\big(\frac{1}{y}-\frac{1}{n}\big)^{j-\frac{1}{2}}dy
 \\&\quad +e^{19/300}\pi^{1/4}\sum_{\substack{j\ge K_p}}
 \frac{2^{-j}\tilde R^{2j}(p-1)^jk^{1/4}}{(2k+1) n{j!}}\int_{\frac{\tilde R^2}{\kappa}}^n\frac{1}{\sqrt{y}}\big(\frac{1}{y}-\frac{1}{n}\big)^{j-\frac{1}{2}}dy\Big\}
 \\&\overset{(a)}{\le}\frac{C_{n,p}}{2} \Big\{\sum_{\substack{1\le j\le K_p-1}}
 \frac{\tilde R^{2j}\|H_{2j}\|_p}{n(2j+1)!}\int_{\frac{\tilde R^2}{\kappa}}^n\frac{1}{\sqrt{y}}\big(\frac{1}{y}-\frac{1}{n}\big)^{j-\frac{1}{2}}dy
 \\&\quad +e^{19/300}\pi^{1/4}\frac{K_p^{1/4}}{(2K_p+1)}\sum_{\substack{j\ge K_p}}
 \frac{2^{-j}\tilde R^{2j}(p-1)^j}{ n{j!}}\int_{\frac{\tilde R^2}{\kappa}}^n\frac{1}{\sqrt{y}}\big(\frac{1}{y}-\frac{1}{n}\big)^{j-\frac{1}{2}}dy\Big\}
  \\&\overset{(b)}{\le}\frac{C_{n,p}}{2} \Big\{\sum_{\substack{1\le j\le K_p-1}}
 \frac{\tilde R^{2j}}{n}\Big(\frac{\|H_{2j}\|_p}{(2j+1)!}-\frac{K_p^{1/4}2^{-j}(p-1)^je^{19/300}\pi^{1/4}}{(2K_p+1)j!}\Big)\int_{\frac{\tilde R^2}{\kappa}}^n\frac{1}{\sqrt{y}}\big(\frac{1}{y}-\frac{1}{n}\big)^{j-\frac{1}{2}}dy
 \\&\quad +e^{19/300}\pi^{1/4}\frac{K_p^{1/4}}{(2K_p+1)n}
\int_{\frac{\tilde R^2}{\kappa}}^n\frac{1}{\sqrt{1-\frac{y}{n}}}\big[e^{\frac{1}{2}(p-1)\tilde R^2\big(\frac{1}{y}-\frac{1}{n}\big)}-1\big]dy\Big\},
\end{align*}
where (a) follows from the fact that $x\rightarrow\frac{x^{1/4}}{(2x+1)}$ is decreasing and (b) from the fact that \begin{align*}e^{\frac{1}{2}(p-1)\tilde R^2(\frac{1}{y}-\frac{1}{n})}-1=\sum_{j=1}^{K_p-1}\frac{(\tilde R^2(p-1))^j2^{-j} }{j!}\big(\frac{1}{y}-\frac{1}{n}\big)^j+\sum_{j=K_p}^{\infty}\frac{(\tilde R^2(p-1))^j2^{-j} }{j!}\big(\frac{1}{y}-\frac{1}{n}\big)^j.\end{align*}

Next suppose $j\ge 4$  is even.  Then $\big((X_i-X'_i)^j\big)_{i\ge 1}$ are almost surely nonnegative. Moreover, 
\begin{align*}\sigma^{-j}\mathbb{E}((X_i-X'_i)^j)\le 2\frac{\tilde R^{j-2}}{\sqrt{n}^j},\qquad \sigma^{-jp}\mathbb{E}((X_i-X'_i)^{jp})\le 2\frac{\tilde R^{jp-2}}{\sqrt{n}^{jp}}.\end{align*} 
Therefore we can invoke a moment inequality for nonnegative random variables (\Cref{lemma:moment-inequality}) to conclude
\begin{align*}n\|\mathbb{E}(Y^j|\tilde W_k)\|_p&\le \sigma_{n,k}^{-j}\|\sum_{i\le n}(X_i-X'_i)^j\|_p
\le  \frac{\tilde R^j}{\sqrt{n}^j}\Big\|\Bin\Big(n, \frac{2}{\tilde R^2}\Big)\Big\|_p.
\end{align*}Moreover by the triangle inequality and \Cref{lemma:MZineq_rio} the following upper also holds\begin{align*}
  n  \big\|\mathbb{E}(Y^j|\tilde W_k)\|_p&\le   \sigma_{n,k}^{-j} \big\|\sum_{i\le n}(X_i-X_i')^j\|_p
  \\&\le n\sigma_{n,k}^{-j}\mathbb{E}((X_1-X'_1)^j)+ \sigma_{n,k}^{-j}\big\|\sum_{i\le n}(X_i-X_i')^j-\mathbb{E}((X_1-X'_1)^j)\|_p\\&\le \frac{\tilde R^{j-2}}{\sqrt{n}^{j-2}}\Big[1+\frac{1}{\sqrt{n}}\tilde U_{n,p}\tilde R\Big].
\end{align*} If instead $k\ne n/2$ we have\begin{align*}&
  n  \big\|\mathbb{E}(Y^j|\tilde W_k)\|_p\\&\qquad+ n\sigma_{n,k}^{-j}\mathbb{E}((X_1-X'_1)^j)\\&\le \frac{\tilde R^{j-2}}{\sqrt{n}^{j}}\Big[k\big(\frac{n-k}{n}\big)^{j/2}+(n-k)\big(\frac{k}{n}\big)^{j/2}+\sqrt{p-1}\tilde R^{2-2/p}\sqrt{k\big(\frac{n-k}{n}\big)^{j}+(n-k)\big(\frac{k}{n}\big)^{j}}\Big].
\end{align*}  Hence we obtain \begin{align*}
    n\|\mathbb{E}(Y^j|\tilde W_k)\|_p&\le \sigma_{n,k}^{-j} \|\sum_{i\le n}(X_i-X'_i)^j\|_p
\le  \frac{B_{p,n}\tilde R^{j-2}}{\sqrt{n}^{j-2}}
\end{align*}
for $$B_{p,n}:= \min\Big(\frac{\tilde R^2}{n}\|\Bin(n,\frac{2}{\tilde R^2})\|_p, 1+\frac{1}{\sqrt{n}}\tilde U_{n,p}\tilde R\Big).$$
This gives us the upper estimate  
{\begin{align*}&
   \int_{-\frac{1}{2}\log(1-\frac{\tilde R^2}{n\kappa})}^{\infty}\frac{e^{-tj}}{\sqrt{1-e^{-2t}}^{j-1}}n\|\mathbb{E}\big(Y^j|S_n\big)\|_pdt
   \\&\le B_{p,n}  \int_{-\frac{1}{2}\log(1-\frac{\tilde R^{2}}{n\kappa})}^{\infty}\frac{e^{-tj}\tilde R^{j-2}}{\sqrt{1-e^{-2t}}^{j-1}\sqrt{n}^{j-2}}dt.
\end{align*} 
To bound $(a_{3,2})$,  it remains to bound
     \begin{align*}&
  \sum_{\substack{j\ge 4\\even}}\frac{\|H_{j-1}\|_p}{j!} \int_{-\frac{1}{2}\log(1-\frac{\tilde R^{2}}{n\kappa})}^{\infty}\frac{e^{-tj}\tilde R^{j-2}}{\sqrt{1-e^{-2t}}^{j-1}\sqrt{n}^{j-2}}dt
  \\&\overset{(a)}{\le} 
  \sum_{\substack{j\ge 3\\odd}}\frac{\tilde R^{j-1}\|H_{j}\|_p}{(j+1)!\sqrt{n}^{j-1}}\int^{\sqrt{1-\frac{\tilde R^2}{n\kappa}}}_0\frac{x^{j}}{\sqrt{1-x^2}^{j}}dx
    \\&\overset{(b)}{\le} 
   \sum_{\substack{j\ge 1}}\frac{\tilde R^{2j}\|H_{2j+1}\|_p}{(2j+2)!}\int^{\sqrt{1-\frac{\tilde R^2}{n\kappa}}}_0\frac{x^{2j+1}}{\sqrt{1-x^2}^{2j+1}n^j}dx
      \\&\overset{(c)}{\le} 
  \frac{1}{2}\sum_{\substack{j\ge 1}}\frac{\tilde R^{2j}\|H_{2j+1}\|_p}{(2j+2)!\sqrt{n}}\int_{{\frac{\tilde R^2}{\kappa}}}^n\frac{1}{\sqrt{y}}\Big(\frac{1}{y}-\frac{1}{n}\Big)^jdy,
\end{align*}
where (a) and (c) are obtained by a change of variable, and (b) is a consequence of the fact that every odd number can be written as $2m+1$ for an $m\ge 1$. 
To upper bound this quantity we will again employ a Hermite polynomial moment bound (\Cref{lemma:Hermite_bound}),  $\|H_{j-1}\|_p\le \sqrt{p-1}^{j-1}\sqrt{j-1!}$, and use Stirling's approximation to deduce 
that, for all $m\in \mathbb{N}\setminus\{0\}$, we have \begin{align*}\begin{split} 
    \sqrt{(2m+1)!}= \sqrt{2m+1}\sqrt{2m!}\ge  \sqrt{2m+1}e^{-19/300} 2^m m! / (\pi m)^{1/4}. 
\end{split}\end{align*}
Hence for any $K_p\in \mathbb{N}_+$ we obtain that      \begin{align*}(a_{3,2})=&
 \sum_{\substack{j\ge 4\\even}} \frac{\|H_{j-1}\|_p}{j!}\int_{-\frac{1}{2}\log(1-\frac{\tilde R^2}{n\kappa})}^{\infty}\frac{e^{-tj}}{\sqrt{1-e^{-2t}}^{j-1}}n\|\mathbb{E}\big(Y^j|S_n\big)\|_pdt
\\&\overset{}{\le} 
   \frac{B_{p,n}}{2}\sum_{\substack{K_p>j\ge 1}}\frac{\tilde R^{2j}\|H_{2j+1}\|_p}{(2j+2)!\sqrt{n}}\int_{{\frac{\tilde R^2}{\kappa}}}^n\frac{1}{\sqrt{y}}\Big(\frac{1}{y}-\frac{1}{n}\Big)^jdy
   \\&\quad+ \frac{B_{p,n}}{4}\sum_{\substack{K_p\le j}}\frac{\tilde R^{2j}(p-1)^{k+1/2}}{(k+1)\sqrt{2k+1!}\sqrt{n}}\int_{{\frac{\tilde R^2}{\kappa}}}^n\frac{1}{\sqrt{y}}\Big(\frac{1}{y}-\frac{1}{n}\Big)^jdy
      \\&\overset{}{\le} 
  \frac{B_{p,n}}{2} \sum_{\substack{K_p>j\ge 1}}\frac{\tilde R^{2j}}{\sqrt{n}}\Big(\frac{\|H_{2j+1}\|_p}{(2j+2)!}-
   \frac{2^{-j}e^{19/300}\pi^{1/4}K_p^{1/4}\sqrt{p-1}^{2k+1}}{2(K_p+1)\sqrt{2K_p+1}j!}\Big)\int_{{\frac{\tilde R^2}{\kappa}}}^n\frac{1}{\sqrt{y}}\Big(\frac{1}{y}-\frac{1}{n}\Big)^jdy
   \\&\quad+\frac{B_{p,n}}{4}e^{19/300}\pi^{1/4} \frac{K_p^{1/4}\sqrt{p-1}}{(K_p+1)\sqrt{2K_p+1}}\frac{1}{\sqrt{n}}\int_{{\frac{\tilde R^2}{\kappa}}}^n\frac{1}{\sqrt{y}}\big(e^{\frac{1}{2}(p-1)\tilde R^2(\frac{1}{y}-\frac{1}{n})}-1\big)dy.
\end{align*}}

Therefore, to conclude,  we obtain that 
\begin{align}\label{fr_al_3}&(a_3)\le (b_{3,1}^{\kappa,p,\tilde R})\\&:=      \frac{B_{p,n}}{2} \Big\{\sum_{\substack{K_p>j\ge 1}}\frac{R^{2k}}{\sqrt{n}}\Big(\frac{\|H_{2j+1}\|_p}{(2j+2)!}-
   \frac{2^{-j}e^{19/300}\pi^{1/4}K_p^{1/4}\sqrt{p-1}^{2k+1}}{2(K_p+1)\sqrt{2K_p+1}j!}\Big)\int_{{\frac{\tilde R^2}{\kappa}}}^n\frac{1}{\sqrt{y}}\Big(\frac{1}{y}-\frac{1}{n}\Big)^jdy
   \\&\quad+\frac{1}{4}e^{19/300}\pi^{1/4} \frac{K_p^{1/4}\sqrt{p-1}}{(K_p+1)\sqrt{2K_p+1}}\frac{1}{\sqrt{n}}\int_{{\frac{\tilde R^2}{\kappa}}}^n\frac{1}{\sqrt{y}}\big(e^{\frac{1}{2}(p-1)\tilde R^2(\frac{1}{y}-\frac{1}{n})}-1\big)dy\Big\}
\\&+\frac{C_{n,p}}{2}\Big\{\sum_{\substack{1\le j\le K_p-1}}
 \frac{\tilde R^{2j}}{n}\Big(\frac{\|H_{2j}\|_p}{(2j+1)!}-\frac{K_p^{1/4}2^{-j}(p-1)^je^{19/300}\pi^{1/4}}{(2K_p+1)j!}\Big)\int_{\frac{\tilde R^2}{\kappa}}^n\frac{1}{\sqrt{y}}\big(\frac{1}{y}-\frac{1}{n}\big)^{j-\frac{1}{2}}dy
 \\&\quad +e^{19/300}\pi^{1/4}\frac{K_p^{1/4}}{(2K_p+1)n}
\int_{\frac{1}{n}}^{\frac{\kappa}{\tilde R^2}}\frac{1}{y^{3/2}\sqrt{y-\frac{1}{n}}}\big[e^{\frac{1}{2}(p-1)\tilde R^2\big(y-\frac{1}{n}\big)}-1\big]dy\Big\}.
\end{align}

Finally, using \Cref{lemma:MZineq_rio} we once again note that the following trivial upper-bound also holds
\begin{align*}
    \mathcal{W}_p(W_k,\mathcal{N}(0,\sigma_{n,k}^2))&\le \frac{\sqrt{p-1}\sqrt{n}}{2}(\sigma+R).
\end{align*}
\end{proof}
\begin{lemma}
    \label{nelson2s} For all $ p\ge 2$ and for all $n\in \mathbb{N}$ we have  
    $$\mathcal{S}_p^{\sigma,R}(n,k)\le { \KRsig p}$$
    for
    \begin{align}\label{KRsig}\notag
K_{R,\sigma}\defeq&\frac{1}{2}(\sigma+\sqrt{2}R))\vee \Big\{\textstyle\max(\frac{R}{\sig}, \frac{3\sigma}{4}\Big(\sqrt{\widetilde R_{\sigma}^2-1}\frac{\sqrt{e}(2e)^{1/2}}{\sqrt{2}}+\frac{\pi^{1/4}e^{19/300}}{4\sqrt{3}}\big[e^{\frac{\tilde R^2}{2} }-1\big]\Big)\\ \tag{$K_{R,\sig}$}&\textstyle+\sigma\sqrt{2}\Big(1+\log(4)\Big)\Big\{\frac{\tilde R^{1-2/p}\sqrt{8\pi}^{1/2}}{3\sqrt{e}}e^{19/300}\pi^{1/4}[e^{\frac{\tilde R^2}{2}}-1]\Big\}
\\ \notag&\textstyle+\sqrt{4}{\sigma}\frac{(\max(\widetilde R_{\sigma}^2-1,1))^{1/2}}{\sqrt{2}}
+{\sigma}\tilde R {\sqrt{e}(2e)^{1/2}}\frac{\pi^{1/4}e^{19/300}}{4\sqrt{3}}\big[e^{\frac{\tilde R^2}{2}}-1\big]
 \\ \notag&\textstyle+\frac{4\sigma\tilde R^{2-2/p}}{\sqrt{2n}}\frac{\pi^{1/4}e^{19/300}}{4\sqrt{3}}\big[e^{\frac{\tilde R^2}{2}}-1\big]+1)\Big\}.
 \end{align} 
 Moreover, for any $0<\sigma_1<\sigma_2\le \frac{R}{2}$ we have $$\max_{\sigma\in [\sigma_1,\sigma_2]}|K_{R,\sigma}|<\infty.$$
\end{lemma}
\begin{proof}
If $n+1\ge p$ then this was already proven in Lemma 9 of \cite{austern2023efficient}. For $p>n$, this is a direct consequence of \Cref{nelson2}.
\end{proof}

\section{Alternative bounds}\label{section:compare}
\subsection[Bounds in Bhattacharjee (2016)]{Bounds in \cite{Bhattacharjee16}}
In this section, under the additional assumption that the random variables $(Y_i)$ have finite support, we derive the computable thresholds obtained in \cite{Bhattacharjee16}. As the bounds in \cite{Bhattacharjee16} depend on constants that are not given in closed form, instead those constants are the solution of inequalities. We first have to derive those constants. As a remainder we denote $X_i=Y_i-\mathbb{E}(Y_i).$

In this subsection, we will assume that the random variables $(X_i)$ take values in a finite alphabet $\mathcal{A}$ that is not allowed to contain $0$, have variance $\sigma^2=1$ and satisfy $\mathbb{E}(X_1^3)=0$. We introduce the following notation.
\begin{itemize}
    \item 
$
\mathcal{D}=\{b-a: a, b \in \mathcal{A}\} \quad \text { and } \quad \mathcal{D}^{+}=\mathcal{D} \cap[0, \infty),
$
\item Define $q = |\mathcal{D}^{+}|+1$, $D = \frac{1}{2} \sum_{d \in \mathcal{D}^{+}} d,$ and $Q:=\sum_{d \in \mathcal{D}^{+}}d^2$. 
\item $W_{k,d}:=\frac{1}{n}\sum_{\substack{i\le k\\j>n}}(X_i-X_j)\mathbb{I}(|X_i-X_j|=d)$.
\end{itemize}
\begin{theorem}[Theorem 2.2 of \cite{Bhattacharjee16}]\label{Milktea}
     For $n \geq 1$, let $(Y_i)$ be random variables that satisfy \assumptionRS and that are in addition assumed to take values in $\mathcal{A} \subset \mathbb{R}$. Let $\gamma^2=n^{-1} \sum_{i=1}^n X_i^2$. Assume that $\sigma=1$.
 Then for all $\nu>0$ there exist positive constants $c_1, c_2$ and $\theta_2$ depending only on $\mathcal{A}$ and $\nu$ such that for any integer $n \geq 1$, an integer $k$ such that $|2 k-n| \leq 1$, and any $\eta \geq \nu$, it is possible to construct a version of $W_k$ and a Gaussian random variable $\tilde Z_k$ with mean 0 and variance $k(n-k) / n$ on the same probability space such that for all $\theta \leq \theta_2$,
\begin{equation}\label{kitkat}
\mathbb{E} \exp \left(\theta\left|W_k-\eta\tilde Z_k\right|\big|\mathcal{U}(X_{1:n})\right) \leq \exp \left(3+\frac{c_1 \theta^2 S_n^2}{n}+c_2 \theta^2 n\left(\gamma^2-\eta^2\right)^2\right).
\end{equation}
\end{theorem}

\begin{corollary}\label{mogetee}
Suppose the assumptions of \Cref{Milktea} hold. Then the bound in \cref{kitkat} holds for all $\nu > 0$ with the following explicit constants:
$$
c_1 = \frac{3C}{2} q, \quad c_2 = 2C,
$$
where
$$
C = \frac{1}{\nu^2} \left( 2\left(\frac{R_s^2}{2} + \frac{1}{8}Q \right)^2 + \frac{1}{4} + 3\left( R_s^2 + \frac{1}{4}Q \right) \right),
$$
and where $\theta_2 := \theta_3 \wedge \theta_4$, with $\theta_4 > 0$ satisfying
$$
2 \exp \left( D \theta_4 + 2C \theta_4^2 \right) \le e,
$$
and $\theta_3$ defined via
$$
2Cq \theta_3^2 = \alpha_1, \quad \text{with $\alpha_1$ satisfying} \quad \frac{1}{1 - 4 \alpha_1 R_s^2} \le \frac{9}{8}.
$$
\end{corollary}
\begin{proof}
    The inequality at the end of the proof of Theorem 2.2 (page 16) in \cite{Bhattacharjee16} states that
 the constants $c_1$ and $c_2$ in \cref{kitkat} can be chosen to be
$$
c_1 = \frac{3C}{2}q, \quad c_2 = 2C.
$$
We now derive the expression for $C$. In the same proof, it is stated that for $k$ such that $|2k-n|\le 1$, the following inequality holds for some constant $C$ depending uniquely on $\mathcal{A}$ and $\nu$: 
$$
\begin{aligned}
\frac{\left(T-\sigma^2\right)^2}{\sigma^2} & \leq \frac{n}{k(n-k) \nu^2}\left(R_s^2 / 2+\left|\widetilde{\sigma}^2-\sigma^2\right|+R_s\left|S_k\right|+\sum_{d \in \mathcal{D}^{+}} \frac{d}{2}\left|W_{k, d}\right|+C_0\right)^2 \\
& \leq C\left(1+n\left(\gamma^2-\eta^2\right)^2+\frac{S_k^2}{k}+\sum_{d \in \mathcal{D}^{+}} \frac{W_{k, d}^2}{k}\right),
\end{aligned}
$$
where $C_0$ is defined in (2.26) of \cite{Bhattacharjee16} and can be taken as
$C_0:=\frac{1}{8} \sum_{d \in \mathcal{D}^{+}}d^2.$ 
Applying the Cauchy-Schwarz inequality, we obtain
\begin{align*}
    & \frac{n}{k(n-k) \nu^2}\left(R_s^2 / 2+
\left|\frac{k(n-k)}{n}\left(\gamma^2-\eta^2\right)\right|
+R_s\left|S_k\right|+\sum_{d \in \mathcal{D}^{+}} \frac{d}{2}\left|W_{k, d}\right|+C_0\right)^2\\
    \leq & \frac{n}{k(n-k) \nu^2}\left((R_s^2 / 2+ C_0)^2 + \left(\frac{k(n-k)}{n}\right)^2\frac{1}{n} + R_s^{2}k + \sum_{d \in \mathcal{D}^{+}} \frac{d^2}{4}k\right)\\
    & \qquad\qquad \times \left(1+n(\gamma^2-\eta^2)^2+\frac{1}{k}|S_k|^2+\sum_{d\in \mathcal{D}^{+}}\frac{W_{k,d}^2}{k}\right).
\end{align*}
Recall that we denote $Q:=\sum_{d \in \mathcal{D}^{+}}d^2$. Then we have
\begin{align*}
&\frac{1}{\nu^2}\left(\frac{n}{k(n-k)}(R_s^2 / 2+ \frac{1}{8}Q)^2 + \frac{k(n-k)}{n}\frac{1}{n} + \frac{n}{n-k} R_s^{2} + \frac{n}{n-k}\frac{1}{4}Q\right)\\
\le& \frac{1}{\nu^2}\left(\frac{n}{\lfloor\frac{n}{2}\rfloor(n-\lfloor\frac{n}{2}\rfloor)}(R_s^2 / 2+ \frac{1}{8}Q)^2 + \frac{\lfloor\frac{n}{2}\rfloor(n-\lfloor\frac{n}{2}\rfloor)}{n}\frac{1}{n} + \frac{n}{n-\lceil\frac{n}{2}\rceil} R_s^{2} + \frac{n}{n-\lceil\frac{n}{2}\rceil}\frac{1}{4}Q\right)\\
\le&\frac{1}{\nu^2}\left(2(R_s^2 / 2+ \frac{1}{8}Q)^2 + \frac{1}{4} + 3\left( R_s^{2} + \frac{1}{4}Q\right)\right).
\end{align*}
Therefore, we can take 
\begin{align*}
C = \frac{1}{\nu^2}\left(2(R_s^2 / 2+ \frac{1}{8}Q)^2 + \frac{1}{4} + 3\left( R_s^{2} + \frac{1}{4}Q\right)\right).
\end{align*}

The value of $\theta_2$ follows directly from the proof of Theorem 2.2 and Lemma 2.9 of \cite{Bhattacharjee16}.
\end{proof}

\begin{theorem}[Lemma 4.1 of \cite{Bhattacharjee16}]\label{thm:BhattacharjeeLemma4.1}
Assume that the random variables $(X_i)$ satisfy the conditions of \Cref{Milktea} and are in addition such that $\mathbb{E}(X_1^3)=0$. There exists a constant $A$ such that there exists a constant $\lambda>0$ such that for any positive integer $n$, it is possible to construct a version of the sequence $\left(S_k\right)_{0 \leq k \leq n}$ and Gaussian random variables $\left(Z_k\right)_{0 \leq k \leq n}$ with mean zero and $\operatorname{Cov}\left(Z_i, Z_j\right)=i \wedge j$ on the same probability space such that
$$
\mathbb{E} \exp \left(\lambda\left|S_n-Z_n\right|\right) \leq A
$$
and
\begin{equation}\label{eqn:BhattacharjeeLemma4.1bound}
\mathbb{E} \exp \left(\lambda \max _{0 \leq k \leq n}\left|S_k-Z_k\right|\right) \leq A \exp (A \log n).
\end{equation}
\end{theorem}

\begin{lemma}[Corollary to Lemma 4.1 of \cite{Bhattacharjee16}]\label{lemma:finte_alphabet_result}
Suppose that the assumptions of \Cref{thm:BhattacharjeeLemma4.1} hold. Then \cref{eqn:BhattacharjeeLemma4.1bound} holds with the following constants:
Let $C$, $K_1$, $K_2$, and $\lambda_0$ be as in Theorem 1.4, for $\nu = \min_{a \in \mathcal{A}} |a|$. 
Let $A \ge 8$ be such that
$$
8^{1/2}  \Big[\exp(C \log n) + 4 \exp(C_1 \log n)\Big]^{1/2} \le A \exp(A \log n), 
$$
where
$$
C_1 := 2 + 8(R_s + 1) \sqrt{2\pi n} \lambda.
$$
Let $\lambda$ be
$$
\lambda=\min \left\{\frac{\theta_1}{2}, \frac{\lambda_0}{4}, \frac{1}{4\sqrt{2}R\sqrt{K_1}}, \frac{1}{{2}R_s^2}, \frac{1}{R_s+1}\right\},
$$
where
$$\theta_1 = \frac{1}{2\sqrt{3}}\min\{\vartheta_{\ell\left(\epsilon^2-1\right)}, \vartheta_{\ell\left(\epsilon\right)}\},
$$
$$
\quad K_1=8 c_1, \quad K_2=18 c_2 \quad \text { and } \quad \lambda_0=\sqrt{\frac{\alpha_1}{32 c_1}} \wedge \frac{\theta_2}{2} \wedge \frac{\theta_5}{\sqrt{72 c_2}} .
$$
Here, $\theta_5$ is the unique positive solution to
$$
\frac{1}{\sqrt{1-R_s^4 \theta^2 / 2}}=\frac{4}{3}.
$$
Moreover, the following inequality holds:
\begin{equation}\label{eqn:bhattercharjee_cor_fixed_n_KMT}
\mathbb{P} \left( \max_{0 \le k \le n} |S_k - Z_k| \ge \lambda^{-1} \left( A \log n - \log \alpha + \log A \right) \right) \le \alpha.
\end{equation}
\end{lemma}
\begin{proof}
The choice of constants follows directly from the proof of Lemma 4.1 in \cite{Bhattacharjee16}, and by tracing the definitions of all the constants involved. \Cref{eqn:bhattercharjee_cor_fixed_n_KMT} follows by applying the Chernoff bound to \cref{eqn:BhattacharjeeLemma4.1bound}.
\end{proof}

\subsection[Bounds in Castelle \& Laurent-Bonvalot (1998)]{Bounds in \cite{castelle1998strong}}\label{section:app_compare_clb_bg}

We now derive the computable bounds from the uniform empirical process literature \cite{ castelle1998strong}.

\begin{lemma}\label{lemma:emp_process_res}
    Suppose $(Y_i)_{i\ge1}$ satisfy Assumption $(R, \sigma)$. There exists $(G_i)_{i\ge1} \overset{\text{i.i.d.}}{\sim} \mathcal{N}(0, \sigma^2)$ such that the following holds for any $\alpha\in(0,1)$:
\begin{equation}
\mathbb P\Big(
\sup_{1\le i\le n}\big|\sum_{j=1}^i Y_j-\sum_{j=1}^i G_j\big|
\ge R\big(60\log n + 30\log(0.67/\alpha)\big)\log n
\Big)\le \alpha.
\end{equation}
\end{lemma}
\begin{proof}
Let $(U_i)_{i\ge1}$ be i.i.d.\ $\mathrm{Unif}(0,1)$ and define the coupling $Y_i = F^{-1}(U_i)$, where $F$ is the CDF of $Y_1$.
For each $k\ge1$, let the empirical CDF be $\hat{F}_k(u) = k^{-1}\sum_{i=1}^k \mathbb 1\{U_i \le u\}$, and define 
$\alpha_k(u) = \frac{1}{\sqrt{k}}\sum_{i=1}^k (\mathbb 1\{U_i \le u\} - u) = \sqrt{k}(\hat{F}_k(u) - u)$ for $u \in [0,1]$. By Theorem~2.2 (ii) of \cite{castelle1998strong}, there exists a Kiefer process $K$ such that, for all $x>0$,
\begin{equation}\label{eq:CLB}
     \mathbb P\!\Big(\sup_{1\le m\le n}\sup_{0\le t\le 1}\big| \sqrt{m}\,\alpha_m(t)-K(m,t)\big| \,\ge\, \big(x+60\log n\big)\,\log n\Big)\ \le\ 0.67\,\exp(-x/30). 
\end{equation}

Note that since $Y_i\in[0,R]$, we have
$Y_i - \mathbb E [Y_i] = \int_0^R (\mathbb 1\{Y_i > t\} - \mathbb P(Y_i > t))\,dt
= \int_0^R (F(t) - \mathbb 1\{U_i \le F(t)\})\,dt,$
we obtain
$$\sum_{i=1}^k (Y_i - \mathbb E [Y_i])
 = \int_0^R (kF(t) - \sum_{i=1}^k \mathbb 1\{U_i \le F(t)\})\,dt
 = -\int_0^R \sqrt{k}\alpha_k(F(t))\,dt.$$
Define $W(m):=-\int_0^R K(m,F(t))\,dt$.  
Because $K$ is the Kiefer process, $(W(m))_{m\ge1}$ is a centered Gaussian process with independent increments, so we can write $W(m)=\sum_{i=1}^m G_i$ for i.i.d.\ Gaussian $G_i$ with $\operatorname{Var}(G_i)=\operatorname{Var}(Y_1)$. Then we have that, for every $m\le n$,  
$$\sum_{i=1}^m(Y_i-\mathbb E[Y_i])-W(m)
=-\int_0^R(\sqrt{m}\,\alpha_m(F(t))-K(m,F(t)))\,dt.$$  
Therefore,
$$\sup_{1\le m\le n}\big|\sum_{i=1}^m(Y_i-\mathbb E[Y_i])-W(m)\big|
\le R\,\sup_{1\le m\le n}\sup_{0\le u\le1}\big|\sqrt{m}\,\alpha_m(u)-K(m,u)\big|.$$ 
Combining this with \eqref{eq:CLB} gives for any $\alpha\in(0,1)$, we obtain
\begin{equation}
\mathbb P\Big(
\sup_{1\le i\le n}\big|\sum_{j=1}^i Y_j-\sum_{j=1}^i G_j\big|
\ge R\big(60\log n + 30\log(0.67/\alpha)\big)\log n
\Big)\le \alpha.
\end{equation}
\end{proof}


\section{Additional lemmas}\label{additional}
\begin{lemma}[Rosenthal's inequality with explicit constants]\label{dino_veg2} \quad Let $(\tilde X_i)_{i\geq 1}$ be a sequence of centered \iid observations. If $\|\tilde X_1\|_p<\infty$ for some $p\ge 2$, then 
 \begin{talign}\|\frac{1}{\sqrt{n}}\sum_{i\le n}\tilde X_i\|_p\le (\frac{p}{2}+1)n^{1/p-1/2}\|\tilde X_1\|_p+2^{1/p}\sqrt{p/2+1}e^{\frac{1}{2}+\frac{1}{p}}\|\tilde X_1\|_2.\end{talign}
\end{lemma}
\begin{proof}
    According to \cite[Thm.~2]{nagaev1978some} we have
    \begin{talign}
        \| \frac{1}{\sqrt{n}}\sum_{i\le n} \tilde X_i\|_p^p 
\le \inf_{c > \frac{p}{2}}  c^p n^{1-\frac{p}{2}} \| \tilde X_1 ||_p^p + p c^{p/2} e^c B(\frac{p}{2}, c-\frac{p}{2})\| \tilde X_1 \|_2^p,
    \end{talign} where $B(\cdot,\cdot)$ is the Beta function.
    The choice $c=\frac{p}{2}+1$ yields
    \begin{talign}
        \| \frac{1}{\sqrt{n}}\sum_{i\le n} \tilde X_i\|_p^p 
&\le \big(\frac{p}{2}+1)^p n^{1-\frac{p}{2}} \| \tilde X_1 ||_p^p + p (\frac{p}{2}+1)^{p/2} e^{\frac{p}{2}+1} B(\frac{p}{2}, 1)\| \tilde X_1 \|_2^p
\\&= (\frac{p}{2}+1)^p n^{1-\frac{p}{2}} || \tilde X_1 ||_p^p + 2 (\frac{p}{2}+1)^{\frac{p}{2}} e^{\frac{p}{2}+1}  || \tilde X_1 ||_2^p .
    \end{talign}
    The subadditivity of the $p$-th root now implies the result.
\end{proof}
\begin{lemma}[Generalization of Rosenthal inequality with explicit constants]\label{lemma:rosenthal_indep_nonidentical}
    The following inequality holds for all $p>2$:
\begin{align*}& \Big\|  \frac{1}{2n}\Big(k\sum_{k<i\le n}(X_i^2-1)+(n-k)\sum_{1\le i\le k}(X_{i}^2-1)\Big)\Big\|_p
    \\&\le \frac{\sqrt{k(n-k)}}{2\sqrt{n}}\min\begin{cases}
        \sqrt{p-1}(R^2_s-1)^{1-1/p}
        \\A_p\sqrt{R^2_s-1}+A^*_{n,p}(R^2_s-1)^{1-1/p}.
    \end{cases}
    \end{align*}
where recall that $A_p:=2^{1 / p} \sqrt{p / 2+1} e^{\frac{1}{2}+\frac{1}{p}}$ and $A_{n,p}^{*}:= \left(\frac{p}{2}+1\right) n^{1 / p-1 / 2}$.
\end{lemma}
\begin{proof}
    Adaptation of proof from \cite{nagaev1978some}.
\end{proof}

\begin{lemma}[Theorem 2.1 from \cite{rio2009moment}: Marcinkiewicz-Zygmund type inequality for martingales] \label{lemma:MZineq_rio}
Let $p>2$ and $\left(S_n\right)_{n \geq 0}$ be a sequence of random variables in $\mathbb{L}^p$. Set $X_k=S_k-S_{k-1}$. Assume that $\mathbb{E}\left(X_k \mid S_{k-1}\right)=0$ a.s.for any positive $k$. Then
$$
\left\|S_n\right\|_p^2 \leq\left\|S_0\right\|_p^2+(p-1)\left(\left\|X_1\right\|_p^2+\left\|X_2\right\|_p^2+\cdots+\left\|X_n\right\|_p^2\right).
$$
\end{lemma}

\begin{lemma}[Lemma 6 from \cite{bonis2020stein}]\label{lemma:Hermite_bound_Bonis} Let $Z$ be a normal random variable and let $\left(M_\alpha\right)_{\alpha \in \mathbb{N}^d} \in \mathbb{R}^d$. Then,
$$
\mathbb{E}\left[\left\|\sum M_\alpha H_\alpha(Z)\right\|^p\right]^{2 / p} \leq \sum \max (1, p-1)^{|\alpha|} \alpha!\left\|M_\alpha\right\|^2.
$$
\end{lemma}

\begin{lemma}[Bound on the norm of the Hermite polynomials]\label{lemma:Hermite_bound} Let $H_k(x) \triangleq e^{x^2 / 2} \frac{\partial^k}{\partial x} e^{-x^2 / 2}$. Shorthand $H_K \triangleq H_k(Z)$. Then the following holds for all $k, p \in \mathbb{N}$ :
$$
\left\|H_k\right\|_p \leq \sqrt{k!}{\sqrt{p-1}^k}.
$$

\end{lemma}

\begin{lemma}[Theorem 1 of \cite{osekowski2012note}]\label{lemma:ose_ineq}
     For $p\ge 4$, the following inequality holds:
 \begin{align}
     \Big\|\sum_{\substack{i\le k\\j>k}}(Y_{i}-Y_{j})^\ell\Big\|_p&\le C_p\Big\{\mathbb{E}\Big[\Big(\sum_{i\le n}\mathbb{E}[D_i^2|\mathcal{F}_{i-1}]\Big)^{p/2}\Big]^{1/p}+\mathbb{E}\Big[\sum_{i\le n}D_i^p\Big]^{1/p}\Big\}\\&\le C_p\Big\{\sqrt{\sum_{i\le n}\Big\|\mathbb{E}[D_i^2|\mathcal{F}_{i-1}]\Big\|_{p/2}}+\Big(\sum_{i\le n}\|D_i\|_p^p\Big)^{1/p}\Big\},
 \end{align} where $C_p:=2\sqrt{2}\Big(\frac{p}{4}+1\Big)^{1/p}(1+\frac{p}{\log(p/2)})$.
\end{lemma}

\begin{lemma}
\label{lemma:moment-inequality}
[Lemma 13 of \cite{austern2023efficient}] Let $\left(\tilde{X}_i\right)_{i \geq 1}$ be a sequence of i.i.d. random variables that are almost surely nonnegative. If, for some $p \geq 2$, $\mathbb{E}\left(\tilde{X}_1\right) \leq a$ and $\mathbb{E}\left(\tilde{X}_i^p\right) \leq b$ for $a, b>0$, then
$$
\left\|\sum_{i=1}^n \tilde{X}_i\right\|_p \leq\left(\frac{b}{a}\right)^{1 /(p-1)}\left\|\operatorname{Binomial}\left(n,\left(\frac{a^p}{b}\right)^{\frac{1}{p-1}}\right)\right\|_p.
$$
\end{lemma}




\begin{lemma}[Improved Marcinkiewicz-Zygmund inequality for symmetric random variables]\label{dino_veg_symm} 
Suppose $(\tilde X_i)_{i\ge1}$ are symmetric centered \iid observations admitting a finite $p$-th absolute moment for some $p\ge 2$.  Then
 \begin{align}\|\frac{1}{\sqrt{n}}\sum_{i\le n}\tilde X_i\|_p\le\sqrt{2}\Big(\frac{\Gamma(\frac{p+1}{2})}{\sqrt{\pi}}\Big)^{1/p}\sqrt{n}\|\tilde X_1\|_p.\end{align}
If $p\ge 4$, $\mathbb{E}(\tilde X_i^2)\le \tilde\sigma^2$, and  $\mathbb{E}(|\tilde X_i|^p)\le b_p$ , we also have 
       \begin{align}
    \|\sum_{i\le n} \tilde X_i\|_p\le     \sqrt{2}\Big(\frac{\Gamma(\frac{p+1}{2})}{\sqrt{\pi}}\Big)^{1/p}\Big[\frac{b_p}{\tilde\sigma^2}\Big]^{1/(p-2)}\sqrt{\|\Bin(n,(\frac{\tilde\sigma^p}{b_p})^{2/(p-2)})\|_{p/2}}.
    \end{align}
\end{lemma}
\begin{lemma}\label{nelson3}
    For all $n\in \mathbb{N}$ and all $p\ge 2$ the following holds
    \begin{align*}
        \mathcal{W}_p(S_n, \mathcal{N}(0,n\sigma^2))&\le 
 s^R_p(n, \sigma):=\sqrt{p-1}\sqrt{n}(R+\sigma)\wedge\tilde \omega_p^R(\sigma,n),
   \end{align*} where\begin{align*}&\sigma^{-1}\tilde \omega_p^R(\sigma,n)\\&:=\inf_{\substack{\kappa\le \tilde R^2/n\\K_p\ge 1}} \frac{\sqrt{n}}{\mnkappa\mathbb{I}(p\ne 2)+\mathbb{I}(p=2)}\Big\{\|Z\|_p\Big(\frac{\pi}{2}-\arcsin(\sqrt{1-\frac{\tilde R^2}{n\kappa}})\Big)+\|Z\|_p D_{n,p} M_{n,\kappa}^2\\&+  \frac{B_{p,n}}{2} \Big\{\sum_{\substack{K_p>j\ge 1}}\frac{R^{2k}}{\sqrt{n}}\Big(\frac{\|H_{2j+1}\|_p}{(2j+2)!}-
   \frac{2^{-j}e^{19/300}\pi^{1/4}K_p^{1/4}\sqrt{p-1}^{2k+1}}{2(K_p+1)\sqrt{2K_p+1}j!}\Big)\int_{{\frac{\tilde R^2}{\kappa}}}^n\frac{1}{\sqrt{y}}\Big(\frac{1}{y}-\frac{1}{n}\Big)^jdy
   \\&\quad+\frac{1}{4}e^{19/300}\pi^{1/4} \frac{K_p^{1/4}\sqrt{p-1}}{(K_p+1)\sqrt{2K_p+1}}\frac{1}{\sqrt{n}}\int_{{\frac{\tilde R^2}{\kappa}}}^n\frac{1}{\sqrt{y}}\big(e^{\frac{1}{2}(p-1)\tilde R^2(\frac{1}{y}-\frac{1}{n})}-1\big)dy\Big\}
\\&+\frac{C_{n,p}}{2}\Big\{\sum_{\substack{1\le j\le K_p-1}}
 \frac{\tilde R^{2j}}{n}\Big(\frac{\|H_{2j}\|_p}{(2j+1)!}-\frac{K_p^{1/4}2^{-j}(p-1)^je^{19/300}\pi^{1/4}}{(2K_p+1)j!}\Big)\int_{\frac{\tilde R^2}{\kappa}}^n\frac{1}{\sqrt{y}}\big(\frac{1}{y}-\frac{1}{n}\big)^{j-\frac{1}{2}}dy
 \\&\quad +e^{19/300}\pi^{1/4}\frac{K_p^{1/4}}{(2K_p+1)n}
\int_{\frac{1}{n}}^{\frac{\kappa}{\tilde R^2}}\frac{1}{y^{3/2}\sqrt{y-\frac{1}{n}}}\big[e^{\frac{1}{2}(p-1)\tilde R^2\big(y-\frac{1}{n}\big)}-1\big]dy\Big\}\Big\}.
\end{align*}
\end{lemma}
\begin{proof}
  Lemma 8 of \cite{austern2023efficient} establishes that $$\mathcal{W}_p(S_n,\mathcal{N}(0,n\sigma^2))\le \tilde \omega_p^R(\sigma,n).$$ Now by the triangular inequality we obtain that 
  \begin{align*}
      \mathcal{W}_p(S_n,\mathcal{N}(0,n\sigma^2))&\le \|S_n\|_p+\|\mathcal{N}(0,n\sigma^2)\|_p
      \\&\overset{(a)}{\le }\sqrt{p-1}\sqrt{n} (R+\sigma),
  \end{align*} where (a) is a consequence of \Cref{lemma:MZineq_rio} and \Cref{lemma:Hermite_bound}. \end{proof}
  \begin{lemma}
    \label{nelson3s} For all $ p\ge 2$ and for all $n\in \mathbb{N}$ we have  
    $$s^R_p(n, \sigma)\le {K^*_{R,\sigma} p}$$
    for
    \begin{align*}
K^*_{R,\sigma}\defeq&(\sigma+R))\vee \Big\{\textstyle\max(\frac{R}{\sig}, \frac{3\sigma}{4}\Big(\sqrt{\tilde R_{\sigma}^2-1}\frac{\sqrt{e}(2e)^{1/2}}{\sqrt{2}}+\frac{\pi^{1/4}e^{19/300}}{4\sqrt{3}}\big[e^{\frac{\tilde R^2}{2} }-1\big]\Big)\\ \tag{$K_{R,\sig}$}&\textstyle+\sigma\sqrt{2}\Big(1+\log(4)\Big)\Big\{\frac{\tilde R^{1-2/p}\sqrt{8\pi}^{1/2}}{3\sqrt{e}}e^{19/300}\pi^{1/4}[e^{\frac{\tilde R^2}{2}}-1]\Big\}
\\ \notag&\textstyle+\sqrt{4}{\sigma}\frac{(\max(\widetilde R_{\sigma}^2-1,1))^{1/2}}{\sqrt{2}}
+{\sigma}\tilde R {\sqrt{e}(2e)^{1/2}}\frac{\pi^{1/4}e^{19/300}}{4\sqrt{3}}\big[e^{\frac{\tilde R^2}{2}}-1\big]
 \\ \notag&\textstyle+\frac{4\sigma\tilde R^{2-2/p}}{\sqrt{2n}}\frac{\pi^{1/4}e^{19/300}}{4\sqrt{3}}\big[e^{\frac{\tilde R^2}{2}}-1\big]+1)\Big\}.
 \end{align*} 
 Moreover, for any $0<\sigma_1<\sigma_2\le \frac{R}{2}$, we have $$\max_{\sigma\in [\sigma_1,\sigma_2]}|K^\star_{R,\sigma}|<\infty.$$
\end{lemma}
\begin{proof}
The proof is a direct consequence of Lemma 9 \cite{austern2023efficient}.
\end{proof}
\begin{lemma}
    \label{increasing_in_R}For all $k\le n$, $\alpha>0$, $R>0$ and $\sigma>0$, let   $(\Delta_k(\alpha, R,\sigma))_k$ and $(\mathcal{D}_k(\alpha, R,\sigma))_k$ denote the outputs of \Cref{ts122} and \Cref{ts12}, respectively. Then, for all $k\le n$, the following functions are non-decreasing on $[2\sigma,\infty)$: $$R\mapsto\mathcal{D}_k(\alpha, R,\sigma) \textrm{ and } R\mapsto\Delta_k(\alpha, R,\sigma).$$
\end{lemma}\begin{proof}
    We remark that $\mathcal{D}_k(\alpha, R,\sigma)$ and $\Delta_k(\alpha, R,\sigma)$ are non-decreasing functions of $(\omega_p^R(n,\sigma))$ and $(s_p^R(n,\sigma))$, respectively. Moreover, by \Cref{theorem:pfWp,nelson3}, the functions $$R\mapsto\omega_p^R(n,\sigma) \textrm{ and }R\mapsto s_p^R(n,\sigma)$$ are increasing on $[2\sigma,\infty).$ The desired result directly follows. 
\end{proof}
\begin{lemma}\label{partir}
    Assume that the conditions of \Cref{empirical} hold and define $\tilde R_k:=R/\hat\sigma_k^L$ then        \begin{align*}&
              \mathbb{P}\Big(\exists k\le n~\textrm{s.t.}~|S_k-Z_k|\ge \hat\sigma_k^U~\mathcal{D}_k(\alpha,\tilde R_k,1)\textrm{ and } \sigma\in \bigcap_{k=1}^\infty[\hat\sigma_k^L,\hat\sigma_k^U]\Big)
              \le\alpha.
         \end{align*}Similarly we also have 
        \begin{align*}&
              \mathbb{P}\Big(\exists k\le n~\textrm{s.t.}~|W_k-Z_k|\ge \hat\sigma_k^U~\Delta_k(\alpha,\tilde R_k,1)\textrm{ and } \sigma\in \bigcap_{k=1}^\infty[\hat\sigma_k^L,\hat\sigma_k^U]\Big)
              \le\alpha.
         \end{align*}
\end{lemma}
\begin{proof}Define $\tilde W_k:=\sigma^{-1}W_k$ and $\tilde S_k:=\sigma^{-1}S_k.$ Using \Cref{thm:main2,thm:main3} we know that there exists centered Gaussian vectors $(\tilde Z^*_k),(Z^*_k)$ that have the same variance as respectively $(\tilde W_k)$ and $(\tilde S_k)$ such that 
    $$ \mathbb{P}\Big(\exists k\le n~\textrm{s.t.}~|\tilde S_k-Z^*_k|\ge\mathcal{D}_k(\alpha,\frac{R}{\sigma},1)\Big)\le \alpha$$ and  $$ \mathbb{P}\Big(\exists k\le n~\textrm{s.t.}~|\tilde W_k-\tilde Z^*_k|\ge\Delta_k(\alpha,\frac{R}{\sigma},1)\Big)\le \alpha.$$
    Define $Z_k:=\sigma Z_k^*$ and $\tilde Z_k:=\sigma \tilde Z_k^*$. We remark that $(Z_k)$ and $(\tilde Z_k)$ are centered Gaussian vectors that have the same variance that $(S_k)$ and $(W_k).$ Moreover we note that 
    \begin{align*}
        & \mathbb{P}\Big(\exists k\le n~\textrm{s.t.}~| S_k-Z_k|\ge\sigma \mathcal{D}_k(\alpha,\frac{R}{\sigma},1)\Big)
         \\=&\mathbb{P}\Big(\exists k\le n~\textrm{s.t.}~| \tilde S_k-Z^*_k|\ge \mathcal{D}_k(\alpha,\frac{R}{\sigma},1)\Big)\\
         \le &\alpha.\end{align*}Similarly we have that  $$ \mathbb{P}\Big(\exists k\le n~\textrm{s.t.}~|\tilde W_k-\tilde Z^*_k|\ge\Delta_k(\alpha,\frac{R}{\sigma},1)\Big)\le \alpha.$$According to \Cref{increasing_in_R} we know that $R\rightarrow \Delta_k(\alpha,R,1)$ is an increasing function of $R$ on $[2,\infty)$. Hence, the following holds
         \begin{align*}&
              \mathbb{P}\Big(\exists k\le n~\textrm{s.t.}~|S_k-Z_k|\ge \hat\sigma_k^U~\mathcal{D}_k(\alpha,\tilde R_k,1)\textrm{ and } \sigma\in \bigcap_{k=1}^\infty[\hat\sigma_k^L,\hat\sigma_k^U]\Big)
              \\&\le  \mathbb{P}\Big(\exists k\le n~\textrm{s.t.}~|S_k-Z_k|\ge \sigma~\mathcal{D}_k(\alpha,R/\sigma,1)\Big)
              \\&\le\alpha.
         \end{align*}Similarly we can prove that 
        \begin{align*}&
              \mathbb{P}\Big(\exists k\le n~\textrm{s.t.}~|W_k-Z_k|\ge \hat\sigma_k^U~\Delta_k(\alpha,\tilde R_k,1)\textrm{ and } \sigma\in \bigcap_{k=1}^\infty[\hat\sigma_k^L,\hat\sigma_k^U]\Big)
              \\&\le\alpha.
         \end{align*}
\end{proof}

\end{document}